\documentclass[aos,preprint]{imsart}

\usepackage[english]{babel}
\usepackage[T1]{fontenc}
\usepackage[latin1]{inputenc}
\usepackage{dsfont}
\usepackage{mathrsfs}
\usepackage{amsmath}
\usepackage{amssymb}
\usepackage{verbatim}
\usepackage{amsthm}
\usepackage{graphicx}
\usepackage{xcolor}
\usepackage{fancybox}
\usepackage[square,numbers]{natbib}
\usepackage{enumitem}
\usepackage{eurosym}

\newtheorem{theorem}{Theorem}[section]
\newtheorem{proposition}[theorem]{Proposition}
\newtheorem{lemma}[theorem]{Lemma}

\newtheorem{definition}{Definition}

\begin{document}

\begin{frontmatter}

\title{Local asymptotic equivalence of pure states ensembles and quantum
Gaussian white noise}

\runtitle{Quantum asymptotic equivalence }


\begin{aug}
\author{\fnms{Cristina} \snm{Butucea,}\ead[label=e1]{cristina.butucea@ensae.fr}}
\author{\fnms{M\u ad\u alin} \snm{Gu\c t\u a}\ead[label=e2]{madalin.guta@nottingham.ac.uk}}
\and
\author{\fnms{Michael} \snm{Nussbaum} \thanksref{t3} \ead[label=e3]{nussbaum@math.cornell.edu}}

\thankstext{t3}{Supported in part by NSF Grant DMS-1407600}
\address{CREST, ENSAE, Universit\'e Paris-Saclay\\ 5, avenue Henry Le Chatelier, \\91120 Palaiseau,\\ France \\ \printead{e1}}
\address{University of Nottingham\\ School of Mathematical Sciences\\ University Park\\ NG7 2RD Nottingham\\UK \\\printead{e2}}
\address{Department of Mathematics \\ Malott Hall, Cornell University \\ Ithaca, NY 14853 \\ USA \\ \printead{e3}}
\affiliation{CREST, ENSAE, Universit\'e Paris-Saclay; University of Nottingham and Cornell University}

\end{aug}

\runauthor{Butucea, C., Gu\c t\u a, M. and Nussbaum, M.}

\bigskip

{\Large \it Dedicated to Richard D. Gill on the occasion of his 66th birthday.}

\begin{abstract}
~Quantum technology is increasingly relying on specialised statistical inference methods for analysing quantum measurement data. This motivates the development of ``quantum statistics'', a field that is shaping up at the overlap of quantum physics and ``classical'' statistics.  One of the less investigated topics to date is that of statistical inference for infinite dimensional quantum systems, which can be seen as quantum counterpart of non-parametric statistics.
In this paper we analyse the asymptotic theory of quantum statistical models consisting of ensembles of quantum systems which are identically prepared in a pure state. In the limit of large ensembles we establish the local asymptotic equivalence (LAE) of this i.i.d. model to a quantum Gaussian white noise model. We use the LAE result in order to establish minimax rates for the estimation of pure states belonging to Hermite-Sobolev classes of wave functions. Moreover, for quadratic functional estimation of the same states we note an elbow effect in the rates, whereas for testing a pure state a sharp parametric rate is attained over the nonparametric Hermite-Sobolev class.
\end{abstract}

\begin{keyword}[class=MSC]
\kwd{Primary }{62B15}
\kwd{; secondary }{62G05, 62G10, 81P50}
\end{keyword}

\begin{keyword}
\kwd{Le Cam distance}
\kwd{local asymptotic equivalence}
\kwd{quantum Gaussian process}
\kwd{quantum Gaussian sequence}
\kwd{quantum states ensemble}
\kwd{nonparametric estimation}
\kwd{quadratic functionals}
\kwd{nonparametric sharp testing rates}
\end{keyword}

\end{frontmatter}

\section{Introduction}
A striking insight of quantum mechanics is that randomness is a fundamental feature of the physical world at the microscopic level. Any observation made on a quantum system such as an atom or a light pulse, results in a non-deterministic, stochastic outcome. The study of the \emph{direct map} from the system's state or preparation to the probability distribution of the measurement outcomes, has been one of the core topics in traditional quantum theory.
%
In recent decades the focus of research has shifted from fundamental physics towards applications at the interface with information theory, computer science, and metrology, sharing the paradigm that individual quantum systems are carriers of a new type of information \cite{NielsenChuang}.

In many quantum protocols, the experimenter has incomplete knowledge and control of the system and its environment, or is interested in estimating an external field parameter which affects the system dynamics. In this case one deals with  a \emph{statistical inverse problem} of inferring unknown state parameters from the measurement data obtained by probing a large number of individual quantum systems.
%
%
The theory and practice arising from tackling such questions is shaping up into the field of \emph{quantum statistics}, which lies at the intersection of quantum theory and statistical inference \cite{Holevo,Helstrom,Hayashi.editor,Paris.editor,Barndorff-Nielsen&Gill&Jupp,Artiles&Gill&Guta}.

One of the central problems in quantum statistics is state estimation: given an ensemble of identically prepared, independent systems with unknown state, the task is to estimate the state by performing appropriate measurements and devising estimators based on the measurement data. A landmark experiment aimed at creating multipartite entangled states \cite{Blatt}
highlighted the direct practical relevance of efficient estimation techniques for large dimensional systems, the complexity of estimating large dimensional states, and the need for solid statistical methodology in computing reliable ``error bars''. This has motivated the development of new methods such as compressed sensing and matrix $\ell_1$-minimisation \cite{Gross,Gross2,Flammia&Gross}, spectral thresholding for low rank states \cite{ButuceaGutaKypraios}, confidence regions \cite{Carpentier,Christandl&Renner,Walter&Renes,Temme&Verstraete,Faist&Renner}.

Another important research direction is towards developing a \emph{quantum decision theory} as the overall mathematical
framework for inference involving quantum systems seen as a form of ``statistical data''.
Typically, the route to finding the building blocks of this theory starts with a decision problem (e.g. testing between two states, or estimating certain parameters of a state) and the problem of finding optimal measurement settings and statistical procedures for treating the (classical, random) measurement data. For instance, in the context of asymptotic binary hypothesis testing, two key results are the quantum Stein lemma \cite{Hiai&Petz,Ogawa&Nagaoka} and the quantum Chernoff bound \cite{Audenaert,Nussbaum&Szkola,Audenaert&Nussbaum,Li}. As in the classical case, they describe the exponential decay of appropriate error probabilities for optimal measurements, and they provide operational interpretations for quantum relative entropy, and respectively quantum Chernoff distance. Similarly,  an important problem in state estimation is to identify measurements which allow for the smallest possible estimation error.   A traditional approach has been to establish a ``quantum Cram\'{e}r-Rao bound'' (QCRB) \cite{Holevo,Helstrom,Braunstein&Caves} for the covariance of unbiased estimators, where the right side is the inverse of the
``quantum Fisher information matrix'', the latter depending only on the structure of the quantum statistical model. However, while the QCRB is achievable asymptotically for one-dimensional parameters, this is not the case for multi-parameter models due to the fact that the measurements which are optimal for different one-dimensional components, are generally incompatible with each other.

These difficulties can be overcome by developing a fundamental theory of comparison and convergence of quantum statistical
models, as an extension of its classical counterpart \cite{Torgersen, LeCam}. While classical ``data processing'' is described by randomisations, physical transformations of quantum systems are described by \emph{quantum
channels} \cite{NielsenChuang}. Following up on this idea, Petz and Jencova \cite{JencovaPetz} have obtained a general characterisation of \emph{equivalent models}, as families of states that are related by quantum channels in both directions. This naturally leads to the notion of \emph{Le Cam distance} between quantum statistical models as the least trace-norm error incurred when trying to map one model into another via quantum channels \cite{Kahn&Guta}. In this framework, the asymptotic theory of state estimation can be investigated by adopting ideas from the classical \emph{local asymptotic normality} (LAN) theory \cite{LeCam}. Quantum LAN theory \cite{GutaKahn2006,GutaJencova2007,Kahn&Guta} shows that the sequence of models describing large samples of identically prepared systems can be approximated by a simpler \emph{quantum Gaussian shift model}, in the neighbourhood of an interior point of the parameter space. The original optimal state estimation problem is then solved by combining LAN theory with  known procedures for estimation of Gaussian states \cite{Guta&Janssens&Kahn,GutaKahnproc,Gill&Guta}.

In this paper we extend the scope of the quantum LAN theory to cover \emph{non-parametric} quantum models; more precisely we will be interested in the set of \emph{pure states} (one-dimensional projections) on \emph{infinite dimensional Hilbert spaces}. Infinite dimensional systems such as light pulses, free particles, are commonly encountered in quantum physics, and their estimation is an important topic in quantum optics \cite{Leonhardt}. The minimax results derived in this paper can serve as a benchmark for the performance of specific methods such as for instance quantum homodyne tomography \cite{Artiles&Gill&Guta,Butucea&Guta&Artiles}, by comparing their risk with the minimax risk derived here.

The paper is organised as follows. In Section \ref{sec.QM} we review the basic notions of quantum mechanics needed for understanding the physical context of our investigation. In particular, we define the concepts of state, measurement and quantum channel which can loosely be seen as quantum analogs of probability distribution and Markov kernels, respectively.
We further introduce the formalism of quantum Gaussian states, the Fock spaces and second quantisation, which establish the quantum analogs of Gaussian distributions, Gaussian sequences and Gaussian processes in continuous time. In Section \ref{sec.quantum.models} we introduce the general notion of a quantum statistical model and the Le Cam distance between two models. In particular, in Section \ref{sec.q.models} we define the i.i.d. and Gaussian quantum models which are analysed in the remainder of the paper.  In Appendix A.1  \cite{Butucea&Guta&Nussbaum} we review results in classical statistics on non-parametric asymptotic equivalence which serve as motivation and comparison to our work.

One of the main results is Theorem~\ref{th.qlan} giving the \emph{local asymptotic equivalence} (LAE) between the non-parametric i.i.d. pure states model and the Gaussian shift model. This extends the existing local asymptotic normality theory from parametric to non-parametric
(infinite dimensional) models.
Section \ref{sec.applications} details three applications of the LAE result in Theorem~\ref{th.qlan}. In Section~\ref{sec.estimation} we derive the asymptotic minimax rates and provide concrete estimation procedures for state estimation with respect to the trace-norm and Bures distances, which are analogues of the norm-one and Hellinger distances respectively.
The main results are Theorems \ref{thm:UBestim} and \ref{thm:lbsobolev.iid} which deal with the upper and respectively lower bound for a model consisting of an ensemble of $n$ independent identically prepared systems in a pure state belonging to a Hermite-Sobolev class $S^\alpha(L)$ of wave functions. In Theorem \ref{thm:UBestim} we describe a specific measurement procedure which provides an estimator whose risk attains the nonparametric rate $n^{-\alpha/(2 \alpha+1)}$. The lower bound follows by using the LAE result to approximate the model with a Gaussian one, combined with the lower bound for the corresponding quantum Gaussian model derived in Theorem \ref{thm:lbsobolev}. In Section \ref{sec.quadratic.functionals} we consider the estimation of a state functional corresponding to the expectation of a power
$N^{2\beta}$ of the number operator. Theorems \ref{thm:Festim} and \ref{thm:Flowbounds} establish the upper and lower bounds for functional estimation for the Hermite-Sobolev class $S^\alpha(L)$. The minimax rates are $n^{-1/2}$ (parametric) if $\alpha \geq 2 \beta$, and $n^{-1 + \beta /\alpha }$ if $\beta < \alpha <2 \beta$. In Section~\ref{sec.testing} we investigate non-parametric testing between a single state and a composite hypothesis consisting of all states outside a ball of shrinking radius. Surprisingly, we find that the minimax testing rates are parametric, in contrast to the non-parametric estimation rates. This fact is closely related to the fact that the optimal estimation and testing measurements are incompatible with each other, so that no single measurement strategy can allow for minimax estimation and testing in the same time. Results on the minimax optimal rate for testing and the
sharp asymptotics are given in Theorems~\ref{thm:Test-1} and~\ref{thm:Test-2} respectively. Further discussion on these topics and proofs of all results are presented in Appendix A and B in  \cite{Butucea&Guta&Nussbaum}, respectively.

\emph{Notation.} Following physics convention, the vectors of a Hilbert
space $\mathcal{H}$ will be denoted by the ``ket'' $|v\rangle$, so that the
inner product of two vectors is the ``bra-ket'' $\langle u|v\rangle\in
\mathbb{C}$ which is linear with respect to the right entry and anti-linear
with respect to the left entry. Similarly, $M:=|u\rangle\langle v| $ is the
rank one operator acting as $M  : |w\rangle \mapsto M|w \rangle = \langle
v |w\rangle |u\rangle$. We denote by $\mathcal{L}(\mathcal{H})$ the space of
bounded linear operators on $\mathcal{H}$ which is a C$^*$-algebra with
respect to the operator norm $\| A \| := \sup_{\psi\neq 0} \|A\psi\|
/\|\psi\| $. Additionally, $\mathcal{T}_1(\mathcal{H})\subset \mathcal{L}(%
\mathcal{H})$ is the space of Hilbert-Schmidt (or trace-class) operators
equipped with the norm-one $\| \tau \|_1: = \mathrm{Tr}(| \tau|)$, where the
operator $|\tau| := (\tau^* \tau)^{1/2}$ is the absolute value of $\tau$,
and $\tau^*$ is the adjoint of $\tau$. Finally, we denote by $\mathcal{T}_2(%
\mathcal{H})\subset \mathcal{L}(\mathcal{H})$ the space of Hilbert-Schmidt
operators equipped with the norm-two $\| \tau \|_2^2 := \mathrm{Tr}(|
\tau|^2)$, which is a Hilbert space with respect to the inner product $%
(\tau, \sigma) := \mathrm{Tr}(\tau^* \sigma)$.


\section{Quantum mechanics background}

\label{sec.QM} 
In this section we review some basic notions of quantum mechanics (QM), in
as much as it is required for understanding the subsequent results of the
paper. Since QM is a probabilistic theory of quantum phenomena, it is
helpful to approach the formalism from the perspective of analogies and differences with
``classical'' probability. We refer to \cite{NielsenChuang} for more details
on the quantum formalism.


\subsection{States, measurements, channels}


The QM formalism assigns to each quantum mechanical system (e.g. an atom,
light pulse, quantum spin) a complex Hilbert space $\mathcal{H}$, called the
space of states. For instance, the finite dimensional space $\mathbb{C}^d$
is the Hilbert space of a system with $d$ ``energy levels'', while $L^2(%
\mathbb{R})$ is the space of ``wave functions'' of a particle moving in one
dimension, or of a monochromatic light pulse.
%
The \emph{state} of a quantum system is represented mathematically by a
density matrix.

\begin{definition}
Let $\mathcal{H}$ be the Hilbert space of a quantum system. A density matrix
(or state) on $\mathcal{H}$ is a linear operator $\rho:\mathcal{H}\to
\mathcal{H}$ which is positive (i.e. it is selfadjoint and has non-negative
eigenvalues), and has trace one.
\end{definition}

We denote by $\mathcal{S}(\mathcal{H})$ the convex space of states on $%
\mathcal{H}$. Its linear span is the space of trace class operators $%
\mathcal{T}_1(\mathcal{H})$, which is the non-commutative analogue of the
space of absolutely integrable functions on a probability space $L_1(\Omega,
\Sigma, \mathbb{P})$. %
%
For any states $\rho_1$ or $\rho_2$, the convex combination $\lambda\rho_1 +
(1-\lambda)\rho_2$ is also a state which corresponds to randomly preparing
the system in either the state $\rho_1$ or $\rho_2$ with probabilities $%
\lambda$ and respectively $1-\lambda$. The extremal elements of the convex
set $\mathcal{S}(\mathcal{H})$ are the one dimensional projections $P_\psi =
|\psi \rangle \langle \psi|$ where $|\psi\rangle$ is a normalised vector,
i.e. $\|\psi\|=1$. Such states are called \emph{pure} (as opposed to mixed
states which are convex combinations of pure ones), and are uniquely
determined by the vector $|\psi\rangle$. Conversely, the vector $|\psi\rangle
$ is fixed by the state up to a complex phase factor, i.e. $|\psi\rangle$
and $|\psi^\prime\rangle :=e^{i\phi} |\psi\rangle$ represent the same state.


Although the quantum state encodes all information about the preparation of
the system, it is not a directly observable property. Instead, any
measurement produces a \emph{random outcome} whose distribution depends on
the state, and thus reveals in a probabilistic way a certain aspect of the
system's preparation. The simplest type of measurement is determined by an
orthonormal basis (ONB) $\{ |i\rangle \}_{i=1}^{\mathrm{dim} \mathcal{H}}$
and a set of possible outcomes $\{\lambda_i\}_{i=1}^{\mathrm{dim} \mathcal{H}}$ in the
following way: the outcome is a random variable $X$ taking the value $%
\lambda_i$ with probability given by the diagonal elements of $\rho$ in this
particular basis
\begin{equation*}
\mathbb{P}_\rho ( [X= \lambda_i]) = \rho_{ii} = \langle i | \rho | i\rangle.
\end{equation*}
More generally, a measurement $M$ with outcomes in a measurable space $(\Omega,
\Sigma)$ is determined by a positive operator valued measure.

\begin{definition}\label{def.POVM}
A positive operator valued measure (POVM) is a map $M:\Sigma\to \mathcal{L}(%
\mathcal{H})$ having the following properties
\begin{itemize}
\item[1)] positivity: $M(E) \geq 0$ for all events $E\in \Sigma$
\item[2)] $\sigma$-additivity: $M(\cup_i E_i) = \sum_i {M}(E_i)$ for
any countable set of mutually disjoint events $E_i$
\item[3)] normalization: $M(\Omega) = \mathbf{1}$.
\end{itemize}
The outcome of the corresponding measurement associated to $M$ has probability
distribution
\begin{equation*}
\mathbb{P}_\rho(E) = \mathrm{Tr}(\rho M(E)), \qquad E\in \Sigma.
\end{equation*}
\end{definition}

The most important example of a POVM, is that associated to the measurement
of an observable, the latter being represented mathematically by a
selfadjoint operator $A:\mathcal{H}\to \mathcal{H}$. The Spectral Theorem
shows that such operators can be ``diagonalised'', i.e. they have a spectral
decomposition
\begin{equation*}
A= \int_{\sigma(A)} x P(dx)
\end{equation*}
where $\sigma(A)$ is the spectrum of $A$, and $\{P(E) :E\in \Sigma\}$ is the
collection of spectral projections of $A$. The corresponding measurement has
outcome $a\in \sigma(A)$ with probability distribution $\mathbb{P}_\rho\left[%
a\in E\right] = \mathrm{Tr}(\rho P(E))$.

Unlike ``classical'' systems which can be observed without disturbing their
state, quantum systems are typically perturbed by the measurement, so the
system needs to be reprepared in order to obtain more information about the
state. In this sense, the system can be seen as a ``quantum sample'' which
it can be converted into a ``classical'' sample only by performing a
measurement. Thus, a measurement can be seen as a ``quantum-to-classical
randomisation'', i.e. a linear map $\mathcal{M}$ which sends a state $\rho$
to the probability density $\mathcal{M} (\rho) \equiv p_\rho := \frac{d%
\mathbb{P}_\rho}{d\mathbb{P}}$ with respect to a reference measure $\mathbb{P%
}$. 
The latter can be taken to be $\mathbb{P}_{\rho_0}$ for a strictly positive
density matrix $\rho_0$. The following lemma summarises this perspective on
measurements.

\begin{lemma}
Let $\mathcal{H}$ be a Hilbert space, and let $(\Omega, \Sigma)$ be a
measurable space. For any fixed state $\rho_0>0$ on $\mathcal{H}$, there is a
one-to-one correspondence between POVMs $M$ over $(\Omega, \Sigma)$ and
quantum-to-classical randomisations, i.e. linear maps
\begin{equation*}
\mathcal{M} : \mathcal{T}_1 (\mathcal{H}) \to L_1(\Omega, \Sigma, \mathbb{P}%
)
\end{equation*}
which are positive and normalised (maps states into probability densities).
The correspondence is given by
\begin{equation*}
\mathbb{P}_\rho (E) = \mathrm{Tr}(M(E)\rho) = \int_A p_\rho(\omega) \mathbb{P%
}_{\rho_0} (d \omega), \qquad \mathcal{M} (\rho) \equiv p_\rho := \frac{d%
\mathbb{P}_\rho}{d\mathbb{P}} .
\end{equation*}
\end{lemma}
For comparison, recall that a linear map $R: L_1(\Omega^{\prime},\Sigma^{\prime}, \mathbb{P}^{\prime})
\to L_1(\Omega,\Sigma, \mathbb{P})$ is a stochastic operator if it maps
probability densities into probability densities \cite{Strasser}. Typically such maps arise from Markov kernels and describe randomizations of dominated statistical experiments (models).

While a measurement is a quantum-to-classical randomization, a
``quantum-to-quantum randomization'' describes how the system's state
changes as a result of time evolution or interaction with other systems. The
maps describing such transformations are called quantum channels.
\begin{definition}
A quantum channel between systems with Hilbert spaces $\mathcal{H}_1$ and $%
\mathcal{H}_2$ is a trace preserving, completely positive linear map
$
T: \mathcal{T}_1(\mathcal{H}_1) \to \mathcal{T}_1(\mathcal{H}_2).
$
\end{definition}
%
The two properties mentioned above are similar to those of a classical
randomization, so in particular $T$ maps states into states. However, unlike
the classical case, $T$ is required to satisfy a stronger positivity
property: $T$ is \emph{completely positive} if $\mathrm{Id}_m \otimes T$ is
positive for all $m\geq 1$, where $\mathrm{Id}_m$ is the identity map on the space of $m$ dimensional matrices.
This ensures that when the system is correlated with an
ancillary system $\mathbb{C}^m$, and the latter undergoes the identity
transformation, the final joint state is still positive, as expected on
physical grounds.


The simplest example of a quantum channel is a unitary transformation $%
\rho\mapsto U\rho U^*$, where $U$ is a unitary operator on $\mathcal{H}$.
More generally, if $|\varphi \rangle\in \mathcal{K}$ is a pure state of an
ancillary system, and $V$ is a unitary on $\mathcal{H}\otimes \mathcal{K}$,
then
\begin{equation*}
\rho\mapsto T(\rho):= \mathrm{Tr}_\mathcal{K} ( V(\rho\otimes
|\varphi\rangle \langle\varphi |)V^* ) 
\end{equation*}
is a quantum channel describing the system state after interacting with the
ancilla. By computing the partial trace $\mathrm{Tr}_\mathcal{K} $ over $%
\mathcal{K}$ with respect to an orthonormal basis $\{|f_i\rangle \}_{i=1}^{%
\mathrm{dim} \mathcal{K}}$ we obtain the following expression
\begin{equation}  \label{eq.kraus}
T(\rho) = \sum_{i} {K}_i \rho K_i^*
\end{equation}
where $K_i$ are operators on $\mathcal{H}$ defined by $\langle \psi |K_i |
\psi^\prime\rangle :=\langle \psi \otimes f_i | U| \psi^\prime\otimes
\varphi \rangle$. Note that by definition, these operators satisfy the
normalisation condition $\sum_i K_i^* K_i = \mathbf{1}$. Conversely, the
Kraus Theorem shows that any quantum channel is of the form \eqref{eq.kraus}
with operators $K_i$ respecting the normalisation condition.


\subsection{Continuous variables, Fock spaces and Gaussian states}

In this section we look at the class of ``continuous variables'' (cv) systems,
which model a variety of physical systems such as light pulses, or free
particles. Such systems play an important role in this work as ``carriers''
of quantum Gaussian states, and in particular in the local asymptotic
equivalence result. We refer to \cite{Leonhardt} for further reading.


\subsubsection{One mode systems}

\label{sec.one.mode} 
We start with the simplest case of a ``one-mode'' cv system, after which we
show how this construction can be extended to more general ``multi-mode'' cv
systems. The Hilbert space of a one-mode system is $L_2(\mathbb{R})$, i.e.
the space of square integrable wave functions on the real line. On this we
define the selfadjoint operators acting on appropriately defined domains as
\begin{equation*}
(Q \psi)(q) = q\psi(q), \qquad (P\psi)(q) = -i \frac{d\psi(q)}{dq}
\end{equation*}
which  satisfy the ``canonical commutation relations'' $QP- PQ = i\mathbf{1}$.
To better understand the meaning of the observable $Q$, let us consider its
measurement for a pure state $\rho = |\psi \rangle \langle \psi|$ with wave function $|\psi\rangle$. The outcome
takes values in $\mathbb{R}$, and its probability distribution has density
with respect to the Lebesgue measure $p^Q_\rho(x) =|\psi(x)|^2$ . Similarly,
the probability density of the observable $P$ is given by $p^P_\rho(x) =|\tilde{\psi}
(x)|^2$, where $\tilde{\psi}\in L^2(\mathbb{R})$ is the Fourier transform of
the function $\psi(\cdot)$. When the system under consideration is the free
particle, $Q$ and $P$ are usually associated to the position and momentum
observables, while for a monochromatic light mode they correspond to the
electric and magnetic fields. Note that the distributions of $P$ and $Q$ are
not sufficient to identify the state, even in the case of a pure state.
However, it turns out that the state is uniquely determined by the
collection of probability distributions of all \emph{quadrature} observables
$X_\phi:= \cos(\phi) \cdot Q + \sin(\phi) \cdot P$ for angles $\phi\in [0,2\pi]$. To
understand this, it is helpful to think of the state of the one-mode cv
system as a quantum analogue of a joint distribution of two real valued
variables, i.e. a 2D distribution. Indeed, in the latter case, the
distribution is determined by collection of marginals along all directions
in the plane (its Radon transform); this fact is exploited in PET tomography
which aims at estimating the 2D distribution from samples of its Radon
transform. In the quantum case, since $Q$ and $P$ do not commute with each
other, they cannot be measured simultaneously and cannot be assigned a joint
distribution in a meaningful way. However, the ``quasi-distribution''
defined below has some of the desired properties, and is very helpful in
visualising the quantum state.
\begin{definition}
For any state $\rho\in \mathcal{T}_1(L_2(\mathbb{R}))$ we define the \emph{%
quantum characteristic function} of $\rho$ 
\begin{equation*}
\widetilde{W}_\rho(u,v) := \mathrm{Tr} (\exp(-iuQ - ivP)\rho). 
\end{equation*}
The inverse Fourier transform of $\widetilde{W}_\rho$ with respect to both
variables is called Wigner function $W_\rho$, or quasi-distribution
associated to $\rho$:
\begin{equation*}
W_\rho(q,p) = \frac{1}{(2\pi)^2} \int\!\int \exp(iuq+ ivp) \widetilde{W}%
_\rho(u,v) du dv.
\end{equation*}
\end{definition}
A consequence of this definition is that the marginal of $W_\rho(q,p)$ along
an arbitrary direction with angle $\phi$ is the probability density of the
quadrature $X_\phi$ introduced above. This is the basis of a quantum state
estimation scheme called ``quantum homodyne tomography'' \cite%
{Leonhardt,Artiles&Gill&Guta}, where the Wigner function plays the role of
the 2D distribution from ``classical'' PET tomography. One of the important
differences however, is that the Wigner functions need not be positive in
general, and satisfy other constraints which are specific to the quantum
setting and can be exploited in the estimation procedure.%

The Wigner function representation offers an intuitive route to defining the
notion of Gaussian state.

\begin{definition}
A state $\rho$ of a one-mode cv system is called Gaussian if its Wigner
function $W_\rho$ is a Gaussian probability density, or equivalently if it
has the quantum characteristic function
\begin{equation*}
\widetilde{W}_\rho (u,v) = \exp \left(-(u,v) \frac{V}{2} (u,v)^T
\right) \cdot \exp (iu q_0 + iv p_0).
\end{equation*}
where $(q_0, p_0)\in \mathbb{R}^2$ and $V$ (a real positive $2\times2$ matrix) are the
mean and variance of $W_\rho$, respectively.
\end{definition}

In particular, all the quadratures $X_\phi$ of a Gaussian state have
Gaussian distribution. As consequence of the commutation relation $QP-PQ= i%
\mathbf{1}$ the observables $Q$ and $P$ cannot have arbitrarily small
variance simultaneously; in particular, the covariance matrix $V$ must
satisfy the ``uncertainty principle'' $\mathrm{Det}(V) \geq 1/4$, where
the equality is achieved if and only if the state is a pure Gaussian state.

We will be particularly interested in \emph{coherent} states $|G(z)\rangle $ which are \emph{%
pure} Gaussian states whose Wigner functions have covariance matrix $V = \mathrm{I}_2 /2$, where $I_2$ is the $2 \times 2 $ identity matrix.
To give a concrete Hilbert space representation,
it is convenient to introduce a special orthonormal basis of $L_2(\mathbb{R})
$, consisting of the eigenvectors $\{|0\rangle, |1\rangle, \dots\}$ of the
\emph{number operator} $N= a^*a$, with $N |k\rangle = k|k\rangle$. Here, the
operators $a^*= (Q-iP)/\sqrt{2}$ and $a= (Q+iP)/\sqrt{2}$ are called \emph{%
creation and annihilation operators} and act as ``ladder operators'' on the
number basis vectors (or Fock states)
\begin{equation*}
a|k\rangle = \sqrt{n}|k-1\rangle, \qquad a^*|k\rangle = \sqrt{k+1}%
|k+1\rangle. 
\end{equation*}
The coherent states denoted by $|G(z)\rangle $ are obtained
by applying the unitary Weyl (displacement) operators to the vacuum state  $|0\rangle $
\begin{equation}\label{eq.coherent.state}
| G(z)\rangle = \exp\left( za^* -\bar{z}a\right) |0\rangle =
\exp(-|z|^2/2)\sum_{k=0}^\infty \frac{z^k}{\sqrt{k!}} |k\rangle,
\end{equation}
where $z\in \mathbb{C}$ is the eigenvalue of the annihilation operator
$a | G(z)\rangle = z | G(z)\rangle$; in particular, the quadrature means are
$\langle G(z)|Q|G(z)\rangle= \sqrt{2} {\rm Re}(z) $ and $\langle G(z)|P|G(z)\rangle = \sqrt{2}{\rm Im}(z)$, and the Wigner function is given by
\begin{equation}\label{eq.wigner.coherent}
W_{|z\rangle }(q,p)=\frac{1}{\pi}\exp \left( -(q-\sqrt{2}x)^{2}-(p-\sqrt{2}%
y)^{2}\right) ,\quad q,\,p\in \mathbb{R}.
\end{equation}%

Equation \eqref{eq.coherent.state} implies that the number operator $N$ has a Poisson distribution with
mean $|z|^2$. Additionally, it can be seen from the Fourier expansion in the
second equality that the unitary $\Gamma(\phi) = \exp(i\phi N)$ acts by
rotating the coherent states by an angle $\phi$ in the complex plane, i.e. $\Gamma(\phi) |G(z)\rangle= |G(e^{i\phi} z)\rangle$.

Another important class of Gaussian states are the mixed diagonal states
\begin{equation}\label{eq.thermal.state}
\Phi(r) = (1-r)\sum_{k=0}^\infty  r^k | k \rangle\langle k | , \quad 0< r <1
\end{equation}
which are also called thermal states, cf. section 3.3 in \cite{Leonhardt}. The corresponding Wigner function is a centred Gaussian \begin{equation*}
W_{\Phi(r)}(q,p) = \frac{1}{2\pi \sigma^2(r)} \exp\left( - \frac{q^2 +p^2 }{ 2 \sigma^2(r)} \right).
\end{equation*}
with covariance matrix $V= \sigma^2(r)\cdot I_2$ where $\sigma^2(r) = \frac{1}{2}\frac{1+r}{1-r}$.

\begin{proposition}
 Consider the family of coherent states $ \{|G(z)\rangle\langle G(z)|, \, z \in \mathbb{C}\}$, with random displacement (location) $z$ distributed according to $\Pi(dz)$, having a Gaussian law with covariance matrix $\sigma^2 \cdot I_2 $. Then, the mixed state $\Phi = \int  |G(z)\rangle\langle G(z)| \Pi(dz)$ is the thermal state
 $
 \Phi(r)$ with $r = \frac{2 \sigma^2}{2 \sigma^2 +1}$.
\end{proposition}
\begin{proof}
Consider the corresponding Wigner function
\begin{eqnarray}
W_{\Phi}(q,p)&=&\int W_{|G(z)\rangle }(q,p)\exp \left( -%
\frac{1}{2\sigma^{2}}(x^{2}+y^{2})\right) \frac{1}{2\pi \sigma^{2}}\,dx dy \nonumber \\
&=&\frac{1}{\pi \sigma^{2}}\int \exp \left( -(q-\sqrt{2}\,x)^{2}-%
\frac{x^{2}}{2\sigma^{2}}\right) \frac{dx}{\sqrt{2\pi }} \nonumber \\
&\times& \int \exp \left( -(p-\sqrt{2}\,y)^{2}-\frac{y^{2}}{2\sigma^{2}}%
\right) \frac{dy}{\sqrt{2\pi }}  \nonumber\\
&=&\frac{1}{\pi (4\sigma^{2}+1)}\exp \left( -\frac{q^{2}+p^{2}}{%
2(2\sigma^{2}+1/2)}\right) .\label{eq.twirling}
\end{eqnarray}%
Therefore, the state $\Phi$ is identical to the thermal state $\Phi(r)$ with $2 \sigma^2+\frac 12 = \frac 12 \frac{1+r}{1-r}$, or equivalently $r= \frac{2\sigma^2}{1+2\sigma^2}$.
\end{proof}
This fact will be used later on in in section \ref{sec.applications} in applications to functional estimation and testing.

\subsubsection{Fock spaces and second quantisation}

\label{sec.Fock.spaces} 

The above construction can be generalised to multimode systems by tensoring
several one-mode systems. Thus, the Hilbert space of a $k$-mode system is $%
L_2(\mathbb{R})^{\otimes k} \cong L_2(\mathbb{R}^{k})$, upon which we define
``canonical pairs'' $(Q_i, P_i)$ acting on the $i$-th tensor as above, and
as identity on the other tensors. Similarly we define the one-mode operators
$a_i, a^*_i, N_i$. The number basis consists now of tensor products $|%
\mathbf{n}\rangle:= \otimes_{i=1}^k |n_i\rangle$ indexed by the sequences of
integers $\mathbf{n} = (n_1, \dots, n_k)$. A multimode coherent state is a
tensor product of one-mode coherent states
\begin{eqnarray}  \label{eq.coherent.multimode}
| G(\mathbf{z})\rangle
& =& \otimes_{i=1}^k |G(z_i)\rangle = \exp\left( \mathbf{%
z}\mathbf{a}^\dagger -\mathbf{a}\mathbf{z}^\dagger \right) |0\rangle \nonumber \\
&=& \exp(-|z|^2/2)\sum_{\mathbf{n}=0}^\infty \left(\prod_{i=1}^k \frac{z_i^n} {%
\sqrt{n_i!}} \right) |\mathbf{n}\rangle\in L_2(\mathbb{R})^{\otimes k}
\end{eqnarray}
where $\mathbf{z} = (z_1, \dots , z_k)$ is the vector of means, $\mathbf{a}
= (a_1, \dots, a_k)$, and $\dagger$ denotes the transposition and adjoint
(complex conjugation) of individual entries.

We will now extend this construction to systems with infinitely many modes.
One way to do this is by defining an infinite tensor product of one-mode
spaces, as completion of the space spanned by tensors in which all but a
finite number of modes are in the vacuum state. Instead, we will present an
equivalent but more elegant construction called \emph{second quantisation}
which will be useful for later considerations.

\begin{definition}
Let $\mathcal{K}$ be a Hilbert space. The \emph{Fock space} over $\mathcal{K}
$ is the Hilbert space
\begin{equation}  \label{eq.Fock.space}
\mathcal{F}(\mathcal{K}) = \bigoplus_{n \geq 0} \mathcal{K}^{\otimes_s n}
\end{equation}
where $\mathcal{K}^{\otimes_s n}$ denotes the $n$-fold symmetric tensor
product, i.e. the subspace of  $\mathcal{K}^{\otimes n}$ consisting of
vectors which are symmetric under permutations of the tensors. The term  $%
\mathcal{K}^{\otimes_s 0} =: \mathbb{C}|0\rangle $ is called the vacuum
state.
\end{definition}

In this definition the space $\mathcal{K}$ should be regarded as the
``space of modes'' rather than physical states. As we will see below, by
fixing an orthonormal basis in $\mathcal{K}$, we can establish an
isomorphism between the Fock space $\mathcal{F}(\mathcal{K})$ and a tensor
product of one-mode cv spaces, one for each basis vector. In particular, if
$\mathcal{K}= \mathbb{C}$, then $\mathcal{F}(\mathbb{C})\cong L^2(\mathbb{R})
$ so that the one-dimensional subspaces in the direct sum in %
\eqref{eq.Fock.space} correspond to the number basis vectors $|0\rangle,
|1\rangle, \dots \in L^2(\mathbb{R})$ of a one-mode cv system.

We now introduce the general notion of coherent state on a Fock space.
\begin{definition}
Let $\mathcal{F}(\mathcal{K})$ be the Fock space over $\mathcal{K}$. For
each $|v\rangle \in\mathcal{K}$ we define an associated coherent state
\begin{equation*}
|G(v) \rangle := e^{-\|v\|^2/2} \bigoplus_{n \geq 0} \frac{1}{\sqrt{n!}}%
|v\rangle^{\otimes n}\in \mathcal{F}(\mathcal{K}).
\end{equation*}
\end{definition}
The coherent vectors form a dense subspace of $\mathcal{F}(\mathcal{K})$.
This fact can be used to prove the following factorisation property, and to
define the annihilation operators below. Let $\mathcal{K}= \mathcal{K}_0
\oplus \mathcal{K}_1$ be a direct sum decomposition of $\mathcal{K}$ into
orthogonal subspaces, and let $|v\rangle = |v_0\rangle\oplus|v_1\rangle$ be
the decomposition of a generic vector $|v\rangle\in\mathcal{K}$. Then the
map
\begin{eqnarray*}
U: \mathcal{F}(\mathcal{K}) &\to& \mathcal{F}(\mathcal{K}_0)\otimes \mathcal{%
F}(\mathcal{K}_1) \\
U: |G(v) \rangle &\mapsto & |G(v_0)\rangle \otimes |G(v_1)\rangle
\end{eqnarray*}
is unitary. We will use this correspondence to identify $\mathcal{F}(%
\mathcal{K})$ with the tensor product $\mathcal{F}(\mathcal{K}_0)\otimes
\mathcal{F}(\mathcal{K}_1)$. By the same argument, for any orthonormal basis
$\{ |e_1\rangle, |e_2\rangle,\dots \}$ of $\mathcal{K}$, the Fock space $%
\mathcal{F}(\mathcal{K})$ is isomorphic with the tensor product of one mode
spaces $\mathcal{F}_i: =\mathcal{F}( \mathbb{C}|e_i\rangle)$ and the
coherent states factorise as
\begin{eqnarray}
\mathcal{F}(\mathcal{K}) &\cong& \bigotimes_i \mathcal{F}_i  \notag \\
|G(u) \rangle &\cong& \bigotimes_i |G(u_i) \rangle ,\qquad u_i = \langle
e_i|u\rangle.  \label{eq.factorisation.Gaussian}
\end{eqnarray}
so that we recover the formula \eqref{eq.coherent.multimode}.

We define the \emph{annihilation operators} through their action on coherent states
as follows: for each mode $|u\rangle \in \mathcal{K}$ the associated
annihilator $a(u):\mathcal{F}(\mathcal{K})\to\mathcal{F}(\mathcal{K})$ is
given by
\begin{equation*}
a(u): |G(v)\rangle = \langle u|v\rangle |G(v)\rangle , \qquad |v\rangle\in
\mathcal{K}.
\end{equation*}
Then the annihilation and (their adjoint) the creation operators satisfy the
commutation relations
\begin{equation*}
a(u) a^* (w) - a^* (w) a(u) = \langle u|v\rangle \mathbf{1}.
\end{equation*}
For each mode we can also define the canonical operators $Q(u), P(u)$ and
the number operator $N(u)$ in terms of $a(u), a^*(u)$ as in the one-mode
case. Moreover, if $|u\rangle = |u_0 \rangle \oplus |u_1 \rangle $ is the
decomposition of $|u\rangle$ as above, then $a(u_0)$ acts as $a(u_0)\otimes%
\mathbf{1}_{\mathcal{F}(\mathcal{K}_1)}$, when the Fock space is represented
in the tensor product form. Similar decompositions hold for $%
a^*(u_0),N(u_0), a(u_1), a^*(u_1), N(u_1)$.

The second quantisation has the following functorial properties which will
be used later on.
\begin{definition}
Let $W: \mathcal{K}\to\mathcal{K}$ be a unitary operator. The quantisation
operator $\Gamma(W)$ is the unitary defined by $\Gamma(W):\mathcal{F}(%
\mathcal{K})\to\mathcal{F}(\mathcal{K})$ by
\begin{equation*} 
\Gamma(W) : = \bigoplus_{n\geq 0} W^{\otimes n}
\end{equation*}
where $W^{\otimes n}$ acts on the $n$-th level of the Fock space $\mathcal{K}%
^{\otimes_s n}$.
\end{definition}
From the definition it follows that the action of $\Gamma(W)$ on coherent
states is covariant in the sense that
\begin{eqnarray*}
\Gamma(W): \mathcal{F}(\mathcal{K}) &\to & \mathcal{F}(\mathcal{K}) \\
\Gamma (W): |G(v) \rangle & \mapsto & |G(Wv)\rangle.
\end{eqnarray*}
In particular, it follows from the definitions that $\Gamma(e^{i\phi}\mathbf{%
1}) = \exp(i\phi N)$, where $N$ is the total number operator, whose action
on the $n$-th level of the Fock space is $N |v\rangle^{\otimes n} = n
|v\rangle^{\otimes n}$. Note that while $|v\rangle$ and $e^{i\phi}|v\rangle$
differ only by a phase and hence represent the same state, the corresponding
coherent states $|G(v)\rangle$ and $\Gamma(e^{i\phi}) |G(v)\rangle =
|G(e^{i\phi} v)\rangle$ are linearly independent and represent different
states.

As in the single mode case, the coherent states can be obtained by acting
with the unitary displacement (or Weyl) operators onto the vacuum
\begin{equation*}
|G(u)\rangle = \exp (a^*(u) - a(u) ) |0\rangle
\end{equation*}
Moreover, the coherent states $|G(u)\rangle$ are Gaussian with respect to all
coordinates. The means of annihilation operators are given by $\langle G(u)
|a(w)|G(u)\rangle = \langle w |v\rangle $, from which we can deduce that the
the coordinates $(Q(w), P(w))$ have means $(\sqrt{2} \mathrm{Re} \langle w
|u\rangle , \sqrt{2} \mathrm{Im} \langle w |u\rangle)$. The covariance of
coherent states is constant (independent of the displacement $u$), and is
given by $\langle 0 | a(w) a^*(v)|0 \rangle = \langle w|v\rangle$. This
implies that orthogonal modes (i.e. $\langle w|v\rangle=0)$ have independent
pairs of coordinates.



\subsection{Metrics on the space of states}
\label{sec.metrics}

For future reference we review here the states space metrics used in the
paper. Recall that the space of states $\mathcal{S(\mathcal{H})}$ on a
Hilbert space $\mathcal{H}$ is the cone of positive, trace one operators in $%
\mathcal{T}_1(\mathcal{H})$. The \emph{norm-one} (or trace-norm) distance between two states $%
\rho_0,\rho_1\in \mathcal{S(\mathcal{H})}$ is given by
\begin{equation*}
\|\rho_0 - \rho_1\|_1 := \mathrm{Tr} (| \rho_0 - \rho_1|)
\end{equation*}
where $|\tau| := \sqrt{\tau^* \tau}$ denotes the absolute value of $\tau$.
The norm-one distance can be interpreted in terms of  the maximum difference between
expectations of bounded observables
\begin{equation*}
\|\rho_0 - \rho_1\|_1 = 2 \sup_{A: \|A\|\leq 1} | \mathrm{Tr}(\rho_0 A) -
\mathrm{Tr}(\rho_1 A) |.
\end{equation*}
Another interpretation is in terms of quantum testing. Let $M= (M_0, M_1)$
be a binary POVM used to test between hypotheses $H_0: = \{\text{measured
state is } \rho_0\}$ and $H_1: = \{\text{measured state is } \rho_1\}$.
The sum of error probabilities is
\begin{equation*}
\mathbb{P}^M_e = \mathrm{Tr} (M_0 \rho_1) + \mathrm{Tr} (M_1 \rho_0).
\end{equation*}
By optimizing over all possible POVM we obtain \cite{Helstrom} the optimal
error probability sum
\begin{equation}  \label{HHbound}
\mathbb{P}^*_e:= \inf_M \mathbb{P}^M_e  = 1 - \frac{1}{2} \|\rho_0 - \rho_1 \|_1.
\end{equation}
In the special case of pure states, the norm-one distance is given by
\begin{equation}  \label{eq.trace.norm.pure}
\| |\psi_0 \rangle\langle \psi_0| - |\psi_1 \rangle\langle \psi_1 |\|_1 = 2
\sqrt{1 - |\langle\psi_0 | \psi_1 \rangle|^2},
\end{equation}
as proven e.g. in \cite{Kargin-Chern}. The previous formula becomes for coherent states
\begin{equation*}
\| |G(\psi_0) \rangle\langle G (\psi_0) | - |G(\psi_1) \rangle\langle
G(\psi_1) |\|_1 = 2 \sqrt{1 - \exp(- \|\psi_0 - \psi_1 \|^2)}.
\end{equation*}

\noindent The second important metric is the \emph{Bures distance} whose square is given by
\begin{equation*}
d^2_b (\rho_0, \rho_1) := 2 (1 - \mathrm{Tr} \left( \sqrt{\sqrt{\rho_0}
\rho_1 \sqrt{\rho_0}})\right)
\end{equation*}
and is a quantum extension of the Hellinger distance. In the case of pure
states the Bures distance becomes
\begin{equation}  \label{eq.Bures.normone.pure}
d_b^2( |\psi_0 \rangle\langle \psi_0| \, , \, |\psi_1 \rangle\langle \psi_1|
) = 2(1 -|\langle\psi_0 | \psi_1 \rangle| )
\end{equation}
so for coherent states it is given by
\begin{equation*}
d^2_b \left(|G(\psi_0) \rangle\langle G (\psi_0) | \, ,\, |G(\psi_1)
\rangle\langle G(\psi_1) |\right) := 2 \left(1 - \exp \left( -\frac{1}{2} \|
\psi_0-\psi_1 \|^2\right)\right).
\end{equation*}

Similarly to the classical case, the following inequality holds for arbitrary states \cite{Fuchs}
\begin{equation}  \label{eq.distances.inequality}
d_b^2(\rho_0, \rho_1)\leq \|\rho_0-\rho_1 \|_1 \leq  2 d_b(\rho_0,
\rho_1).
\end{equation}
Moreover, since $|\langle\psi_0|\psi_1\rangle |^2 \leq |\langle\psi_0|\psi_1\rangle |$, the additional inequality holds for \emph{pure} states
\begin{equation}  \label{eq.distances.inequality2}
\|\rho_0 -\rho_1 \|_1 \geq \sqrt{2} d_b(\rho_0, \rho_1).
\end{equation}
This means that for pure states, the trace and Bures distances are equivalent (up to constants).

Finally, we will be using the fact that both the norm-one and the Bures
distance are contractive under quantum channels. $T:\mathcal{T}_1(\mathcal{H}%
)\to \mathcal{T}_1(\mathcal{H}^\prime)$, i.e.
\begin{equation*}
\|T(\rho_0) - T(\rho_1)\|_1 \leq \|\rho_0 - \rho_1\|_1, \qquad
d_b^2(T(\rho_0), T(\rho_1)) \leq d_b^2(\rho_0, \rho_1).
\end{equation*}


\section{Quantum statistical models}
\label{sec.statistical.models}

In this section we review key elements of quantum statistics, and introduce
the quantum statistical models which will be analysed later on. For comparison,
we briefly review certain asymptotic equivalence results for related
classical statistical models.


 The classical density model consists of $n$
observations $X_{1},\ldots,X_{n}$ which are independent, identically
distributed (i.i.d.) with common probability density $f$. In the
Gaussian white noise model, a function
$g  \in \mathbb{L}_{2}[0,1]$ is observed with Gaussian white noise of variance $n^{-1}$, i.e.
\begin{equation}
dY_{t}=g(t)dt+\frac{1}{\sqrt{n}}dW_{t},\quad t\in\lbrack0,1].\label{gwm}%
\end{equation}
This model is equivalent to the  Gaussian sequence model, where we observe a sequence of
Gaussian random variables with means equal to the coefficients $ \theta_{j} $ of $g$ in
some orthonormal basis of $\mathbb{L}_{2}[0,1]$
\begin{equation}
y_{j}=\theta_{j}+\frac{1}{\sqrt{n}}\xi_{j},\qquad i=1,2,\ldots
\label{eq.gsm}%
\end{equation}
where $\{\xi_{i}\}_{i\geq1}$ are Gaussian i.i.d. random variables.

In \cite{Nussbaum96} it was shown that for densities $f$  on $[0,1]$,   the  i.i.d.  model is  asymptotically equivalent to the white noise model \eqref{gwm} for $g= f^{1/2} $,  in the sense that
the Le Cam distance of the models converges to zero as $n\rightarrow\infty$ when $f$ varies in a certain smoothness class of functions.
For  recent related results and extensions cf.  \cite{RaySchmidt};
in  Appendix A.1  \cite{Butucea&Guta&Nussbaum} we present a more detailed  review of asymptotic equivalence results for classical statistical models.

\subsection{Quantum models, randomisations and convergence}
\label{sec.quantum.models}

In this subsection we introduce the basic notions of a theory of quantum
statistical models which is currently still in its early stages, cf. \cite%
{GutaJencova2007,Gill&Guta} for more details. We will focus on the notions
of quantum-to-classical randomisation carried out through measurements, and
quantum-to-quantum randomisations implemented by quantum channels, which
allow us to define the equivalence and the Le Cam distance between models.

In analogy to the classical case, we make the following definition.
\begin{definition}
A quantum statistical model over a parameter space $\Theta$ consists of a
family of quantum states $\mathcal{Q} = \{\rho_\theta :\, \theta\in \Theta \}
$ on a Hilbert space $\mathcal{H}$, indexed by an unknown parameter $%
\theta\in \Theta$.
\end{definition}
A simple example is a family of pure states $\{ \rho_{\theta}= |\psi_\theta
\rangle \langle\psi_\theta | : \, \theta\in \mathbb{R}\}$ with $|\psi_\theta
\rangle:= \exp(i\theta H)|\psi\rangle$, where $H$ is a selfdajoint operator
generating the one-dimensional family of unitaries $\exp(i\theta H)$, and $%
|\psi\rangle\in \mathcal{H}$ is a fixed vector. Physically, the parameter $%
\theta$ could be for instance time, a phase, or an external magnetic field.
Another example is that of a completely unknown state of a finite
dimensional system, which can be parametrised in terms of its density matrix
elements, or the eigenvalues and eigenvectors. In order to increase the
estimation precision one typically prepares a number $n$ of identical and
independent copies of the state $\rho_\theta$, in which case the
corresponding model is $\mathcal{Q}_n:= \{\rho_\theta^{\otimes n} :\,
\theta\in \Theta \}$. Our work deals with \emph{non-parametric} quantum
statistical models for which the underlying Hilbert space is infinite
dimensional, as we will detail below.

In order to obtain information about the parameter $\theta$, we need to
perform measurements on the system prepared in $\rho_\theta$. Using the
random measurement data, we then employ statistical methods to solve
specific decision problems. For instance, the task of estimating an unknown
quantum state (also known as quantum tomography) is a key component of
quantum engineering experiments \cite{Blatt}. In particular, the estimation
of large dimensional states has received significant attention in the
context of compressed sensing \cite{Gross,Flammia&Gross}, and estimation of low rank
states \cite{ButuceaGutaKypraios}. Suppose that we perform a measurement $M$
on the system in state $\rho_\theta$, and obtain a random outcome $O\in
\Omega$ with distribution $\mathbb{P}^M_\theta(E) := \mathrm{Tr}(\rho_\theta
M(E))$, cf. section \ref{sec.QM}. The measurement data is therefore
described by the classical model $\mathcal{P}^M := \{ \mathbb{P}^M_\theta
:\, \theta\in \Theta\}$, and the estimation problem can be treated using
``classical'' statistical methods. The measurement map
\begin{eqnarray*}
\mathcal{M}&:&\mathcal{T}_1 \to L_1(\Omega, \Sigma, \mathbb{P}) \\
\mathcal{M}&:&\rho_\theta\mapsto p_\theta := \frac{d\mathbb{P}_\theta}{d\mathbb{P}}
\end{eqnarray*}
can be seen as a randomisation from a quantum to a classical model, which
intuitively means that $\mathcal{Q}$ is more informative that $\mathcal{P}^M$
for any measurement $M$. Here $\mathbb{P}$ can be chosen to be the
distribution corresponding to an arbitrary full rank (strictly positive)
state $\rho$ which insures the existence of all probability densities $%
p_\theta$. One of the distinguishing features of quantum statistics is the
possibility to choose appropriate measurements for specific statistical
problems (e.g. estimation, testing) and the fact that optimal measurements
for different problems may be incompatible with each other. In the
applications section we will discuss specific instances of this phenomenon.

Beside measurements, the quantum model $\mathcal{Q}$ can be transformed into
another quantum model $\mathcal{Q}^\prime := \{ \rho_\theta^\prime :\theta
\in \Theta \}$ on a Hilbert space $\mathcal{H}^\prime$ by means of a \emph{%
quantum randomisation}, i.e. by applying a quantum channel
\begin{eqnarray*}
T&:& \mathcal{T}_1(\mathcal{H}) \to \mathcal{T}_1(\mathcal{H}^\prime) \\
T&:& \rho_\theta \mapsto \rho_\theta^\prime.
\end{eqnarray*}
The model $\mathcal{Q}^\prime$ is less informative than $\mathcal{Q}$ in the
sense that for any measurement $M^\prime$ on $\mathcal{H}^\prime$ one can
construct the measurement $M:= M^\prime\circ T$ on $\mathcal{H}$ such that $%
\mathbb{P}^{M^\prime}_\theta = \mathbb{P}^{M}_\theta$ for all $\theta$. If
there exists another channel $S$ such that $S(\rho_\theta^\prime) =
\rho_\theta$ for all $\theta$ we say (in analogy to the classical case) that
the models $\mathcal{Q}$ and $\mathcal{Q}^\prime$ are \emph{equivalent}; in
particular, for any statistical decision problem, one can match a procedure
for one model with a procedure with the same risk, for the other model. A
closely related concept is that of \emph{quantum sufficiency} whose theory
was developed in \cite{JencovaPetz}. More generally, we define the Le Cam
distance in analogy to the classical case \cite{LeCam}.

\begin{definition}
Let $\mathcal{Q}$ and $\mathcal{Q}^\prime$ be two quantum models over $\Theta
$. The deficiency between $\mathcal{Q}$ and $\mathcal{Q}^\prime$ is defined
by
\begin{equation*}
\delta \left(\mathcal{Q},\mathcal{Q}^\prime\right) := \inf_T
\sup_{\theta\in\Theta} \| T(\rho_\theta) - \rho_\theta^\prime\|_1
\end{equation*}
where the infimum is taken over all channels $T$. The \emph{Le Cam} distance
between $\mathcal{Q}$ and $\mathcal{Q}^\prime$ is defined as
\begin{equation}  \label{eq.q.LeCam}
\Delta \left(\mathcal{Q},\mathcal{Q}^\prime\right):= \mathrm{max}\left(
\delta \left(\mathcal{Q},\mathcal{Q}^\prime\right)\,, \, \delta \left(%
\mathcal{Q}^\prime,\mathcal{Q}\right)\right).
\end{equation}
\end{definition}
Its interpretation is that models which are ``close'' in the Le Cam distance
have similar statistical properties. In practice, this metric is often used
to approximate a sequence of models by another sequence of simpler models,
providing a method to establish asymptotic minimax risks. In particular, the
approximation of i.i.d. quantum statistical models by quantum Gaussian ones
has been investigated in \cite{GutaKahn2006,GutaJencova2007,Kahn&Guta}, in
the case of finite dimensional systems with arbitrary mixed states. Our goal
is to extend these results to non-parametric models consisting of pure
states on infinite dimensional Hilbert spaces. The following lemma will be
used later on.

\begin{lemma}
\label{lemma.distance.channels} Let $\mathcal{Q},\mathcal{Q}^\prime$ be two
quantum models as defined above. Let $\rho_i = \sum_j \mu_{i,j}
\rho_{\theta_{i,j}} $ be two arbitrary mixtures ($i=1,2$) of states in $%
\mathcal{Q}$ and let $\rho^\prime_i = \sum_j \mu_{i,j}
\rho^\prime_{\theta_{i,j}} $ be their counterparts in $\mathcal{Q}^\prime$.
Then
\begin{equation*}
\| \rho^\prime_1 - \rho^\prime_2\|_1 - 2 \Delta(\mathcal{Q},\mathcal{Q}%
^\prime) \leq  \| \rho_1 - \rho_2\|_1 \leq \| \rho^\prime_1 -
\rho^\prime_2\|_1 + 2 \Delta(\mathcal{Q},\mathcal{Q}^\prime) .
\end{equation*}
\end{lemma}
\begin{proof}
Since quantum channels are contractive with respect to the norm-one
\begin{equation*}
\| S (\rho^\prime_1) - S (\rho^\prime_2) \|_1 \leq \|\rho^\prime_1
-\rho^\prime_2\|_1
\end{equation*}
and by the triangle inequality we get
\begin{eqnarray*}
\| \rho_1 - \rho_2\|_1
&\leq & \| \rho_1 - S (\rho^\prime_1) \|_1 + \| S
(\rho^\prime_1) - S (\rho^\prime_2) \|_1+ \|S (\rho^\prime_2) -\rho_2\|_1 \\
&\leq &2 \Delta(\mathcal{Q},\mathcal{Q}^\prime) +\|\rho^\prime_1
-\rho^\prime_2\|_1
\end{eqnarray*}
The second inequality can be shown in a similar way.
\end{proof}


\subsection{The i.i.d. and the quantum white noise models}

\label{sec.q.models} 
We now introduce the non-parametric quantum models investigated in the
paper.
Let $\mathcal{H}$ be an infinite dimensional Hilbert space and let $B:=
\{|e_0 \rangle , |e_1\rangle, \dots \}$ be a fixed orthonormal basis in $%
\mathcal{H}$. The Fourier decomposition of an arbitrary vector is written as
$|\psi\rangle = \sum_{j=0}^\infty \psi_j |e_j\rangle$. Since most of the
models will consist of pure states, we will sometimes define them in terms
of the Hilbert space vectors rather than the density matrices, but keep in
mind that the vectors are uniquely defined only up to a complex phase.

Let us consider the general problem of estimating an unknown pure quantum
state in $\mathcal{H}$. For \emph{finite} dimensional systems, the risk with
respect to typical rotation invariant loss functions scales linearly with
the number of parameters \cite{Gill&Massar}, hence with the dimension of the
space. Therefore, since $\mathcal{H}$ is infinite dimensional, it is not
possible to develop a meaningful estimation theory without any prior
information about the state. Motivated by physical principles and
statistical methodology we introduce the following \emph{Hermite-Sobolev
classes} \cite{BongioanniTorrea} and \cite{Bongioanni}
of pure states characterised by an appropriate decay of the coefficients
with respect to the basis $B$:
\begin{equation}  \label{eq:ClassStates}
S^\alpha(L) := \left\{ |\psi \rangle\langle \psi | :\, \sum_{j=0}^\infty
|\psi_j|^2 j^{2 \alpha} \leq L ,~ \mathrm{and} ~ \|\psi \|=1\right\}, \qquad
\alpha >0,\quad L>0.
\end{equation}
To gain some intuition about the meaning of this class, let us assume that $B
$ is the Fock basis of a one-mode cv system. Then the constraint translates
into the moment condition for the number operator $\langle \psi |
N^{2\alpha} |\psi\rangle \leq L$; this is a mild assumption considering that
all experimentally feasible states have finite moments to all orders. Even
more, the coefficients of typical states such as coherent, squeezed, and
Fock states decay exponentially with the photon number.

Our first model describes $n$ identical copies of a pure state belonging to
the Sobolev class
\begin{equation}  \label{eq.Q.n.model}
\mathcal{Q}_n : = \{ |\psi\rangle \langle \psi |^{\otimes n} \, :\,
|\psi\rangle\langle \psi| \in S^\alpha(L) \}.
\end{equation}
In section \ref{sec.estimation} we show that the minimax rate of $\mathcal{Q}%
_n$ for the norm-one and Bures distance loss functions is $%
n^{-\alpha/(2\alpha+1)}$. This is identical to the minimax rate of the
classical i.i.d. model described in Appendix A.1  \cite{Butucea&Guta&Nussbaum}.%

We now introduce the corresponding quantum Gaussian model. Let $\mathcal{F}%
:= \mathcal{F}(\mathcal{H})$ be the Fock space over $\mathcal{H}$, and let $%
\left \vert G(\sqrt{n}\psi)\right\rangle\in \mathcal{F}$ be the coherent
state with ``displacement'' vector $\sqrt{n}\psi$. As discussed in section %
\ref{sec.Fock.spaces}, the vector $\sqrt{n}\psi$ should be seen now as the
\emph{expectation} of the infinite dimensional Gaussian state rather than a
quantum state in itself, for which reason we have omitted the ket notation.
We define the coherent states model
\begin{equation}  \label{eq.G.n.model}
\mathcal{G}_{n}=\left\{ \left\vert G(\sqrt{n}\psi) \right\rangle
\left\langle G(\sqrt{n}\psi) \right| \, :\, |\psi\rangle \in \mathcal{H},
\mathrm{~such~ that~} |\psi\rangle\langle \psi | \in S^{\alpha}\left(
L\right) \right\} .
\end{equation}
Using the factorisation property \eqref{eq.factorisation.Gaussian} with
respect to the orthonormal basis $B$, we see that the model is equivalent to
the product of independent one-mode coherent Gaussian states of mean $\sqrt{n%
}\psi_i$
\begin{equation*}
\left|G(\sqrt{n}\psi)\right\rangle \cong \bigotimes_{i= 1}^\infty \left|G(%
\sqrt{n}\psi_i)\right\rangle
\end{equation*}
which is analogous to the classical Gaussian sequence model $\mathcal{N}_n$
defined in equation \eqref{eq.gsm}.

Similarly, we can draw an analogy with the white noise model $\mathcal{F}_n$
by realising $\mathcal{H}$ as $L^2([0,1])$. Let us define the \emph{quantum
stochastic process} \cite{Parthasarathy} on $\mathcal{F}(L^2([0,1]))$
\begin{equation*}
B(t) := a\left( \chi_{[0,t]} \right) + a^*\left( \chi_{[0,t]} \right)
\end{equation*}
and note that $[B(t), B(s)]=0$ for all $t,s\in [0,1]$ so that $\{B(t) \, :
\, t \in [0,1] \}$ is a commutative family of operators. This implies that $%
\{B(t) : t\in[0,1]\}$ have a joint probability distribution which is
uniquely determined by the quantum state, and can be regarded as a classical
stochastic process. If the state is the vacuum $|0\rangle$, the process is
Gaussian and has the same distribution as the Brownian motion. Consider now
the process $X(t):= W(\sqrt{n}\psi)^* B(t) W(\sqrt{n}\psi). $ which is
obtained by applying a unitary Weyl transformation to $B(t)$. In physics
terms we work here in the ``Heisenberg picture'' where the transformation
acts on operators while the state is fixed. Using quantum stochastic
calculus one can derive the following differential equation for $X(t)/\sqrt{n%
}$
\begin{equation*}
\frac{1}{\sqrt{n}} dX(t) = \psi(t) dt + \frac{1}{\sqrt{n}}dB(t).
\end{equation*}
Therefore, $X(t)/\sqrt{n}$ is similar to the process \eqref{gwm} with the
exception that it has a complex rather than real valued drift function. Note
that in this correspondence $\psi(t)$ plays the role of $f^{1/2}$, which
agrees with the intuitive interpretation of the wave function as square root
of the state $|\psi\rangle\langle \psi |$. Alternatively, one can use the
Schr\"{o}dinger picture, where the state is $|\sqrt{n}\psi\rangle = W(\sqrt{n%
}\psi) |0\rangle$, such that the process $B(t)$ has the same law as $X(t)$
under the vacuum state.

In section \ref{sec.estimation} we show that the minimax rate of $\mathcal{G}%
_n$ for loss functions based on the norm-one and the Bures distance, is $%
n^{-\alpha/(2\alpha+1)}$. Although the rate is identical to that of the
corresponding classical model, the result does not follow from the classical
case but relies on an explicit measurement strategy for the upper bounds,
and on the quantum local asymptotic equivalence Theorem \ref{th.qlan} for
the lower bound. Furthermore, the minimax rate for the estimation of certain
quadratic functionals are established in section \ref%
{sec.quadratic.functionals}, and the minimax testing rates are derived in
section \ref{sec.testing}. While the former are similar to the classical
ones, the quantum testing rates are parametric as opposed to non-parametric
in the classical case. This reflects the fact that in the quantum case, the
optimal measurements for different statistical problems are in general
incompatible with each other and in some cases they differ significantly
from what is expected on classical basis.

\section{Local asymptotic equivalence for quantum models}

\label{sec.QLAN} 

In this section we prove that the sequence \eqref{eq.Q.n.model} of
non-parametric pure states models is locally asymptotically equivalent (LAE)
with the sequence \eqref{eq.G.n.model} of quantum Gaussian models, in the
sense of the Le Cam distance. This is one of the main results of the paper
and will be subsequently used in the applications. Throughout the section $%
|\psi_0\rangle$ is a fixed but arbitrary state in an infinite dimensional
Hilbert space $\mathcal{H}$. We let $\mathcal{H}_0:= \{ |\psi \rangle \in
\mathcal{H} \, :\, \langle \psi_0 |\psi \rangle =0 \} $ denote the
orthogonal complement of $\mathbb{C}|\psi_0 \rangle$. Any vector state $%
|\psi\rangle \in \mathcal{H}$ decomposes uniquely as
\begin{equation}  \label{eq.psi-u}
|\psi\rangle = |\psi_{u} \rangle:= \sqrt{1- \|u\|^2} |\psi_0\rangle +
|u\rangle, \qquad |u\rangle \in \mathcal{H}_0
\end{equation}
where the phase has been chosen such that the overlap $\langle \psi
|\psi_0\rangle$ is real and positive. Therefore, the pure states are
uniquely parametrised by vectors $|u\rangle\in\mathcal{H}_0$.

Further to the i.i.d. and Gaussian models $\mathcal{Q}_n$ and $\mathcal{G}_n$
defined in \eqref{eq.Q.n.model} and respectively \eqref{eq.G.n.model}, we
now introduce their local counterparts which are parametrised by the local
parameter $|u\rangle$ rather than by $|\psi\rangle$. Let $\gamma_{n}$ be a
sequence such that $\gamma_n=o(1)$, and define the pure state models
\begin{eqnarray}  \label{eq.e.n}
\mathcal{Q}_n(\psi_0,\gamma_n ) &:=& \{ |\psi_{u}^{\otimes n}\rangle \in \mathcal{H}%
^{\otimes n}\, :\, |u\rangle \in \mathcal{H}_0 , \| u\| \leq \gamma_n \} \\
\mathcal{G}_n (\psi_0, \gamma_n) &:=& \{ |G(\sqrt{n} u) \rangle \in \mathcal{F}(%
\mathcal{H}_0) \, :\, |u\rangle \in \mathcal{H}_0 , \| u\| \leq \gamma_n \}.
\label{eq.g.n}
\end{eqnarray}
The LAE Theorem below shows that these local models are asymptotically
equivalent. An interesting fact is that LAE holds without imposing global
restrictions such as defined by the Sobolev classes, rather it suffices that
the local balls shrink at an arbitrary slow rate $\gamma_n = o(1)$.
This contrasts with the classical case where both types of conditions are
needed, as explained in Appendix A.1  \cite{Butucea&Guta&Nussbaum}. However, since
the state cannot be ``localised'' without any prior knowledge, in
applications we need to make additional assumptions which allow us to work
in a small neighbourhood and make use of local asymptotic equivalence. In
particular, the convergence holds for the restricted models where the
Sobolev condition is imposed on top of the local one. This will be used in
establishing the estimation, testing, and functional estimation results.



\begin{theorem}
\label{th.qlan}
Let $\mathcal{Q}_n(\psi_0,\gamma_n) $ and $\mathcal{G}_n(\psi_0, \gamma_n) $
be the models defined in \eqref{eq.e.n} and respectively \eqref{eq.g.n}
where  $\gamma_n=o(1)$.
Then the following convergence holds uniformly over states $|\psi_0\rangle$:
\begin{eqnarray}
&& \underset{n\to\infty}{\lim\sup}\sup_{|\psi_0\rangle\in \mathcal{H}}
\Delta(\mathcal{Q}_n (\psi_0, \gamma_n) , \mathcal{G}_n (\psi_0, \gamma_n) ) =0  \label{eq.LAN1}
\end{eqnarray}
where $\Delta(\cdot, \cdot)$ is the quantum Le Cam distance defined in
equation \eqref{eq.q.LeCam}.
\end{theorem}

The proof is given in \cite{Butucea&Guta&Nussbaum}.

\section{Applications}

\label{sec.applications} 

In this section we discuss three major applications of the local asymptotic
equivalence result in Theorem \ref{th.qlan}, namely to the estimation of
pure states, estimation of a physically meaningful quadratic functional, and
finally to testing between pure states. We stress that local asymptotic
equivalence allows us to translate these problems into similar but easier
ones involving Gaussian states. This strategy has already been successfully
employed \cite{GutaKahn2006} in finding asymptotically optimal estimation
procedures for \emph{finite dimensional mixed states}, which otherwise
appeared to be a difficult problem due to the complexity of the set of
possible measurements.

As discussed in section \ref{sec.q.models}, we will assume that we are given
$n$ independent systems, each prepared in a state $|\psi\rangle \in \mathcal{%
H}$ belonging to the Sobolev ellipsoid $S^\alpha (L)$ defined in equation %
\eqref{eq:ClassStates}. The corresponding quantum statistical model $%
\mathcal{Q}_n$ was defined in equation \eqref{eq.Q.n.model}, and the
Gaussian counterpart model $\mathcal{G}_n$ was defined in equation %
\eqref{eq.G.n.model}.

Here is a summary of the results. In Theorem \ref{thm:lbsobolev} we show
that the estimation rates over such ellipsoids are $n^{- \alpha/(2 \alpha
+1)}$; this is similar to the well-known rates, e.g. for density estimation,
in nonparametric statistics (see~\cite{Tsyb}). The estimation of the
quadratic functional
\begin{equation*}
F(\psi) = \sum_{j \geq 0} |\psi_j|^2 j^{2 \beta}, \mbox{ for some fixed }
\beta >0
\end{equation*}
of the unknown pure state presents two regimes: a parametric rate $n^{-1}$
for the MSE is attained when the unknown state has enough "smoothness" (that
is $\alpha \ge 2 \beta$), whereas a nonparametric rate $n^{-2(1- \beta /
\alpha) }$ is obtained when $\beta < \alpha < 2\beta$. This double regime is
known in nonparametric estimation for the density model, with different
values for both the rates and the values of the parameters where the
phase-transition occurs, cf \cite{Cai-Low-nonquad}, \cite{klem} and references therein.

Parametric rates and sharp asymptotic constants are obtained for the testing problem of
a pure state against an alternative described by the Sobolev-type ellipsoid
with an $L_2$-ball removed. In the classical density model only
nonparametric rates for testing of order $n^{-2 \alpha/(4 \alpha +1)}$ can
be obtained for the $L_2$ norm. In our quantum i.i.d. model, the parametric rate
$n^{-1/2}$ is shown to be minimax for testing $H_0: \psi = \psi_0$, for some
$\psi_0$ in $S^\alpha (L)$ over the nonparametric set of alternatives:
\begin{equation*}
H_1: \psi \in S^\alpha(L) \mbox{ is such that } \| |\psi\rangle\langle \psi | -|\psi_0 \rangle \langle \psi_0 | \|_1 \geq c n^{-1/2} .
\end{equation*}
The sharp asymptotic constant we obtain for testing is specific for ensembles
of pure states.  As we discuss in the sequel, quantum
testing of states allows us to optimize over the measurements, and thus
to obtain the most distinguishable likelihoods for the underlying unknown
quantum state.

\subsection{Estimation} \label{sec.estimation}
We consider the problem of estimating an unknown pure state belonging to the
Hermite-Sobolev class $S^\alpha(L)$ given an ensemble of $n$ independent,
identically prepared systems. The corresponding sequence of statistical
models $\mathcal{Q}_n$ was defined in equation \eqref{eq.Q.n.model}. We
first describe a specific measurement procedure which provides an estimator
whose risk attains the nonparametric rate $n^{-2\alpha/(2 \alpha+1)}$. We
prove the lower bounds for estimating a Gaussian state in the model $%
\mathcal{G}_n$ defined in \eqref{eq.G.n.model}. Subsequently we use LAE to
establish a lower bound showing that the rate is optimal in the i.i.d. model
as well.

Before deriving the bounds we briefly review the definitions of the loss
functions used here and the relations between them, cf. section \ref{sec.metrics}. Recall that the trace norm distance
between states $\rho $ and $\rho^\prime$ is given by $\|\rho -
\rho^\prime\|_1 := \mathrm{Tr}(|\rho - \rho^\prime|)$, and is the quantum
analogue of the norm-one distance between probability densities. The square
of the Bures distance is given by $d_b^2:= 2 (1- \mathrm{Tr}(\sqrt{\sqrt{\rho%
} \rho^\prime \sqrt{\rho}}))$, and is a quantum extension of the Hellinger
distance. These distances satisfy the inequalities \eqref{eq.distances.inequality}.

In the case of pure states (i.e. $\rho= |\psi \rangle \langle\psi |$, and $%
\rho^\prime= |\psi^\prime \rangle \langle\psi^\prime |$) these metrics
become (cf. \eqref{eq.trace.norm.pure} and \eqref{eq.Bures.normone.pure}),
\begin{equation*}
\left\|\rho -\rho^\prime \right\|_1 = 2 \sqrt{1 - |\langle\psi| \psi^\prime
\rangle|^2} , \qquad d_b^2(\rho \, , \,\rho^\prime ) = 2(1 -|\langle\psi|
\psi^\prime \rangle| ).
\end{equation*}
Since vectors are not uniquely defined by the states, the distances cannot
be expressed directly in terms of the length $\|\psi- \psi^\prime\|$.
However if we consider a reference vector $|\psi_0\rangle$ and define the
representative vector $|\psi\rangle$ such that $\langle \psi_0 |\psi\rangle
\geq 0$, then we can write (as in section \ref{sec.QLAN})
\begin{equation*}
|\psi_u\rangle = \sqrt{1 -\|u\|^2} |\psi_0 \rangle + | u\rangle , \quad
|\psi_{u^\prime} \rangle = \sqrt{1 -\|u^\prime\|^2} |\psi_0 \rangle + |
u^\prime\rangle , \quad |u\rangle,|u^\prime\rangle\perp |\psi_0\rangle
\end{equation*}
and the distances have the same (up to a constant) quadratic approximation
\begin{eqnarray}  \label{eq.lcoap.approx.bures.normone}
\| \rho_u - \rho_{u^\prime}\|^2_1
& =& 4\|u- u^\prime\|^2 + O( \mathrm{max}
(\|u\|, \|u^\prime \|)^4 ),\nonumber \\
 d_b^2(\rho_u \, , \,\rho_{u^\prime} )
&=& \|u- u^\prime\|^2 + O( \mathrm{max} (\|u\|, \|u^\prime \|)^4 ),
\end{eqnarray}
where the correction terms are of order 4 as $ \|u\|$ and $\|u'\|$ tend to $0$.
Below we show that asymptotically with $n$ the estimation risk for norm-one
square and Bures distance square will have the same rate as that of
estimating the local parameter $u$ with respect to the Hilbert space
distance. 

\subsubsection{Upper bounds}

\label{sec.upper.bound.estimation}

We first describe a two steps measurement procedure, which provides an
estimator whose risk has rate $n^{-2\alpha/(2\alpha+1)}$.

\begin{theorem}
\label{thm:UBestim} Consider the i.i.d. quantum model $\mathcal{Q}_n$ given
by equation \eqref{eq.Q.n.model}. There exists an estimator
$\widehat \rho_n:= |\widehat\psi_n\rangle\langle\widehat\psi_n | $ such that
\begin{equation*}
\limsup_{n \to \infty} \sup_{|\psi\rangle \in S^\alpha(L)} n^{2\alpha/(2\alpha+1)}
\mathbb{E}_\rho \left[ d^2 (\hat\rho_n , \rho )\right] \leq C ,
\end{equation*}
where $\rho:= |\psi\rangle \langle \psi|$,  $d(\hat\rho_n , \rho )$ denotes either the trace-norm distance, or the Bures distance, and  $C>0$ is a constant depending only on $\alpha > 0$ and $L>0$.
\end{theorem}

The proof is given in \cite{Butucea&Guta&Nussbaum}.


\subsubsection{Lower bounds - Unimprovable rates}

We will first consider the Gaussian model $\mathcal{G}_n$ given by equation %
\eqref{eq.G.n.model} which is indexed by Hilbert space vectors $\psi\in
\mathcal{H}$ in the Sobolev class $S^\alpha(L)$, playing the role of means
of quantum Gaussian states $|G(\sqrt{n}\psi )\rangle$. In Theorem \ref%
{thm:lbsobolev} we find a lower bound for the mean square error of any
estimator $\hat{\psi}$. This is then used in conjunction with the local
asymptotic equivalence Theorem \ref{th.qlan} to obtain a lower bound for the
risk of the i.i.d. model $\mathcal{Q}_n$, with respect to the norm-one and
Bures distances.


\begin{theorem}
\label{thm:lbsobolev} Consider the quantum Gaussian model $\mathcal{G}_n$
given by equation \eqref{eq.G.n.model}. There exists some constant $c>0$
depending only on $\alpha$ and $L$ such that
\begin{equation*}
\liminf_{n \to \infty} \inf _{\widehat \psi_n} \sup_{\psi \in S^\alpha(L)}
n^{2\alpha/(2\alpha+1)} \mathbb{E}_\psi \left[ \| \widehat \psi_n - \psi
\|_2^2 \right] \geq c ,
\end{equation*}
where the infimum is taken over all estimators $\widehat \psi_n$, understood
as combination of measurements and classical estimators.
\end{theorem}
The proof is given in \cite{Butucea&Guta&Nussbaum}.


We now proceed to consider the i.i.d. model $\mathcal{Q}_n$ defined in %
\eqref{eq.Q.n.model}. We are given $n$ copies of an unknown pure state $%
|\psi\rangle\langle \psi|$, with $\psi$ in the Sobolev class $\mathcal{S}%
^\alpha(L)$. The goal is to find an asymptotic lower bound for the
estimation risk (with respect to the Bures or norm-one loss functions) which
matches the upper bound derived in section \ref{sec.upper.bound.estimation}.
Since both loss functions satisfy the triangle inequality, it can be shown
that by choosing estimators which are mixed states, rather than pure states,
one can improve the risk by at most a constant factor 2. Therefore we
consider estimators which are pure states. In order to fix the phase of the
vector representing the true and the estimated state, we will assume that $%
\langle \psi | e_0\rangle \geq 0$ and $\langle \hat{\psi} | e_0\rangle \geq 0
$. 

\begin{theorem}
\label{thm:lbsobolev.iid} Consider the i.i.d. quantum model $\mathcal{Q}_n$
given by equation \eqref{eq.Q.n.model}. There exists some constant $c>0$
depending only on $\alpha > 0$ and $L>0$ such that
\begin{equation*}
\liminf_{n \to \infty} \inf _{|\widehat \psi_n\rangle}\, \sup_{|\psi\rangle \in S^\alpha(L)}
n^{2\alpha/(2\alpha+1)} \mathbb{E}_\rho \left[ d^2 (\hat\rho_n , \rho )\right]
\geq c ,
\end{equation*}
where $\rho:= |\psi\rangle\langle \psi|$, the infimum is taken over all estimators
$\widehat \rho_n:= |\widehat \psi_n\rangle \langle \widehat \psi_n|$ (defined by a combination of measurement and a classical estimator), and the loss
function $d (\hat\rho , \rho )$ is either the norm-one or the Bures
distance.
\end{theorem}
The proof is given in \cite{Butucea&Guta&Nussbaum}.


\subsection{Quadratic functionals}

\label{sec.quadratic.functionals} 

This section deals with the estimation of the quadratic functional
\begin{equation*}
F(\psi) = \sum_{j \geq 0} |\psi_j|^2 \cdot j^{2 \beta}, \mbox{ for some
fixed } 0 < \beta < \alpha,
\end{equation*}
which is well defined for all pure states $|\psi\rangle $ in the ellipsoid $%
S^\alpha(L)$. If the Hilbert space $\mathcal{H}$ is represented as $L_2(
\mathbb{R})$ and $\{| j\rangle : j\geq 0\}$ is the Fock basis (cf. section %
\ref{sec.one.mode}) then $F(\psi)$ is the moment of order $2 \beta$ of the
number operator $N$:
\begin{equation*}
F(\psi) = \mathrm{Tr}(|\psi\rangle \langle \psi| \cdot N^{2
\beta}). 
\end{equation*}
Below we derive upper and lower bounds for the rate of the quadratic risk
for estimating $F(\psi)$, which is of order $n^{-1}$ if $\alpha \geq 2 \beta$,
and $n^{-2(1- \beta /\alpha) }$ if $\beta < \alpha <2 \beta$.


\subsubsection{Upper bounds}


Let us describe an estimator $\widehat F_n$ of $F(\psi)$ in the quantum
i.i.d. model. We consider the measurement of the number operator with
projections $\{|j\rangle \langle j|\}_{j \geq 0}$. For a pure state $%
|\psi\rangle = \sum_{j \geq 0} \psi_j |j\rangle $, we obtain an outcome $X$
taking values $j \in \mathbb{N}$ with probabilities $p_j:= \mathbb{P}_\psi
(X = j) = |\psi_j|^2, $ for $j \geq 0$. By measuring each quantum sample $%
|\psi\rangle$ separately, we obtain i.i.d. copies $X_1,\dots , X_n$ of $X$,
allowing us to estimate each $p_j$ empirically, by
\begin{equation*}
\hat p_j = \frac 1n \sum_{k=1}^n I(X_k=j), \quad j \geq 0.
\end{equation*}
which is an unbiased estimator of $p_j$ with variance $p_j(1-p_j)/n$. The
estimator of the quadratic functional is defined as
\begin{equation}  \label{Festim}
\widehat F_n = \sum_{j=1}^N \hat p_j \cdot j^{2 \beta}
\end{equation}
for an appropriately chosen truncation parameter $N$ defined below. The next
theorem, shows that a parametric rate can be attained for estimating the
quadratic functional $F(\psi)$ if $\alpha \geq 2 \beta$, whereas a
nonparametric rate is attained if $\beta < \alpha <2\beta$. 

\begin{theorem}
\label{thm:Festim} Consider the i.i.d. quantum model $\mathcal{Q}_n$ given
by equation \eqref{eq.Q.n.model}.
Let $\widehat F_n$ be the estimator \eqref{Festim} of $F(\psi)$ with $N \asymp
n^{1/4(\alpha-\beta)} $, for $\alpha \geq 2 \beta$, respectively $N \asymp
n^{1/2 \alpha }$, for $\beta < \alpha < 2 \beta$. Then
\begin{eqnarray}  \label{rate:eta}
\sup_{\psi \in S^\alpha(L)} \mathbb{E}_\psi \left( \hat F_n -
F(\psi)\right)^2
&=& O \left(\eta_n^2 \right) \nonumber \\
\mbox{ where } \eta_n^2
&=& \left\{
\begin{array}{ll}
n^{-1}, & \mbox{ if } \alpha \geq 2 \beta \\
n^{-2(1- \beta /\alpha)}, & \mbox{ if } \beta < \alpha <2 \beta.%
\end{array}
\right.
\end{eqnarray}
\end{theorem}

The proof is given in \cite{Butucea&Guta&Nussbaum}.


\subsubsection{Lower bounds}

The next Theorem proves the optimality of the previously attained rate for the estimation of quadratic functionals.

\begin{theorem}
\label{thm:Flowbounds} Consider the i.i.d. quantum model $\mathcal{Q}_n$
given by equation \eqref{eq.Q.n.model}. Then, there exists some constant $c>0
$ depending only on $\alpha$, $\beta$ (with $\alpha> \beta >0$), and $L>0$ such that
\begin{equation*}
\liminf_{n \to \infty} \inf_{\widehat F_n} \sup_{\psi \in S^\alpha(L)}
\eta_n^{-2} \cdot \mathbb{E}_\psi \left( \widehat F_n - F(\psi) \right)^2 \geq
c,
\end{equation*}
where the infimum is taken over all measurements and resulting estimators $%
\widehat F_n$ of $F(\psi)$.
\end{theorem}

 Further discussion on quadratic functionals can be found in Appendix A.2  \cite{Butucea&Guta&Nussbaum}; proofs are presented in Appendix B.


\subsection{Testing}

\label{sec.testing} 

In the problem of testing for signal in classical Gaussian white noise, over a
smoothness class with an $L_{2}$-ball removed, minimax rates of convergences
(separation rates) are well known \cite{IngsterSuslina}; they are expressed in
the rate of the ball radius tending to zero along with noise intensity, such
that a nontrivial asymptotic power is possible. We will consider an analogous
testing problem here for pure states. Accordingly, let $\rho=|\psi
\rangle\langle\psi|$ denote pure states, let $\rho_{0}=|\psi_{0}\rangle
\langle\psi_{0}|$ be a fixed pure state to serve as the null hypothesis, and
let
\begin{equation}%
B\left(  \varphi\right)  =\left\{  \Vert\rho-\rho_{0}\Vert_{1}\geq
\varphi\right\}
\label{distance-cond}%
\end{equation}
be the complement of a trace norm ball around $\rho_{0}$. We want to test in
the i.i.d. quantum model $\mathcal{Q}_{n}$ given by equation
\eqref{eq.Q.n.model} the following hypotheses about a pure state $\rho$ :
\begin{equation}%
\begin{array}
[c]{ll}%
H_{0}: & \rho=\rho_{0}\\
H_{1}(\varphi_{n}): & \rho\in S^{\alpha}\left(  L\right)  \cap B\left(
\varphi_{n}\right)  \text{ }%
\end{array}
\label{test-a}%
\end{equation}
for $\{\varphi_{n}\}_{n\geq1}$ a decreasing sequence of positive real numbers.
Consider a binary POVM $M=(M_{0},M_{1})$, acting on the
product states $\rho^{\otimes n}$, cf. Definition  \ref{def.POVM}. We denote
the testing risk between two fixed hypotheses by the sum of the two error
probabilities
\[
R_{n}^{T}(M)=R_{n}^{T}(\rho_{0}^{\otimes n},\rho^{\otimes n},M)=\mathrm{Tr}%
(\rho_{0}^{\otimes n}\cdot M_{1})+\mathrm{Tr}(\rho^{\otimes n}\cdot M_{0}).
\]
In the minimax $\alpha$-testing approach which dominates the literature on the
classical Gaussian white noise case, one would require $\mathrm{Tr}(\rho
_{0}^{\otimes n}\cdot M_{1})\leq\alpha$ while trying to minimize the worst
case type 2 error $\sup_{\rho\in S^{\alpha}\left(  L\right)  \cap B\left(
\varphi_{n}\right)  }\mathrm{Tr}(\rho^{\otimes n}\cdot M_{0})$. However we
will consider here the so-called detection problem \cite{Ingst-Stepan} where
the target is the worst case total error probability
\begin{eqnarray*}
\mathbb{P}_{e}^{M}\left(  \varphi_{n}\right)
&=&
\sup_{\rho\in S^{\alpha}\left(
L\right)  \cap B\left(  \varphi_{n}\right)  }R_{n}^{T}(\rho_{0}^{\otimes
n},\rho^{\otimes n},M) \\
&=&
\mathrm{Tr}(\rho_{0}^{\otimes n}\cdot M_{1})+\sup
_{\rho\in S^{\alpha}\left(  L\right)  \cap B\left(  \varphi_{n}\right)
}\mathrm{Tr}(\rho^{\otimes n}\cdot M_{0}).
\end{eqnarray*}
The minimax total error probability is then obtained by optimizing over $T:$
\[
\mathbb{P}_{e}^{\ast}\left(  \varphi_{n}\right)  =\inf_{M \text{ binary POVM}%
}\mathbb{P}_{e}^{M}\left(  \varphi_{n}\right)  .
\]

\subsubsection{Separation rate}

A sequence $\{\varphi_{n}^{\ast}\}_{n\geq1}$ is called a \textit{minimax
separation rate} if any other sequence $\{\varphi_{n}\}_{n\geq1}$ fulfills
\begin{equation}
\mathbb{P}_{e}^{\ast}\left(  \varphi_{n}\right)  \rightarrow1\text{ if
}\varphi_{n}/\varphi_{n}^{\ast}\rightarrow0\text{ and }\mathbb{P}_{e}^{\ast
}\left(  \varphi_{n}\right)  \rightarrow0\text{ if }\varphi_{n}/\varphi
_{n}^{\ast}\rightarrow\infty. \label{sep-rate-def}%
\end{equation}
Below we establish that $\varphi_{n}^{\ast}=n^{-1/2}$ is a separation rate in
the current problem, even though the alternative $H_{1}(\cdot)$ in
(\ref{test-a}) is a nonparametric set of pure states. Recall relations
(\ref{HHbound}), (\ref{eq.trace.norm.pure}) describing the total optimal error
for testing between simple hypotheses given by two pure states.

\begin{theorem}
\label{thm:Test-1} Consider the i.i.d. quantum model $\mathcal{Q}_{n}$ given
by equation \eqref{eq.Q.n.model}, and the testing problem (\ref{test-a}).
Assume that $\rho_{0}$ is in the interior of $S^{\alpha}\left(  L\right)  $,
i.e $\rho_{0}\in S^{\alpha}\left(  L^{\prime}\right)  $ for some $L^{\prime
}<L$. Then $\varphi_{n}^{\ast}=n^{-1/2}$ is a \textit{minimax }separation rate.
\end{theorem}
The proof is given in \cite{Butucea&Guta&Nussbaum}.

\subsubsection{Sharp asymptotics}

Having identified the optimal rate of convergence in the testing
problem, 
we will go a step further and aim
at a sharp asymptotics for the minimax testing error. We will adopt the
approach of \cite{Ermakov}, extended in \cite{IngsterSuslina}, where testing
analogs of the Pinsker-type sharp risk asymptotics in nonparametric estimation
were obtained. The result will be framed as follows: if the radius is chosen
$\varphi_{n}\sim cn^{-1/2}$ for a certain $c>0$, then the minimax testing
error behaves as $\mathbb{P}_{e}^{\ast}\left(  \varphi_{n}\right)  \sim
\exp\left(  -c^{2}/4\right)  $. Thus the sharp asymptotics is expressed as a
type of scaling result: a choice of constant $c$ in the radius implies a
certain minimax error asymptotics depending on $c$.

To outline the problem, consider the upper and lower error bounds obtained in
the proof of the separation rate, i.e. the proof of Theorem \ref{thm:Test-1} in \cite{Butucea&Guta&Nussbaum}. The upper risk bound obtained is
\begin{equation}
\mathbb{P}_{e}^{M_{n}}\left(  \varphi_{n}\right)  \leq\exp\left(  -c_{n}%
^{2}/4\right)  \label{upper-bound-test-2}%
\end{equation}
if $\varphi_{n}=c_{n}n^{-1/2}$, where $M_{n}$ is the sequence of projection
tests \newline
$M_{n}=(\rho_{0}^{\otimes n},I-\rho_{0}^{\otimes n})$.
The corresponding lower risk
bound  is
\[
\inf_{M \text{ binary POVM}}\mathbb{P}_{e}^{M}\left(  \varphi_{n}\right)
\geq1-\sqrt{1-\left(  1-c_{n}^{2}n^{-1}/4\right)  ^{n}}.
\]
If $c_{n}=c$ we can summarize this as
\[
1-\sqrt{1-\exp\left(  -c^{2}/4\right)  }+o\left(  1\right)  \leq\mathbb{P}%
_{e}^{\ast}\left(  \varphi_{n}\right)  \leq\exp\left(  -c^{2}/4\right)  .
\]
Our result will be that the upper bound is sharp and represents the minimax
risk asymptotics.
\begin{theorem}
\label{thm:Test-2} Consider the i.i.d. quantum model $\mathcal{Q}_{n}$ given
by equation \eqref{eq.Q.n.model}, and the testing problem (\ref{test-a}).
Assume that $\rho_{0}\in S^{\alpha}\left(  L^{\prime}\right)  $ for some
$L^{\prime}<L$. At the minimax separation rate for the radius, i.e. for
$\varphi_{n}\asymp n^{-1/2}$ we have
\[
\lim_{n}\;\;n^{-1}\varphi_{n}^{-2}\;\log\mathbb{P}_{e}^{\ast}\left(  \varphi
_{n}\right)  =-1/4.
\]
\end{theorem}

 Further discussion on nonparametric testing can be found in Appendix A.3  \cite{Butucea&Guta&Nussbaum}; proofs are presented in Appendix B.


\subsection{Discussion: state estimation}
\
\smallskip
\

\textit{Tomography and optimal rates. } Consider a model where the Sobolev-type assumption $\rho\in S^{\alpha}\left(
L\right)  $ about the pure state $\rho=\left\vert \psi\right\rangle
\left\langle \psi\right\vert $ (cf. (\ref{eq:ClassStates})) is replaced by a
finite dimensionality assumption: $\rho\in\mathcal{H}_{d}$ where
\[
\mathcal{H}_{d}=\left\{  \left\vert \psi\right\rangle \left\langle
\psi\right\vert :\psi_{j}=0\text{, }j\geq d\right\}
\]
and $d$ is known. One observes $n$ identical copies of the pure state
$\rho=\left\vert \psi\right\rangle \left\langle \psi\right\vert $, with
possibly $d=d_{n}\rightarrow\infty$, i. e. the model $\mathcal{Q}_{n}$ of
\eqref{eq.Q.n.model} is replaced by
\[
Q_{n}:=\left\{  \rho^{\otimes n}:\rho\in\mathcal{H}_{d}\right\}  .
\]
Since $\mathcal{H}_{d}$ can be written $\mathcal{H}_{d}=\mathcal{S}_{1,d}$
where%
\[
\mathcal{S}_{r,d}:=\left\{  \rho:\left\langle e_{i}|\rho|e_{j}\right\rangle
=0,i,j\geq d,\mathrm{rank}(\rho)=r\right\}  ,
\]
the model is effectively a special case of the $d\times d$ density matrices of
$\mathrm{rank}(\rho)=r$ considered in \cite{Kolt-Xia-MLearn}. In
\cite{Kolt-Xia-MLearn} however, it is not known in advance that $r=1$ but
$\rho$ is a density matrix of possibly low rank $r$, and the aim is estimation
of $\rho$ using quantum state tomography performed on $n$ identical copies of
$\rho$. Data are obtained by defining an observable $\otimes_{i=1}^{n}E_{i}$
where $E_{1},\ldots,E_{n}$ are i.i.d. uniformly selected elements of the Pauli
basis of the linear space of $d\times d$ Hermitian matrices, and applying the
corresponding measurement to $\rho^{\otimes n}$. Let $\hat{\rho}_{n}^{\ast}$
denote an arbitrary estimator of $\rho$ based on that measurement. A lower
asymptotic risk bound for norm-one risk is established; in the special case
$d^{2}r^{2}=o\left(  n\right)  $ it reads as
\begin{equation}
\inf_{\hat{\rho}_{n}^{\ast}}\sup_{\rho\in\mathcal{S}_{r,d}}\mathbb{E}_{\rho
}\left[  \left\Vert \hat{\rho}_{n}^{\ast}-\rho\right\Vert _{1}^{2}\right]
\geq c \frac{r^{2}d^{2}}{n}\label{KX-bound}%
\end{equation}
for some $c>0$ (Theorem 10 in \cite{Kolt-Xia-MLearn}). It is also shown in
\cite{Kolt-Xia-MLearn}\textit{ }that (\ref{KX-bound}) is attained, up to a
different constant and logarithmic terms, by an entropy penalized least
squares type estimator based on measurement of $\otimes_{i=1}^{n}E_{i}$, even
when the rank $r$ is unknown. Analogous optimal rates for $d\times d$ mixed
states $\rho$ with Pauli measurements, but under sparsity assumptions on  the entries of the matrix
$\rho$ have been obtained in \cite{CaiWangZhou-tomo}.
\newline Returning to our setting of pure states, where
$r=1$ is known, with an infimum over \textit{all measurements of}
$\rho^{\otimes n}$ and corresponding estimators $\hat{\rho}_{n}$, according to
\cite{Hayashi}\textbf{ }one has
\begin{equation}
\inf_{\hat{\rho}_{n}}\sup_{\rho\in\mathcal{S}_{1,d}}\mathbb{E}_{\rho}\left[
\left\Vert \hat{\rho}_{n}-\rho\right\Vert _{1}^{2}\right]  = \frac{4\left(
d-1\right)  }{d+n}\label{Hay-bound}%
\end{equation}
and the bound is attained by an estimator of the pure state $\rho$ based on
the covariant measurement, cf. equation (B.8) \cite{Butucea&Guta&Nussbaum}.
Comparing (\ref{KX-bound}) for $r=1$ and $d_{n}\rightarrow\infty$,
$d_{n}=o\left(  n\right)  $ with (\ref{Hay-bound}), we find that the latter
bound is of order $d_{n}/n$ whereas the former is of order $d_{n}^{2}/n$. It
means that for estimation of finite dimensional pure states, estimators based
on the Pauli type measurement $\otimes_{i=1}^{n}E_{i}$ do not attain the
optimal rate when $d_{n}\rightarrow\infty$. It may be conjectured that the
same holds for the optimal rate over $\rho\in S^{\alpha}\left(  L\right)  $,
i.e. our rate of Theorem \ref{thm:UBestim}. We emphasize again that our
results establish lower asymptotic risk bounds over all quantum measurements
and estimators, whereas lower risk bounds within one specific measurement
scheme \cite{Koltchinskii-von-Neumann} \cite{Kolt-Xia-MLearn}\textit{
}\cite{CaiWangZhou-tomo} are essentially results of non-quantum classical
statistics.

\smallskip

\textit{Separate measurements.} A notable fact is also that $\otimes_{i=1}^{n}E_{i}$ is a \textit{separate
}(or local) measurement, i.e. produces independent random variables (or random
elements) $Y_{1},\ldots,Y_{n}$ each based on a measurement of a copy of $\rho
$, whereas the covariant measurement (cp.
equation (B.8) \cite{Butucea&Guta&Nussbaum})
we used for attainment our risk bound of Theorem \ref{thm:UBestim}
is of \textit{collective }(or joint)\textit{ }type with regard to the product
$\rho^{\otimes n}$. Separate measurements are of interest from a practical
point of view since collective measurements of large quantum systems may be
unfeasible in implementations \cite{Nair-realiz}. \textbf{ }In
\cite{Bagan-etal-pure} it is shown that for fixed $d=2$, the bound
(\ref{Hay-bound}) can be attained asymptotically as $n\rightarrow\infty$ (up
to a factor $1+o\left(  1\right)  $) by a separate measurement of
$\rho^{\otimes n}$; it is an open question whether in our infinite dimensional
setting, the optimal rate of Theorem \ref{thm:UBestim} can be attained by a
separate measurement. For mixed qubits ($d=2$), an asymptotic efficiency gap
between separate and collective measurements is known to exist
\cite{Bagan-etal-sepa}.



\bigskip

\section*{Supplement to "Local asymptotic equivalence of pure states ensembles and quantum Gaussian white noise"}

\medskip

A more detailed over\-view of asymptotic equivalence for classical models is provided in Appendix A.1. The results on quadratic functionals and nonparametric testing are further discussed   in Appendix A.2 and A.3. Proofs of all results are given  in Appendix B.

\section{Further Discussion} \label{sec:furtherdisc}

\subsection{Classical models}

\label{sec.classical.models} 
Here we review several asymptotic normality results for classical models
which are analogous to the quantum models investigated in the paper.

A \emph{classical} statistical model is defined as a family of probability
distributions $\mathcal{Q} = \{\mathbb{P}_f :\, f\in \mathcal{W}\}$ on a
measurable space $(\mathcal{X}, \mathcal{A})$, indexed by an unknown,
possibly infinite dimensional parameter $f$ to be estimated, which belongs
to a parameter space $\mathcal{W}$.
In the asymptotic framework considered here we assume that we are given a
(large) number $n$ of independent, identically distributed samples $X_1,
\dots , X_n$ from $\mathbb{P}_f$, from which we would like to estimate $f$.
If $d :\mathcal{W}\times \mathcal{W}\to \mathbb{R}_+$ is a chosen loss
function, then the risk of an estimator $\hat{f}_n = \hat{f}_n(X_1, \dots ,
X_n)$ is
\begin{equation*}
R(\hat{f}_n, f) = \mathbb{E}_f \left[ d(\hat{f}_n, f)^2\right].
\end{equation*}

In nonparametric statistics, the parameter of the model $f$ is often a
function that belongs to a smoothness class.
We consider two classes $\mathcal{W}$: the periodic Sobolev class
$\mathcal{S}^{\alpha}(L)$ of functions on $[0,1]$ with smoothness $\alpha
>1/2$, and the H\"{o}lder class $\Lambda^{\alpha}(L)$, with smoothness
$\alpha>0$. For any $f\in\mathbb{L}_{2}[0,1]$, let $\left\{  f_{j}%
,j\in\mathbb{Z}\right\}  $ be the set of Fourier coefficients with respect to
the standard trigonometric basis.   The classes are defined as
$$
\mathcal{S}^{\alpha}(L):=\left\{  f:[0,1]\rightarrow\mathbb{R}\,:\sum
_{j\in\mathbb{Z}}\,\int|f_{j}|^{2}|j|^{2\alpha}du\leq L\right\}  .
$$
and
$$
\Lambda^{\alpha}(L):=\left\{  f:[0,1]\rightarrow\mathbb{R}\,:\,|f(x)-f(y)|\leq
L|x-y|^{\alpha},\,x,\,y\in\lbrack0,1]\right\}  .
$$
In addition, when densities $f$ are considered, we will assume that
$\mathcal{W}$  includes an additional restriction to a class%
$$
\mathcal{D}_{\varepsilon}=\left\{  f:[0,1]\rightarrow\mathbb{[\varepsilon
},\infty\mathbb{)}:\int_{[0,1]}f(x)dx=1\right\}
$$
for some $\varepsilon>0$.

\textit{Density model.} The classical density model consists of $n$
observations $X_{1},\ldots,X_{n}$ which are independent, identically
distributed (i.i.d.) with common probability density $f$
\[
\mathcal{P}_{n}=\left\{  \mathbb{P}_{f}^{\otimes n}\,:\,f\in\mathcal{W}%
\right\}  .
\]
\textit{Gaussian regression model with fixed equidistant design. } In this
model, we observe $Y_{1},...,Y_{n}$ such that
\[
Y_{i}=f^{1/2}\left(  \frac{i}{n}\right)  +\xi_{i},\quad i=1,...,n,
\]
where the errors $\xi_{1},...,\xi_{n}$ are i.i.d., standard Gaussian
variables. Denote the Gaussian regression model by
\[
\mathcal{R}_{n}=\left\{  \bigotimes_{i=1}^{n}\mathcal{N}\left(  f^{1/2}\left(
\frac{1}{n}\right)  ,1\right)  :f\in\mathcal{W}\right\}  .
\]
\textit{Gaussian white noise model.} In this model the square-root density
$f^{1/2}$ is observed with Gaussian white noise of variance $n^{-1}$, i.e.
\begin{equation}
dY_{t}=f^{1/2}(t)dt+\frac{1}{\sqrt{n}}dW_{t},\quad t\in\lbrack0,1].\label{gwm-supp}%
\end{equation}
If we denote by $\mathbb{Q}_{f}$ the probability distribution of
$\{Y(t)\,:\,t\in\lbrack0,1]\}$, the corresponding model is
\[
\mathcal{F}_{n}:=\left\{  \mathbb{Q}_{f}\,:\,f\in\mathcal{W}\right\}  .
\]
\textit{Gaussian sequence model.} In this model we observe a sequence of
Gaussian random variables with means equal to the coefficients of $f^{1/2}$ in
some orthonormal basis of $\mathbb{L}_{2}[0,1]$ for $f\in\mathcal{F}$
\begin{equation}
y_{j}=\theta_{j}(f^{1/2})+\frac{1}{\sqrt{n}}\xi_{j},\qquad i=1,2,\ldots
\label{eq.gsm-supp}%
\end{equation}
where $\{\xi_{i}\}_{i\geq1}$ are Gaussian i.i.d. random variables. We denote
this model
\[
\mathcal{N}_{n}=\left\{  \bigotimes_{j\geq1}\mathcal{N}\left(  \theta
_{j}\left(  f^{1/2}\right)  ,\frac{1}{n}\right)  :f\in\mathcal{W}\right\}  .
\]

In \cite{Nussbaum96} it was shown that the sequences of models $\mathcal{P}%
_{n}$ and $\mathcal{F}_{n}$ are asymptotically equivalent in the sense that
their Le Cam distance converges to zero as $n\rightarrow\infty$ when
$\mathcal{W}=\Lambda^{\alpha}(L)\cap\mathcal{D}_{\varepsilon}$  with $\alpha>1/2$; in \cite{BrownLow}, a similar result was
established for $\mathcal{R}_{n}$ and $\mathcal{F}_{n}$ (more precisely, with
$f^{1/2}$ any real valued function  $f^{1/2}\in\Lambda^{\alpha}(L)$). Later,
\cite{Rohde} showed that models $\mathcal{F}_{n}$ and $\mathcal{N}_{n}$ are
asymptotically equivalent over periodic Sobolev classes $f^{1/2}\in
\mathcal{S}^{\alpha}(L)$ with smoothness $\alpha>1/2$. Among many other
results \cite{GramaNussbaum} considered generalized linear models,
\cite{BCLZ2002} regression models with random design and \cite{Reiss}
multivariate and random design, \cite{GolubevNussbaumZhou} compared the
stationary Gaussian process with the Gaussian white noise model $\mathcal{F}%
_{n}$. In \cite{RaySchmidt} sharp rates of convergence are obtained for the
equivalence of $\mathcal{P}_{n}$ and $\mathcal{F}_{n}$, including also Poisson process models.

In all classical results, the underlying nonparametric function was assumed
to belong to a smoothness class in order to establish asymptotic equivalence
of models. In the quantum setup of pure states and Gaussian states that we
discuss in Section \ref{sec.QLAN}, no such smoothness assumption is needed.

\subsection{Quadratic Functionals\\}

\textit{The elbow phenomenon.} The change of regime which occurs in the
optimal MSE rate $\eta_{n}^{2}$ in \eqref{rate:eta} has been
described as the elbow phenomenon in the literature \cite{Cai-Low-nonquad}. In
the classical Gaussian sequence model, it takes the following shape. Consider
observations introduced in (\ref{eq.gsm}):
\begin{equation*}
y_{j}=\vartheta_{j}+n^{-1/2}\xi_{j}\text{, }j=1,2,\ldots,
\end{equation*}
where $\left\{  \xi_{j}\right\}  $ are i.i.d. standard normal, and the
parameter $\vartheta=\left(  \vartheta_{j}\right)  _{j=1}^{\infty}$ satisfies
a restriction $\sum_{j=1}^{\infty}j^{2\alpha}\vartheta_{j}^{2}\leq L$ for some
$\alpha>0$. For estimation of the quadratic functional $\tilde{F}\left(
\vartheta\right)  =\sum_{j=1}^{\infty}j^{2\beta}\vartheta_{j}^{2}$ with
$\beta<\alpha$, the minimax MSE\ rate of convergence is
\[
\tilde{\eta}_{n}^{2}=\left\{
\begin{tabular}
[c]{l}%
$n^{-1}$ if $\alpha\geq2\beta+1/4$\\
$n^{-2\frac{4\left(  \alpha-\beta\right)  }{4\alpha+1}}$ if $\beta
<\alpha<2\beta+1/4$%
\end{tabular}
\ \ \right.  =n^{-2\tilde{r}}\text{ for }\tilde{r}=\min\left(  \frac{1}%
{2},\frac{4\left(  \alpha-\beta\right)  }{4\alpha+1}\right)
\]
(cf \cite{klem} and references cited therein). The same rate holds for
estimation of the squared $L_{2}$-norm of the $\beta$-th derivative of a
density in an $\alpha$-H\"{o}lder class, cf. \cite{Bick-Rit-88}. Comparing
with our rate $\eta_{n}^{2}$ in \eqref{rate:eta} which can be
written $\eta_{n}^{2}=n^{-2r}$ for $r=\min\left(  \frac{1}{2},\frac{4\left(
\alpha-\beta\right)  }{4\alpha}\right)  $, we see that both rates exhibit the
elbow phenomenon, but at different critical values for $(\alpha,\beta)$, and
the rate for the quantum case is slightly faster in the region $\alpha
<2\beta+1/4$.

\smallskip

\textit{A tail functional of a discrete distribution.} Our method of proof for
the optimal rate $\eta_{n}^{2}=n^{-2r}$ shows that it is also the optimal rate
in the following non-quantum problem: suppose $P=\left\{  p_{j}\right\}
_{j=0}^{\infty}$ is a probability measure on the nonnegative integers,
satisfying a restriction $\sum_{j=0}^{\infty}j^{2\alpha}p_{j}\leq L$, and the
aim is to estimate the linear functional $F_{0}\left(  P\right)  =\sum
_{j=0}^{\infty}j^{2\beta}p_{j}$
on the basis of $n$ i.i.d. observations $X_{1},\ldots,X_{n}$
having law $P$. Indeed, Theorem 5.4 shows that the estimator $\hat{F}_{n}%
=\sum_{j=0}^{N}j^{2\beta}\hat{p}_{j}$ with $\hat{p}_{j}=n^{-1}\sum_{i=1}%
^{n}I\left(  X_{i}=j\right)  $ attains the rate $\eta_{n}^{2}$ for mean square
error, for an appropriate choice of $N$. On the other hand, the observations
$X_{1},\ldots,X_{n}$ are obtained from one specific measurement in the quantum
model \eqref{eq.Q.n.model}, in such a way that $p_{j}=\left\vert
\psi_{j}\right\vert ^{2}$ for $j\geq0$ and $F_{0}\left(  P\right)  =F\left(
\psi\right)  $. If the rate $\eta_{n}^{2}$ is unimprovable in the quantum
model then it certainly is in the present derived (less informative) classical
model. In the latter model, we note that since $F_{0}\left(  P\right)  $ is
linear and the law $P$ is restricted to a convex body, optimality of the rate
$\eta_{n}^{2}$ can be confirmed by standard methods, e.g. based on the concept
of modulus of continuity \cite{Don-Liu}. The current problem is thus an
example where the elbow phenomenon is present for estimation of a linear
functional; a specific feature here is that the probability measure $P$ is discrete.

\smallskip

\textit{Fuzzy quantum hypotheses.} Our method of proof of the lower bound for
quadratic functionals, which works in the approximating quantum Gaussian
model, utilizes the well-known idea of setting up two prior distributions and
then invoking a testing bound between simple hypotheses. This has been
described as the method of fuzzy hypotheses in the literature \cite{Tsyb}. A
summary of the present quantum variant could be as follows. First, the
Gaussian quantum model is represented in a fashion analogous to the classical
sequence model (\ref{eq.gsm}) where the $\vartheta_{j}$
correspond to the displacement parameter $u_{j}$ in certain Gaussian pure
states (the coherent states). These displacement parameters are then assumed
to be random as independent, non-identically distributed normal, for
$j=1,\ldots,N$ where $N=o(n)$. Now Gaussian averaging over the displacements
$u_{j}$ leads to certain non-pure Gaussian states, i.e. the thermal states as
the alternative, which happen to commute with the vacuum pure state
(corresponding to $u_{j}=0$) as the null hypothesis. Even though both are
again Gaussian states, by commutation the problem is reduced to testing
between two ordinary discrete probability distributions, i.e. the point mass
at $0$ and a certain geometric distribution with parameter $r_{j}$, depending
on $j=1,\ldots,N$. The combined error probability for this classical testing
problem with $N$ independent observations gives the lower risk bound.

\subsection{Nonparametric Testing\\}

\textit{The separation rate }$n^{-1/2}$. Recall that for the classical
Gaussian sequence model (\ref{eq.gsm}), for the testing
problem%
\begin{equation}%
\begin{array}
[c]{ll}%
H_{0}: & \vartheta=0\\
H_{1}(\varphi_{n}): & \sum_{j=1}^{\infty}j^{2\alpha}\vartheta_{j}^{2}\leq
L\text{ and }\left\Vert \vartheta\right\Vert _{2}\geq\varphi_{n}%
\end{array}
\label{test-a-2}%
\end{equation}
(Sobolev ellipsoid with an $L_{2}$-ball removed), the separation rate is
$\varphi_{n}=n^{-2\alpha/\left(  4\alpha+1\right)  }$ \cite{IngsterSuslina}.
We established that $\varphi_{n}=n^{-1/2}$ is the separation rate for the
quantum nonparametric testing problem (\ref{test-a}) involving a pure state
$\rho$. While this ``parametric'' rate for a nonparametric problem is somewhat
surprising, it should be noted that  there also exist testing problems for
classical i.i.d. data with nonparametric alternative where that separation
rate applies; cf \cite{IngsterSuslina}, sec. 2.6.2.

In our case, the rate $n^{-1/2}$  appears to be related to the fast rate
$\varphi_{n}^{2}=n^{-1}$ in the following nonparametric classical  problem:
given $n$ i.i.d. observations $X_{1},\ldots,X_{n}$ having law $P=\left\{
p_{j}\right\}  _{j=0}^{\infty}$ on the nonnegative integers, the hypotheses
are
\begin{equation}%
\begin{array}
[c]{ll}%
H_{0}: & P=\delta_{0}\text{ (the degenerate law at }0\text{)}\\
H_{1}(\varphi_{n}): & \left\Vert P-\delta_{0}\right\Vert _{1}\geq\varphi
_{n}^{2}.
\end{array}
\label{test-a-3}%
\end{equation}
For that, note first that
\[
\left\Vert P-\delta_{0}\right\Vert _{1}=1-p_{0}+\sum_{j=1}^{\infty}%
p_{j}=2\left(  1-p_{0}\right)  .
\]
The likelihood ratio test for $\delta_{0}$ against any $P\in H_{1}(\varphi
_{n})$ rejects if $\max_{1\leq j\leq n}X_{j}$ $>0$, thus it does not depend on
$P$. The pertaining sum of error probabilities is
\[
P\left(  \max_{1\leq j\leq n}X_{j}=0\right)  =p_{0}^{n}=\left(  1-\frac{1}%
{2}\left\Vert P-\delta_{0}\right\Vert _{1}\right)  ^{n}\leq\left(  1-\frac
{1}{2}\varphi_{n}^{2}\right)  ^{n}%
\]
and with a supremum over $P\in H_{1}(\varphi_{n})$, the upper bound is
attained.  This means that for $\varphi_{n}=cn^{-1/2}$, the minimax sum of
error probabilities tends to $\exp\left(  -c^{2}/2\right)  $, so that
$\varphi_{n}^{2}=n^{-1}$ is the separation rate here as claimed.

In fact there is a direct connection to the quantum nonparametric testing
problem (\ref{test-a}): in the latter, for $n=1$, consider a measurement
defined as follows. Let $\left\{  \left\vert \tilde{e}_{j}\right\rangle
\right\}  _{j=0}^{\infty}$ be an orthonormal basis in $\mathcal{H}$ such that
$\rho_{0}=\left\vert \tilde{e}_{0}\right\rangle \left\langle \tilde{e}%
_{0}\right\vert $ and consider the POVM $\left\{  \left\vert \tilde{e}%
_{j}\right\rangle \left\langle \tilde{e}_{j}\right\vert \right\}
_{j=0}^{\infty}$; the corresponding measurement yields  a probability measure
$P$ on the nonnegative integers. Here the state $\rho_{0}$ is mapped into
$\delta_{0}$ and an alternative state $\rho$ is mapped into $P=\left\{
p_{j}\right\}  _{j=0}^{\infty}$ such that $p_{0}=\mathrm{Tr}\left(  \rho
_{0}\rho\right)  $. Condition \textbf{(}\ref{distance-cond})\textbf{ }on the
distance of the two states implies (cp (\ref{eq.trace.norm.pure}))
\[
\varphi_{n}\leq\left\Vert \rho-\rho_{0}\right\Vert _{1}=2\sqrt{1-\mathrm{Tr}%
\left(  \rho_{0}\rho\right)  }=2\sqrt{1-p_{0}}=\sqrt{2\left\Vert P-\delta
_{0}\right\Vert _{1}}%
\]
so that up to a constant, the testing problem (\ref{test-a-3}) is obtained.

In the quantum problem (\ref{test-a}), we noted that the optimal test between
$\rho_{0}$ and a specific alternative $\rho$ depends on $\rho$, but found that
the test (binary POVM) $M_{n}=\left\{  \rho_{0}^{\otimes n},I-\rho
_{0}^{\otimes n}\right\}  $ is minimax optimal in the sense of the rate and
also in the sense of a sharp risk asymptotics. The sharp minimax optimality
seems to be a specific result for the quantum case. We note that the optimal
test $M_{n}$ can be realized via a measurement $\left\{  \left\vert \tilde
{e}_{j}\right\rangle \left\langle \tilde{e}_{j}\right\vert \right\}
_{j=0}^{\infty}$ as described above, applied separately to each component of
$\rho^{\otimes n}$, resulting in independent identically distributed  r.v.'s
$X_{1},\ldots,X_{n}$. The test $M_{n}$ then amounts to rejecting $H_{0}$ if
$\max_{1\leq j\leq n}X_{j}>0$.
Note that this measurement is incompatible with the one (\ref{eq.covmeas}) providing the optimal rate for state estimation.


\textit{Other separation rates.} In our proof of the lower bound for quadratic
functionals, we formulate  the nonparametric testing problem for pure states (\ref{eq.test-for-quad}) where the alternative includes the
restriction $\sum_{j\geq0}\left\vert \psi_{j}\right\vert ^{2}j^{2\beta}%
\geq\eta_{n}$, and establish that the rate $\eta_{n}=n^{-1+\beta/\alpha}$ is
unimprovable there. Introduce a seminorm
\[
\left\Vert \psi\right\Vert _{2,\beta}=\left(  \sum_{j\geq1}\left\vert \psi
_{j}\right\vert ^{2}j^{2\beta}\right)  ^{1/2}%
\]
(excluding the term for $j=0$) and write the restriction as
\begin{equation}
\left\Vert \psi\right\Vert _{2,\beta}\geq\varphi_{n}=\eta_{n}^{1/2}%
;\label{restric-beta}%
\end{equation}
then the case $\beta=0$ gives (cp (\ref{eq.trace.norm.pure}))
\[
\varphi_{n}^{2}\leq\sum_{j\geq1}\left\vert \psi_{j}\right\vert ^{2}%
=1-\left\vert \psi_{0}\right\vert ^{2}=1-\left\vert \left\langle \psi
|e_{0}\right\rangle \right\vert ^{2}=\frac{1}{4}\left\Vert \left\vert
e_{0}\right\rangle \left\langle e_{0}\right\vert -\left\vert \psi\right\rangle
\left\langle \psi\right\vert \right\Vert _{1}^{2},
\]
in other words, for $\rho_{0}=\left\vert e_{0}\right\rangle \left\langle
e_{0}\right\vert $ and $\rho=\left\vert \psi\right\rangle \left\langle
\psi\right\vert $, the restriction (\ref{restric-beta}) is equivalent to
$\left\Vert \rho-\rho_{0}\right\Vert _{1}\geq2\varphi_{n}$. In that sense, the
testing problems (\ref{test-a}) and (\ref{eq.test-for-quad}) in  are
equivalent up to a constant, if $\beta=0$ and $\rho_{0}=\left\vert
e_{0}\right\rangle \left\langle e_{0}\right\vert $. For $\beta>0$, the testing
problem (\ref{eq.test-for-quad}) in  is a quantum pure state analog of
the generalization of the classical problem (\ref{test-a-2}) where $\left\Vert
\vartheta\right\Vert _{2}\geq\varphi_{n}$ is replaced by $\left\Vert
\vartheta\right\Vert _{2,\beta}\geq\varphi_{n}$ ($\alpha$-ellipsoid with a
$\beta$-ellipsoid removed); the separation rate in the latter is $\varphi
_{n}=n^{-2\left(  \alpha-\beta\right)  /\left(  4\alpha+1\right)  }$ , cf.
\cite{IngsterSuslina}, sec. 6.2.1. In (\ref{eq.test-for-quad}) the
separation rate is $\varphi_{n}=n^{-1/2+\beta/2\alpha}$, i.e. of the more
typical nonparametric form as well.

\section{Proofs} \label{sec:proofs}

\begin{proof}[Proof of Theorem~\protect\ref{th.qlan}]



The direct map channel $T_n$ is defined as an isometric embedding
\begin{eqnarray*}
T_n : \mathcal{T}_1(\mathcal{H}^{\otimes_s n})& \to & \mathcal{T}_1(\mathcal{%
F} (\mathcal{H}_0)) \\
\rho &\mapsto& V_n \rho V_n^*.
\end{eqnarray*}
where $V_n: \mathcal{H}^{\otimes_s n}\to \mathcal{F}(\mathcal{H}_0)$ is an
isometry defined below. Since we deal with pure states, it suffices to prove
that
\begin{equation}\label{eq.convergence.psiu}
\underset{n\to\infty}{\lim\sup}\sup_{|\psi_0\rangle \in \mathcal{H}}
\sup_{\|u \|\leq \gamma_n} \left\| V_n \psi_{u}^{\otimes n} - G(\sqrt{n} u)
\right\| = 0.
\end{equation}
We now define the isometric embedding $V_n$ by showing its explicit action
on the vectors of an ONB. For any permutation $\sigma\in S_n$, let
\begin{equation*}
U_\sigma: |u_1\rangle \otimes \dots \otimes |u_n\rangle \mapsto |
u_{\sigma^{-1}(1)} \rangle \otimes \dots \otimes |u_{\sigma^{-1}(n)} \rangle
\end{equation*}
be the unitary action on $\mathcal{H}^{\otimes n}$ by tensor permutations.
Then $P_s := \frac{1}{n!} \sum_{\sigma\in S_n} U_{\sigma}$ is the orthogonal
projector onto the subspace of symmetric tensors $\mathcal{H}^{\otimes_s n}$%
. We construct an orthonormal basis in $\mathcal{H}^{\otimes_s n}$ as
follows.

Let $B_0:= \{ |e_1\rangle , |e_2\rangle, \dots\}$ be an orthonormal basis in
$\mathcal{H}_0$. Let $\tilde{\mathbf{n}} =(n_0, \mathbf{n})= (n_0, n_1,
\dots)$ be an infinite sequence of integers such that $\sum_{i\geq 0} n_i= n$%
, and note that only a finite number of $n_i$s are different from zero. Then
the symmetric vectors
\begin{equation*}
|\tilde{\mathbf{n}} \rangle = |n_0, n_1, n_2, \dots \rangle := \sqrt{ \frac{%
n!}{ n_0 ! \cdot n_1! \cdot \dots } } \,  P_s \left[ |\psi_0\rangle^{\otimes
n_0} \otimes \bigotimes_{i\geq 1} |e_i\rangle^{\otimes n_i} \right]
\end{equation*}
form an ONB of $\mathcal{H}^{\otimes_s n}$.

As discussed in section \ref{sec.Fock.spaces} the Fock space $\mathcal{F}(%
\mathcal{H}_0)$ can be identified with the infinite tensor product of
one-mode Fock spaces $\bigotimes_{i\geq 1} \mathcal{F}(\mathbb{C}|e_i
\rangle )$ which has an orthonormal number basis (or Fock basis) consisting
of products of number basis vectors of individual modes
\begin{equation*}
|\mathbf{n} \rangle := \bigotimes_{i\geq 1} |n_i\rangle
\end{equation*}
where $n_i\neq 0$ only for a finite number of indices. We define $V_n:
\mathcal{H}^{\otimes_s n}\to \mathcal{F}(\mathcal{H}_0)$ as follows
\begin{equation*}
V_n: |\tilde{\mathbf{n}}\rangle \mapsto |\mathbf{n}\rangle.
\end{equation*}
Its image consists of states with at most $n$ ``excitations'', with $%
|\psi_0\rangle^{\otimes n}$ being mapped to the vacuum state $|\mathbf{0}%
\rangle$. 
%
We would like to show that the embedded state $V_n |\psi_u \rangle^{\otimes
n}$ are well approximated by the coherent states $|G(\sqrt{n} u)\rangle$
uniformly over the local neighbourhood $\|u\|\leq \gamma_n$. For this we
will make use of the covariance and functorial properties of the second
quantisation construction in order to reduce the non-parametric LAE
statement to the corresponding one for 2-dimensional systems.

Let $|u\rangle\in \mathcal{H}_0$ be a fixed unit vector. Let $j: \mathbb{C}%
^2 \mapsto \mathcal{H}$ be the isometric embedding
\begin{equation*}
j:|0\rangle \mapsto |\psi_0\rangle, \qquad j:|1\rangle \mapsto |u\rangle
\end{equation*}
and let $j_0: \mathbb{C} |1\rangle\to \mathcal{H}_0$ be the restriction of $j
$ to the one dimensional subspace $\mathbb{C}|1\rangle$.
Since second quantisation is functorial under contractive maps, there is a
corresponding isometric embedding $J_0= \Gamma (j_0)$ satisfying
\begin{eqnarray}
J_0: \mathcal{F}(\mathbb{C} |1\rangle ) &\to& \mathcal{F} (\mathcal{H}_0)
\notag \\
|G(\alpha)\rangle &\mapsto & |G(j_0 (\alpha)) \rangle = |G(\alpha u )\rangle.
\label{eq.j.embedding}
\end{eqnarray}
Let $\tilde{V}_n : \left(\mathbb{C}^2\right)^{\otimes_s n}\to \mathcal{F}(%
\mathbb{C} |1\rangle)$ be the isometry constructed in the same way as $V_n$,
where $|0\rangle$ plays the role of $|\psi_0\rangle$ and $\mathbb{C}|1\rangle
$ is the analogue of $\mathcal{H}_0$. As before, let $|\tilde{\psi}_{\alpha} \rangle
= \sqrt{1-|\alpha|^2 } |0\rangle+ \alpha|1\rangle$, with $|\alpha|\leq 1$.
Then by the properties of the embedding map $V_n$ we have
\begin{equation}  \label{eq.j.v}
J_0 \tilde{V}_n |\tilde{\psi}_{\alpha}  \rangle^{\otimes n}= V_n | \psi_{\alpha u
}\rangle^{\otimes n}.
\end{equation}
From equations \eqref{eq.j.embedding} and \eqref{eq.j.v} we find
\begin{equation*}
\sup_{|\alpha|\leq \gamma_n} \left\| V_n \psi_{\alpha u}^{\otimes n} - G(%
\sqrt{n} \alpha u) \right\| = \sup_{|\alpha|\leq \gamma_n} \left\| \tilde{V}%
_n \tilde{\psi}_{\alpha} ^{\otimes n} - G(\sqrt{n} \alpha) \right\|
\end{equation*}
Since the right-hand side of the above equality is independent of $|u\rangle$
the same equality holds with supremum on the left side taken over all $%
|u\rangle \in\mathcal{H}_0$ with $ \| u\| =1$, which is the same as the supremum in equation \eqref{eq.convergence.psiu}.
Therefore the LAE for the non-parametric models has been reduced to that of a two-dimensional
(qubit) model.
This approximation has been established in the more general case of mixed states in \cite{GutaKahn2006,Guta&Janssens&Kahn}, but the current case of pure states allows an improvement in rate.
The product state $|\tilde{\psi}_{\alpha} \rangle^{\otimes n}$ is mapped into the
following pure state on the Fock space $\mathcal{F}(\mathbb{C} |1\rangle)$
$$
\tilde{V}_n |\tilde{\psi}_{\alpha}  \rangle^{\otimes n} =
\sum_{k=0}^n c_{k,n}(\alpha) |k\rangle, \quad  c_{k,n}(\alpha) = \alpha^k (1- |\alpha|^2)^{(n-k)/2} \sqrt{n \choose k}.
$$
On the other hand, in view of \eqref{eq.coherent.state} the coherent state can be written as
$$
G(\sqrt{n} \alpha) = \sum_k c_k(\sqrt{n}\alpha) |k\rangle, \quad
c_k(\sqrt{n}\alpha):= \exp(- n |\alpha|^2/2 )\frac{(\sqrt{n} \alpha)^{k}}{\sqrt{k!}}.
$$
Set $\alpha=\phi_{\alpha}\left\vert \alpha\right\vert $ where $\phi_{\alpha}$
is a phase; then it follows that  $c_{k,n}\left(  \alpha\right)  =\phi
_{\alpha}^{k}$ $c_{k,n}\left(  \left\vert \alpha\right\vert \right)  $ and
$c_{k}\left(  \sqrt{n}\alpha\right)  =\phi_{\alpha}^{k}$ $c_{k}\left(
\sqrt{n}\left\vert \alpha\right\vert \right)  $. With this we have
\begin{align}
\left\Vert \tilde{V}_{n}\tilde{\psi}_{\alpha}^{\otimes n}-G\left(  \sqrt
{n}\alpha\right)  \right\Vert ^{2}  & =\sum_{k=0}^{\infty}\left\vert
c_{k,n}\left(  \alpha\right)  -c_{k}\left(  \sqrt{n}\alpha\right)  \right\vert
^{2}\nonumber\\
& =\sum_{k=0}^{\infty}\left\vert c_{k,n}\left(  \left\vert \alpha\right\vert
\right)  -c_{k}\left(  \sqrt{n}\left\vert \alpha\right\vert \right)
\right\vert ^{2}.\label{squared-Hell}%
\end{align}
Let $X$ be a binomial r.v. with parameters $n,\left\vert \alpha\right\vert
^{2}$ and $Y$ be a Poisson r.v. with parameter  $n\left\vert \alpha\right\vert
^{2}$. Note that $c_{k,n}\left(  \left\vert \alpha\right\vert \right)
=P\left(  X=k\right)  ^{1/2}$ and $c_{k}\left(  \sqrt{n}\left\vert
\alpha\right\vert \right)  =P\left(  Y=k\right)  ^{1/2}$, and that therefore
(\ref{squared-Hell}) is the squared Hellinger distance between these two laws.
According to Theorem 1.3.1 (ii) in \cite{Reiss-book} we have
\[
\sum_{k=0}^{\infty}\left\vert c_{k,n}\left(  \left\vert \alpha\right\vert
\right)  -c_{k}\left(  \sqrt{n}\left\vert \alpha\right\vert \right)
\right\vert ^{2}\leq3\left\vert \alpha\right\vert ^{4}.
\]
Since $\left\vert \alpha\right\vert \leq\gamma_{n}=o(1)$, we have shown the
first part of LAE in which the i.i.d. and Gaussian models are expressed
in terms of the local parameter $|u\rangle$
\begin{equation}
\underset{n\rightarrow\infty}{\lim\sup}\sup_{|\psi_{0}\rangle\in\mathcal{H}%
}\sup_{\Vert u\Vert\leq\gamma_{n}}\left\Vert V_{n}\psi_{u}^{\otimes n}%
-G(\sqrt{n}u)\right\Vert =0.\label{30a}%
\end{equation}
Conversely, we define the reverse channel $S_{n}:\mathcal{T}_{1}%
(\mathcal{F}(\mathcal{H}_0))\rightarrow\mathcal{T}_{1}\left(  \mathcal{H}%
^{\otimes_{s}n}\right)  $ as follows. Let $P_{n}$ denote the orthogonal
projection in $\mathcal{F}(\mathcal{H}_{0})$ onto the image space of $V_{n}$,
i.e. the subspace with total excitation number at most $n$
\[
\mathcal{F}_{\leq n}(\mathcal{H}_{0}):=\mathrm{Lin}\{|n_{1},n_{2},\dots
\rangle\,:\,\sum_{i\geq1}n_{i}\leq n\}.
\]
Let $R_{n}:\mathcal{F}(\mathcal{H}_{0})\rightarrow\mathcal{H}^{\otimes_{s}n}$
be a right inverse of $V_{n}$, i.e. $R_{n}V_{n}=\mathbf{1}$. Then the reverse
channel is defined as
\[
S_{n}(\rho)=R_{n}P_{n}\rho P_{n}R_{n}^{\ast}+\mathrm{Tr}(\rho(1-P_{n}%
))|\psi_{0}\rangle\langle\psi_{0}|^{\otimes n}.
\]
Operationally, the action of $S_{n}$ consists of two steps. We first perform a
projection measurement with projections $P_{n}$ and $(\mathbf{1}-P_{n})$; if
the first outcome occurs the conditional state of the system is $P_{n}\rho
P_{n}/\mathrm{Tr}(P_{n}\rho)$ , while if the second outcome occurs the state
is $(\mathbf{1}-P_{n})\rho(\mathbf{1}-P_{n})/\mathrm{Tr}((\mathbf{1}%
-P_{n})\rho)$. In the second stage, if the first outcome was obtained we map
the projected state through the map $R_{n}$ into a state in $\mathcal{H}%
^{\otimes_{s}n}$, while if the second outcome was obtained, we prepare the
fixed state $|\psi_{0}\rangle\langle\psi_{0}|^{\otimes n}$.

When applied to the pure Gaussian states $|G(\sqrt{n}u)\rangle$, the output of
$S_{n}$ is the mixed state
\[
S_{n}(|G(\sqrt{n}u)\rangle\langle G(\sqrt{n}u)|)=p_{u}^{n}|\phi_{u}^{n}%
\rangle\langle\phi_{u}^{n}|+(1-p_{u}^{n})|\psi_{0}\rangle\langle\psi
_{0}|^{\otimes n}%
\]
where
\[
|\phi_{u}^{n}\rangle:=\frac{R_{n}P_{n}|G(\sqrt{n}u)\rangle}{\sqrt{p_{u}^{n}}%
},\quad p_{u}^{n}=\Vert P_{n}G(\sqrt{n}u)\Vert^{2}.
\]
The key observation is that the Gaussian states are almost completely
supported by the subspace $\mathcal{F}_{\leq n}(\mathcal{H}_{0})$, uniformly
with respect to the ball $\Vert u\Vert\leq\gamma_{n}$. Indeed, since
$V_{n}\psi_{u}^{\otimes n}$ is in $\mathcal{F}_{\leq n}\left(  \mathcal{H}%
_{0}\right)  $, from (\ref{30a}) and the properties of projections it follows
\[
\limsup_{n\rightarrow\infty}\sup_{\left\vert \psi_{0}\right\rangle }%
\sup_{\left\Vert u\right\Vert \leq\gamma_{n}}\left\Vert P_{n}G\left(  \sqrt
{n}u\right)  -G\left(  \sqrt{n}u\right)  \right\Vert =0,
\]
so that
\begin{equation}
\underset{n\rightarrow\infty}{\lim\sup}\sup_{|\psi_{0}\rangle}\sup_{\Vert
u\Vert\leq\gamma_{n}}\left(  1-p_{u}^{n}\right)  =0.\label{30b}%
\end{equation}
Now again from (\ref{30a}) and the fact that $R_{n}$ is the inverse of $V_{n}$
it follows
\[
\limsup_{n\rightarrow\infty}\sup_{\left\vert \psi_{0}\right\rangle }%
\sup_{\left\Vert u\right\Vert \leq\gamma_{n}}\left\Vert \psi_{u}^{\otimes
n}-R_{n}P_{n}G\left(  \sqrt{n}u\right)  \right\Vert =0,
\]
which in conjunction with (\ref{30b}) implies
\[
\underset{n\rightarrow\infty}{\lim\sup}\sup_{|\psi_{0}\rangle}\sup_{\Vert
u\Vert\leq\gamma_{n}}\left\Vert S_{n}(|G(\sqrt{n} u)\rangle\langle
G(\sqrt{n} u)|)-|\psi_{u}\rangle\langle\psi_{u}|^{\otimes n}\right\Vert
_{1}=0.
\]
This completes the proof of \eqref{eq.LAN1}.
\end{proof}


\begin{proof}[Proof of Theorem~\protect\ref{thm:UBestim}]
According to inequalities \eqref{eq.distances.inequality} and \eqref{eq.distances.inequality2} the two distances are equivalent on pure states, so it suffices to prove the upper bound for the trace-norm distance.

Firstly, a projective operation is applied to each of the $n$ copies
separately, whose aim is to truncate the state to a finite dimensional
subspace of dimension $d_n = [n^{1/(2\alpha+1)}]+1$. Let $P_n$ be the
projection onto the subspace $\mathcal{H}_n$ spanned by the first $d_n$
basis vectors $\{|e_0\rangle, \dots ,|e_{d_n-1}\rangle\}$. For a given state $%
|\psi\rangle$ the operation consists of randomly projecting the state with $%
P_n$ or $(\mathbf{1} - P_n)$, which produces i.i.d. outcomes $O_i\in\{0,1\}$
with $\mathbb{P}(O_i= 1) = p_n = \|P_n \psi\|^2$. The posterior state
conditioned on the measurement outcome is
\begin{equation*}
|\psi\rangle\langle \psi| \mapsto \left\{
\begin{array}{ccc}
|\psi^{(n)} \rangle \langle\psi^{(n)} | := \frac{P_n |\psi\rangle\langle
\psi|P_n}{p_n} & \text{with probability } ~ & ~p_n \\
&  &  \\
\frac{(\mathbf{1}- P_n) |\psi\rangle\langle \psi |(\mathbf{1}- P_n)}{ 1- p_n
} & \text{with probability } & ~~1-p_n%
\end{array}
\right.
\end{equation*}
Since $|\psi\rangle \langle \psi|\in S^\alpha(L)$, the probability
$1-p_n$ is bounded as
\begin{equation}  \label{eq.1minuspn}
1-p_n = \sum_{i=d_n}^\infty |\psi_i|^2 = \sum_{i=d_n}^\infty
i^{-2\alpha} i^{2\alpha} |\psi_i|^2 \leq d_n^{-2\alpha} \sum_{i=1}^\infty
i^{2\alpha} |\psi_i|^2 = n^{-2\alpha/(2\alpha+1)} L.
\end{equation}
Let $\tilde{n} =\sum_{i=1}^n O_i$ be the number of systems for which the
outcome was equal to 1, so that $\tilde{n}$ has binomial distribution $%
\mathrm{Bin}(n, \,p_n)$. Then $\mathbb{E}(\tilde{n}/ n)=p_n$ and $\mathrm{Var}(%
\tilde{n}/ n)=p_n (1-p_n)/n=O(1/n)$. Therefore $\tilde{n}/ n \to $1 in
probability.

In the second step we discard the systems for which the outcome was $0$, and
we collect those with outcome $1$, so that the joint state is $|\psi^{(n)}
\rangle \langle\psi^{(n)} |^{\otimes \tilde n}$ which is supported by the
symmetric subspace $\mathcal{H}_n^{\otimes_s \tilde{n}}$. In order to
estimate the truncated state $|\psi^{(n)} \rangle$ (and by implication $%
|\psi\rangle $), we perform a \emph{covariant} measurement $M_n$ \cite%
{Hayashi} whose space of outcomes is the space of pure states $\hat{\rho}_n=|%
\hat{\psi}_n\rangle\langle \hat{\psi}_n |$ over $\mathcal{H}_n$, and the
infinitesimal POVM element is
\begin{equation}  \label{eq.covmeas}
M_n(d \hat{\rho}) ={\binom{{\tilde{n}+d_n-1} }{{d_n-1}}} \,\hat{\rho}^{\otimes n} \, d%
\hat{\rho}.
\end{equation}
The covariance property means that the unitary group has a covariant action
on states and their corresponding probability distributions
\begin{equation*}
\mathbb{P}^{M_n}_{U\rho U^*} (d\hat{\rho}) = \mathrm{Tr} (U\rho U^* \cdot d%
\hat{\rho}) =\mathbb{P}^{M_n}_{\rho} (d (U^*\hat{\rho} U)).
\end{equation*}
Recall that the trace-norm distance squared for pure states is given by
$ d_1^2(\rho, \rho^\prime) := \|\rho - \rho^\prime \|_1^2
=4 (1- |\langle \psi |\psi^\prime \rangle|^2)$.
In \cite{Hayashi} it has been
shown that, conditionally on $\tilde{n}$, the risk of the estimator $\hat{\rho}$ with respect to the trace-norm square distance is\footnote{Reference \cite{Hayashi} uses a fidelity distance erroneously called ``Bures distance" , which for pure states coincides with the trace-norm distance up to a constant}
\begin{equation*}
\mathbb{E}^{\tilde{n}} \left[d_1^2(\hat{\rho}_n, \rho^{(n)})\right] = \frac{%
4(d_n-1 )}{d_n + \tilde{n}}.
\end{equation*}
Using the triangle inequality we have $d_1^2(\hat{\rho%
}_n, \rho ) \leq 2 (d_1^2(\hat{\rho}_n, \rho^{(n)}) + d_1^2(\rho,
\rho^{(n)}))$. Since $|\psi^{(n)}\rangle= P_n |\psi\rangle /\sqrt{p_n}$, the
bias term is $d_1^2 (\rho, \rho^{(n)}) = 4(1 - p_n)$, which by %
\eqref{eq.1minuspn} is bounded by $4 n^{-2\alpha/(2\alpha+1)} L$. Therefore
\begin{equation*}
\mathbb{E} \left[d_b^2(\hat{\rho}_n, \rho) \right] \leq 8\mathbb{E} \left[%
\frac{(d_n-1) }{d_n + \tilde{n}}\right] + 8 n^{-2\alpha/(2\alpha+1)} L.
\end{equation*}
For an arbitrary small $\varepsilon >0 $, we have
\begin{equation*}
\mathbb{E} \left[\frac{(d_n-1) }{d_n + \tilde{n}}\right] \leq P \left[ \frac{%
\tilde n}n < 1 - \varepsilon \right] + \mathbb{E} \left[\frac{(d_n-1) }{d_n
+ n \cdot \tilde{n}/n} \cdot I (\frac {\tilde n}n \geq 1- \varepsilon )%
\right] \leq O\left(\frac 1n \right) + C\frac{d_n}n.
\end{equation*}
Putting together the last two upper bounds concludes the proof.

\end{proof}

\begin{proof}[Proof of Theorem~\protect\ref{thm:lbsobolev}]
Let us denote by $R_n^E = \inf _{\widehat \psi_n} \sup_{\psi \in S^\alpha(L)}
\mathbb{E}_\psi \left[ \| \widehat \psi_n - \psi \|_2^2 \right]$ the minimax
risk.

The first step is to reduce the set of states $S^\alpha(L)$ to a finite
hypercube denoted $S_{1:N}^\alpha(L)$ consisting of certain ``truncated''
vectors $|\psi\rangle = \sum_{1 \leq i\leq N} \psi_i |e_i\rangle$ which have
$N\asymp n^{1/(2\alpha+1)}$ non-zero coefficients with respect to the
standard basis. This will provide a lower bound to the minimax risk. The
coefficients are chosen as
\begin{equation*}
\psi_j = \pm \frac {\sigma_j}{\sqrt{n}}, \quad \sigma_j^2 = \lambda(1-
(j/N)^{2\alpha}), \quad j=1,\ldots, N, \text{ for some fixed }\lambda >0
\end{equation*}
and we check that they satisfy the ellipsoid constraint
\begin{eqnarray*}
\sum_{j \geq 1} |\psi_j|^2 j^{2 \alpha} &=& \frac {\lambda}n \sum_{j=1}^N
(j^{2 \alpha}- j^{4 \alpha} N^{-2\alpha}) \leq \frac {N^{2 \alpha +1}}n
\frac{2 \alpha \lambda}{(2\alpha+1)(4\alpha +1)} (1+o(1)) \leq L
\end{eqnarray*}
for an appropriate choice of $\lambda >0$.

Using the factorisation property \eqref{eq.factorisation.Gaussian} we can
identify the corresponding Gaussian states with the $N$-mode state defined
by $|\phi \rangle = \otimes_{j=1}^N |G(\sqrt{n} \psi_j) \rangle $, where the
remaining modes are in the vacuum state and can be ignored.

Thus
\begin{eqnarray*}
R_n^E & \geq & \inf _{\widehat \psi} \sup_{\psi \in S_{1:N}^\alpha(L)}
\mathbb{E}_\psi \left[ \| \widehat \psi - \psi \|_2^2 \right] \\
&=& \inf _{\widehat \psi} \sup_{\psi \in S_{1:N}^\alpha(L)} \mathbb{E}_\psi %
\left[ \sum_{j=1}^N | \widehat \psi_j - \psi_j |^2 \right].
\end{eqnarray*}
The supremum over the finite hypercube $S_{1:N}^\alpha(L)$ is bounded from
below by the average over all its elements. This turns the previous maximal
risk into a Bayesian risk, that we can further bound from below as follows:
\begin{eqnarray}
R_n^E & \geq & \inf _{\widehat \psi} \frac 1{2^N} \sum_{\psi \in
S_{1:N}^\alpha (L)} \sum_{j=1}^N \mathbb{E}_{\psi} \left[| \widehat \psi_j -
\psi_j |^2 \right]  \notag \\
& = & \inf _{\widehat \psi} \sum_{j=1}^N \frac 1{2^N} \sum_{\psi \in
S_{1:N}^\alpha (L)} \mathbb{E}_{\psi} \left[| \widehat \psi_j - \psi_j |^2 %
\right]  \notag \\
& \geq & \sum_{j=1}^N \inf _{\widehat \psi_j} \frac 1{2^N} \sum_{\psi \in
S_{1:N}^\alpha (L) } \mathbb{E}_{\psi} \left[ | \widehat \psi_j - \psi_j |^2 %
\right] .  \label{I1}
\end{eqnarray}
In the second line $\widehat \psi$ is the result of an arbitrary measurement
and estimation procedure of the state $|G(\sqrt{n} \psi)\rangle$. In the
third line each infimum is over procedures for estimating the component $%
\psi_j$ only; since such procedure may not be compatible with a single
measurement, the third line is upper bounded by the second.

\bigskip

The second major step in the proof of the lower bounds is to reduce the risk
over all measurements, to testing two simple hypotheses. Let us bound from
below the term \eqref{I1} for arbitrary fixed $j$ between 1 and $N$:

\begin{equation*}
\frac 1{2^N} \sum_{\psi \in S_{1:N}^\alpha (L) } \mathbb{E}_{\psi} \left[
| \widehat \psi_j - \psi_j |^2 \right]
\end{equation*}
\begin{equation*}
= \frac 12 \left\{ \frac 1{2^{N-1}} \sum_{\psi \in S_{(j +)}^\alpha(L)}
\mathbb{E}_{\psi} \left[ | \widehat \psi_j - \sigma_j/\sqrt{n} |^2 \right]
\right. \left. + \frac 1{2^{N-1}} \sum_{\psi \in S_{(j-)}^\alpha(L)} \mathbb{%
E}_{\psi} \left[ | \widehat \psi_j - (-\sigma_j/\sqrt{n} ) |^2 \right]
\right\}
\end{equation*}
\begin{equation}
= \frac 12 \left\{ \mathbb{E}_{\rho^+_j} \left[ | \widehat \psi_j -
\sigma_j/\sqrt{n} |^2 \right] + \mathbb{E}_{\rho^-_j} \left[ | \widehat
\psi_j - (-\sigma_j/\sqrt{n} ) |^2 \right] \right\} ,
 \label{I2}
\end{equation}

where the sum over $\psi \in S_{(j \pm)}^\alpha(L)$ means that the $j^{th}$
coordinate is fixed to $\pm \sigma_j /\sqrt{n}$ and all $k^{th}$
coordinates, for $k \ne j$, take values in $\{ \sigma_k/\sqrt{n}, - \sigma_k/%
\sqrt{n}\}$. In the third line, we denote by $\rho^\pm_j$ the average state
over states in $S^\alpha_{(j\pm)}(L)$.

Let us define the testing problem of the two hypotheses $H_0: \rho =
\rho^+_j $ against $H_1: \rho = \rho^-_j $. For a given estimator $\widehat
\psi_j$ we construct the test
\begin{equation*}
\Delta = I \left( \left| \widehat \psi_j - \frac{\sigma_j}{\sqrt{n}} \right|
> \left| \widehat \psi_j - (- \frac{\sigma_j}{\sqrt{n}}) \right| \right),
\end{equation*}
and decide $H_1$ or $H_0$, if $\Delta $ equals 1 or 0, respectively. By the
Markov inequality, we get that
\begin{eqnarray*}
\mathbb{E}_{\rho^{\pm}_j } \left[ \left| \widehat \psi_j - (\pm \frac{%
\sigma_j}{\sqrt{n}} ) \right|^2 \right] &\geq &\frac{\sigma_j^2}n \mathbb{P}%
_{\rho^{\pm}_j} \left( \left| \widehat \psi_j - ( \pm \frac{\sigma_j}{\sqrt{n%
}}) \right| \geq \frac {\sigma_j}{\sqrt{n}} \right) .  \label{t0}
\end{eqnarray*}
On the one hand,
\begin{equation}  \label{t1}
\mathbb{P}_{\rho^+_j} \left( | \widehat \psi_j - \sigma_j/\sqrt{n} | \geq
\frac {\sigma_j}{\sqrt{n}} \right) \geq \mathbb{P}_{\rho^+_j}(\Delta = 1).
\end{equation}
Indeed, under $\mathbb{P}_{\rho^+_j}$, the event $\Delta = 1$ implies that $%
|\widehat \psi_j - \frac{\sigma_j}{\sqrt{n}}| > |\widehat \psi_j + \frac{%
\sigma_j}{\sqrt{n}}|$, which further implies by the triangular inequality
that
\begin{equation*}
\left| \widehat \psi_j - \frac{\sigma_j}{\sqrt{n}} \right| \geq \frac{2
\sigma_j}{\sqrt{n}} - \left| \widehat \psi_j + \frac{\sigma_j}{\sqrt{n}}
\right| \geq \frac{ 2 \sigma_j}{\sqrt{n}} - \left| \widehat \psi_j - \frac{%
\sigma_j}{\sqrt{n}} \right| ,
\end{equation*}
giving $| \widehat \psi_j - \psi_j | \geq \frac {\sigma_j}{\sqrt{n}}$. By a
similar reasoning for the $\mathbb{P}_{\rho^-_j}$ distribution we get
\begin{equation}  \label{t2}
\mathbb{P}_{\rho^-_j} \left( | \widehat \psi_j + \sigma_j/\sqrt{n} | \geq
\frac {\sigma_j}{\sqrt{n}} \right) \geq \mathbb{P}_{\rho^-_j}(\Delta = 0).
\end{equation}

By using (\ref{t1}) and (\ref{t2}) in (\ref{I2})
\begin{eqnarray*}
&&\frac 12 \left\{ \mathbb{E}_{\rho^+_j} \left[ \left| \widehat \psi_j -
\sigma_j/\sqrt{n} \right|^2 \right] + \mathbb{E}_{\rho^-_j} \left[ \left|
\widehat \psi_j - (-\sigma_j/\sqrt{n} ) \right|^2 \right] \right\} \\
&\geq &
\frac{\sigma_j^2}{2n} \left( \mathbb{P}_{\rho^+_j}(\Delta = 1) + \mathbb{P}%
_{\rho^-_j}(\Delta = 0) \right).
\end{eqnarray*}
To summarise, we have lower bounded the MSE by the probability of error for
testing between the states $\rho^\pm_j$. At closer inspection, these states
are of the form $|G( \sigma_j )\rangle\langle G( \sigma_j ) | \otimes \rho$
and $|G( - \sigma_j )\rangle\langle G( - \sigma_j ) | \otimes \rho$ where $%
\rho$ is a fixed state obtained by averaging the coherent states of all the
modes except $j$. Recall that the optimal testing error in (\ref{HHbound})
gives a further bound from below
\begin{equation*}
\mathbb{P}_{\rho^+_j}(\Delta = 1) + \mathbb{P}_{\rho^-_j}(\Delta = 0) \geq 1
- \frac 12 \|\rho_j^+ - \rho_j^-\|_1.
\end{equation*}
Moreover, the state $\rho$ can be dropped without changing the optimal
testing error
\begin{equation*}
\|\rho_j^+ - \rho_j^-\|_1 = \| |G(\sigma_j)\rangle \langle G(\sigma_j)|  -
|G(-\sigma_j)\rangle \langle G(-\sigma_j)| \|_1 = 2\sqrt{1 - \exp(- 4
\sigma_j^2)}.
\end{equation*}
We conclude that
\begin{equation*}
\inf _{\widehat \psi_j} \frac 12 \left\{ \mathbb{E}_{\rho^+_j} \left[ \left|
\widehat \psi_j - \sigma_j/\sqrt{n} \right|^2 \right] + \mathbb{E}%
_{\rho^-_j} \left[ \left| \widehat \psi_j - (-\sigma_j/\sqrt{n} ) \right|^2 %
\right] \right\} \geq \frac {\sigma_j^2}{4 n} \cdot \exp(-4 \sigma_j^2)
\end{equation*}
and we further use this in (\ref{I2}) to get
\begin{eqnarray*}
R_n^E &\geq &\sum_{j=1}^N \frac {\sigma_j^2}{4 n} \cdot \exp(-4 \sigma_j^2) \\
& =&
\frac Nn \cdot \frac {\lambda}{4 N} \sum_{j=1}^N \left( 1- (\frac jN)^{2
\alpha} \right)  \exp \left( -4 \cdot \lambda ( 1- (\frac jN)^{2 \alpha} )
\right) \geq  c \frac Nn.
\end{eqnarray*}
Indeed, the average over $j$ is the Riemann sum associated to the integral
of a positive function and can be bounded from below by some constant $c>0$
depending on $\alpha$. Moreover, $N/n \asymp n^{-2 \alpha/(2\alpha+1)}$ and
thus we finish the proof of the theorem.
\end{proof}


\begin{proof}[Proof of Theorem \protect\ref{thm:lbsobolev.iid}]
Let $\tilde{R}_n^E =
\inf _{|\widehat \psi_n\rangle} \sup_{|\psi\rangle \in S^\alpha(L)}
\mathbb{E}_\rho \left[ d (\hat\rho_n , \rho)^2 \right] $ be the minimax risk
for $\mathcal{Q}_n$.

We bound from below the risk by restricting to (pure) states in a neigbourhood $%
\Sigma_n(e_0)$ of the basis vector $|e_0\rangle$ defined as follows. As in %
\eqref{eq.psi-u} we write the state and the estimator in terms of their
corresponding local vectors
\begin{equation*}
|\psi \rangle = \sqrt{1- \|u\|^2} |e_0\rangle + |u\rangle, \qquad |\hat{\psi}
\rangle = \sqrt{1- \|\hat{u}\|^2} |e_0 \rangle + |\hat{u}\rangle, \qquad
|u\rangle, |\hat{u}\rangle \perp |e_0\rangle.
\end{equation*}
Then the neighbourhood is given by $\Sigma_n(e_0) :=\{ |\psi_u\rangle : \|
u\| \leq \gamma_n \}$; we choose  $\gamma_n =o(1) $ with a rate to be determined later. Such states are
described by the local model $\mathcal{Q}_n (e_0, \gamma_n)$, cf. equation %
\eqref{eq.e.n}. The risk is bounded from below by
\begin{equation*}
\tilde{R}_n^E \geq \inf _{|\widehat \psi_n\rangle} \, \sup_{|\psi\rangle \in S^\alpha(L) \cap
\Sigma_n(e_0)} \mathbb{E}_\rho \left[ d (\hat\rho_n , \rho )^2 \right] .
\end{equation*}
By using the triangle inequality we can assume that $\hat{\psi} \in
\Sigma_n(e_0)$, while incurring at most a factor 2 in the risk. By using the
quadratic approximation \eqref{eq.lcoap.approx.bures.normone} we find that
\begin{equation}  \label{eq.d.hilbert.norm}
d^2 (\widehat\rho_n , \rho ) = k \|u-\hat{u}\|^2 + O(\gamma_n^4)
\end{equation}
where $k=1$ or $k= 4$ depending on which distance we use.
At this point we impose a condition on $\gamma_n$:
\begin{equation}\label{new_gamma}
  \gamma_n^4 = o\left(n^{-2\alpha / (2 \alpha +1)} \right).
\end{equation}
Since now $O(\gamma_n^4)$ decreases faster than $n^{-2\alpha/(2\alpha+1)}$, the second term does not
contribute to the asymptotic rate and can be neglected, so that the problem
has been reduced to that of estimating the local parameter $u$ with respect
to the Hilbert space distance. To study the latter, we further restrict the
set of states to a hypercube similar to the one in the proof of Theorem \ref%
{thm:lbsobolev}, consisting of states $|\psi_u\rangle$ with ``truncated''
local vectors $| u\rangle = \sum_{1 \leq i\leq N} u_i |e_i\rangle$ belonging
to $S_{1:N}^\alpha(L)$. As before, there are $N\asymp n^{1/(2\alpha+1)}$
non-zero coefficients of the form
\begin{equation*}
u_j = \pm \frac {\sigma_j}{\sqrt{n}}, \quad \sigma_j^2 = \lambda(1- (j/N)^{2
\alpha}), \quad j=1,\ldots, N.
\end{equation*}
It has been already shown that such vectors belong to the ellipsoid $%
S^\alpha(L)$. Additionally, we show that they also belong to the
local ball $\Sigma_n(e_0)$. Indeed
\begin{align*}
\|u\|^2 = \sum_{j=1}^{N}\left\vert u_{j}\right\vert ^{2} & =\frac{1}{n}%
\sum_{j=1}^{N}\sigma_{j}^{2}=\frac{1}{n}\sum_{j=1}^{N}\lambda \left(
1-\left( j/N\right) ^{2 \alpha }\right)  \notag \\
& = \frac{N}{n} \left( \frac{1}{N}\sum_{j=1}^{N}\lambda \left( 1-\left(
j/N\right) ^{2 \alpha }\right) \right)  \leq C_{1}\frac{N}{n} ,
\end{align*}
where we used that as $N\rightarrow\infty$ the expression between the parentheses tens to a finite integral. 
As $N$ scales as $n^{1/(2\alpha+1)}$, the upper bound becomes
\begin{equation*}
\|e_0 - \psi_u\|^2 \leq C_2 n^{-2\alpha/(2\alpha+1)} = o(\gamma_n^2)
\end{equation*}
provided that $\gamma_n$ fulfills
\begin{equation}\label{new-gamma2}
  n^{-2\alpha/(2\alpha+1)} = o(\gamma_n^2)
\end{equation}
and then the state $|\psi_u\rangle$ belongs to the local ball $\Sigma_n(e_0)$. Taking
into account \eqref{eq.d.hilbert.norm} the risk is therefore lower bounded
as
\begin{equation*}
\tilde{R}_n^E \geq \inf _{\widehat u} \sup_{u \in S_{1:N}^\alpha(L) }
\mathbb{E}_{\rho_u} \left[ \|u-\hat{u}\|^2 \right] + o(n^{-1}).
\end{equation*}
where $\rho_u = |\psi_u\rangle\langle \psi_u|$, and the infimum is now taken over the local component $|\hat{u}\rangle $
of an estimator $|\hat{\psi} \rangle = \sqrt{1- \|\hat{u}\|^2} |e_0 \rangle
+ |\hat{u}\rangle$. Now, if we choose $\gamma _n$ as
$$
\gamma_n = n^{-\alpha/(2 \alpha +1)} \log(n),
$$
then both \eqref{new_gamma} and \eqref{new-gamma2} are fulfilled.

The first term is further lower bounded by passing to
the Bayes risk for the uniform distribution over $S_{1:N}^\alpha(L)$,
similarly to the proof of Theorem \ref{thm:lbsobolev}
\begin{equation*}
\tilde{R}_n^E \geq \sum_{j=1}^N \,\inf _{\widehat u_j} \frac 1{2^N} \sum_{u
\in S_{1:N}^\alpha (L) } \mathbb{E}_{\psi_u} \left[ | \widehat u_j - u_j |^2 %
\right] + o(n^{-1}).
\end{equation*}
By following the same steps we get
\begin{eqnarray}
&& \frac 1{2^N} \sum_{u \in S_{1:N}^\alpha (L) } \mathbb{E}_{\rho_u} \left[ |
\widehat u_j - u_j |^2 \right] \nonumber \\
&=& \frac 12 \left\{ \mathbb{E}_{\tau^+_j} %
\left[ | \widehat \psi_j - \sigma_j/\sqrt{n} |^2 \right] + \mathbb{E}%
_{\tau^-_j} \left[ | \widehat \psi_j - (-\sigma_j/\sqrt{n} ) |^2 \right]
\right\} ,  \notag \\
&\geq& \frac{\sigma_j^2}{2n} \left( \mathbb{P}_{\tau^+_j}(\Delta = 1) +
\mathbb{P}_{\tau-}(\Delta = 0) \right) \geq \frac{\sigma_j^2}{2n} \cdot ( 1
- \frac{1}{2}\|\tau^+_j - \tau^-_j\|_1) ,  \label{eq.lower.bound.x}
\end{eqnarray}
where we denote by $\tau^\pm_j$ the average state over states $%
|\psi_u\rangle \langle \psi_u|^{\otimes n}$ with $u\in S^\alpha_{(j\pm)}(L)$%
, and $\Delta$ is a test for the hypotheses $H_0: \tau = \tau^+_j $ and $%
H_1: \tau = \tau^-_j $. In the last inequality we used the Helstrom bound
\cite{Helstrom} which expresses the optimal average error probability for
two states discrimination in terms of the norm-one distance between states.

We now make use of the local asymptotic equivalence result in Theorem \ref%
{th.qlan}. From \eqref{eq.LAN1} we know that there exist quantum channels $%
S_n$ such that
\begin{equation*}
\delta_n := \max_{u\in S_{1:N}^\alpha (L)} \left\| |\psi_u\rangle\langle
\psi_u |^{\otimes n} - S_n\left( |G(\sqrt{n} u)\rangle\langle G(\sqrt{n} u)
| \right) \right\|_1 \leq \Delta(\mathcal{Q}_n, \mathcal{G}_n)= o(1).
\end{equation*}
By Lemma \ref{lemma.distance.channels} we get
\begin{equation*}
\| \tau^+_j - \tau^-_j\|_1 \leq \|\rho^+_j - \rho^-_j\|_1 + 2\delta_n
\end{equation*}
where $\rho^\pm_j$ are the corresponding mixtures in the Gaussian model as
defined in the proof of Theorem \ref{thm:lbsobolev}.
From \eqref{eq.lower.bound.x} we then get
\begin{equation*}
\frac 1{2^N} \sum_{u \in S_{1:N}^\alpha (L) } \mathbb{E}_{\rho_u} \left[ |
\widehat u_j - u_j |^2 \right] \geq \frac{\sigma_j^2}{2n} \cdot ( 1 - \frac{1%
}{2}\|\rho^+_j - \rho^-_j\|_1 - \delta_n) \geq \frac{\sigma_j^2}{4n} \cdot
(\exp(-4 \sigma_j^2) - 2\delta_n).
\end{equation*}
Indeed, as we have
$$
\| \rho_j^+ - \rho_j^-\|_1 = \| |G(\sigma_j) \rangle \langle G(\sigma_j)| - |G(-\sigma_j) \rangle \langle G(-\sigma_j)| \|_1 = 2 \sqrt{1 - \exp(-4 \sigma_j^2)},
$$
we obtain
\begin{eqnarray*}
 \frac{\sigma_j^2}{2n} \cdot ( 1 - \frac{1}{2}\|\rho^+_j - \rho^-_j\|_1 - \delta_n)
 & \geq & \frac{\sigma_j^2}{2n} \cdot (1 - \sqrt{1-\exp(-4 \sigma_j^2)} - \delta_n)\\
 & = & \frac{\sigma_j^2}{2n} \cdot \left( \frac{\exp(-4 \sigma_j^2)}{1+ \sqrt{1-\exp(-4 \sigma_j^2)}} - \delta_n \right)\\
 & \geq & \frac{\sigma_j^2}{4n} \cdot (\exp(-4 \sigma_j^2) - 2\delta_n).
 \end{eqnarray*}
Now note that
\begin{align*}
  \min_j \exp(-4 \sigma_j^2) & = \exp(-4 \lambda (1-N^{-{2 \alpha}})) \\
   &= \exp(-4 \lambda) (1+ o(1))
\end{align*}
and $\delta_n = o(1)$, so that
$$
\min_j (\exp(-4 \sigma_j^2) - 2\delta_n) \geq C_3 >0
$$
for sufficiently large $n$. Consequently,
\begin{align*}
 \tilde{R}_n^E & \geq C_3 \sum_{j=1}^N \frac{\sigma_j^2}{4n} = \frac{C_3 \lambda}{4} \frac Nn \left(N^{-1} \sum_{j=1}^N (1 - (j/N)^{2 \alpha}) \right)  \\
   & \asymp \frac Nn \asymp n^{-2\alpha/(2 \alpha +1)}.
\end{align*}
\end{proof}

\begin{proof}[Proof of Theorem~\protect\ref{thm:Festim}]
The usual bias-variance decomposition yields
\begin{equation*}
\mathbb{E}_\psi \left( \widehat F_n - F(\psi)\right)^2 = \left( \mathbb{E}_\psi
\widehat F_n - F(\psi)\right)^2  + Var_\psi \left( \widehat F_n\right).
\end{equation*}
The bias can be upper bounded as
\begin{eqnarray*}
&& \left| F(\psi) - \mathbb{E}_\psi \widehat F_n \right| =\left| F(\psi ) -
\sum_{j=1}^N p_j \cdot j^{2\beta} \right| 
\\
&=& \sum_{j\geq N+1} p_j \cdot
j^{2\beta} \leq N^{-2 (\alpha - \beta)}\sum_{j\geq N+1} p_j \cdot
j^{2\alpha}\leq L N^{-2(\alpha - \beta)}.
\end{eqnarray*}
For the variance, let us note that the vector
\begin{equation*}
\widehat V = n \cdot(\hat p_1, \dots , \hat p_N, \hat p^*_{N+1}), \quad
\mathrm{with} \quad \widehat p^*_{N+1} = n^{-1} \sum_{k=1}^n I(X_k \geq
N+1),
\end{equation*}
has a multinomial distribution with parameters $n$ and probability vector $%
V:=(p_1, \dots ,p_N, p^*_{N+1} = \sum_{j \geq N+1} p_j)^\top$. The
covariance matrix of a multinomial vector writes $n\cdot (\mathrm{Diag}(V) -
V \cdot V^\top) $, where $\mathrm{Diag}(V)$ denotes the diagonal matrix with
entries from $V$. In particular, if $\widehat p := (\hat p_1,...,\hat
p_N)^\top$, $p := (p_1,...,p_N)^\top$ and $B := (1, 2^{2 \beta}, ..., N^{2
\beta})^\top$ then
\begin{equation*}
Cov_\psi (\widehat F_n) = Cov_\psi ( B^\top \cdot \widehat p) = B^\top \cdot
Cov_\psi (\widehat p) \cdot B  = \frac 1n \cdot B^\top \cdot (\mathrm{Diag}(
p) - p \cdot p ^\top)\cdot B.
\end{equation*}
This gives
\begin{equation*}
Cov_\psi (\widehat F_n) \leq \frac 1n \cdot B^\top \cdot \mathrm{Diag}(p) \cdot
B = \frac 1n \sum_{j=1}^N p_j \cdot j^{4\beta}.
\end{equation*}
The bound of this last term and the resulting bound of the risk is treated
separately for the two cases.

\textbf{a)} Case $\alpha \geq 2 \beta$. In that case,
\begin{equation*}
\sum_{j=1}^N p_j \cdot j^{4\beta} \leq \sum_{j=1}^N p_j \cdot j^{2\alpha}
\leq L \mbox{ implying that } Var(\widehat F_n) \leq \frac Ln.
\end{equation*}
The upper bound of the risk is, in this case,
\begin{equation*}
\mathbb{E}_\psi \left( \widehat F_n - F(\psi)\right)^2 \leq L^2
N^{-4(\alpha-\beta)} + \frac Ln.
\end{equation*}
If we choose $N\asymp n^{1/(4(\alpha-\beta))}$ or larger, then the
parametric rate is attained for the risk:
\begin{equation*}
\mathbb{E}_\psi \left( \widehat F_n - F(\psi)\right)^2 = O(1) \cdot n^{-1}.
\end{equation*}

\textbf{b)} Case $\beta < \alpha < 2 \beta$. Here we have,
\begin{equation*}
Cov_\psi ( \widehat F_n) \leq \frac 1n \sum_{j=1}^N p_j \cdot j^{4\beta} \leq
\frac 1n \sum_{j=1}^N p_j \cdot j^{4\beta-2 \alpha} j^{2 \alpha} p_j \leq
\frac {N^{4 \beta - 2\alpha}}n L.
\end{equation*}
The upper bound of the risk becomes
\begin{equation*}
\mathbb{E}_\psi \left( \widehat F_n - F(\psi)\right)^2 \leq L^2
N^{-4(\alpha-\beta)} + \frac {N^{4 \beta - 2\alpha}}n L.
\end{equation*}
The optimal choice of the parameter $N$ that balances the two previous terms
is $N \asymp n^{1/(2 \alpha)}$, giving the attainable rate for the quadratic
risk
\begin{equation*}
\mathbb{E}_\psi \left( \widehat F_n - F(\psi)\right)^2 = O(1)\cdot n^{-2 ( 1 -
\beta / \alpha )}.
\end{equation*}
Cases \textbf{a)} and \textbf{b)} together prove that the rate $\eta_n ^2$
is attainable.
\end{proof}

\begin{proof}
[Proof of Theorem \ref{thm:Flowbounds}]Denote by 
$$
R_{n}^{F} = \inf_{\widehat
F_{n}} \sup_{\psi\in S^{\alpha}(L)} \eta_{n}^{-2} \cdot\mathbb{E}_{\psi
}\left(  \widehat F_{n} - F(\psi) \right)  ^{2}
$$ the minimax risk.

The case \textbf{a)} where $\alpha\geq2 \beta$ reduces to the Cram\'er-Rao
bound that proves that the parametric rate $1/n$ is always a lower bound for
the mean square error for estimating $F(\psi)$.

We prove that in the case \textbf{b)} where $\beta< \alpha< 2 \beta$, this
bound from below increases to $n^{-2(1-\beta/\alpha)}$ (up to constants). By
the Markov inequality,
\begin{equation}
\label{IF1}\eta_{n}^{-2} \cdot\mathbb{E}_{\psi}\left(  \widehat F_{n} -
F(\psi) \right)  ^{2} \geq\frac14 \cdot\mathbb{P}_{\psi}\left(  | \widehat
F_{n} - F(\psi)| \geq\frac{\eta_{n}}2 \right)  .
\end{equation}

Let us restrict the set of pure states $S^{\alpha}(L)$ to its intersection
with the local model $\mathcal{Q}_{n}(e_{0},\gamma_{n})$ (see equation
\eqref{eq.e.n}) where $|\psi_{u}\rangle=\sqrt{1-\Vert u\Vert^{2}}\cdot
|e_{0}\rangle+|u\rangle$ is such that $\Vert u\Vert\leq\gamma_{n}$, with
$\gamma_{n}=(\log n)^{-1}$. In other words, $u$ belongs to the set
\[
s^{\alpha}(L,\gamma_{n})=\left\{  u\in\ell_{2}(\mathbb{N}^{\ast}):\sum
_{j\geq1}|u_{j}|^{2}j^{2\alpha}\leq L\mbox{ and }\Vert u\Vert\leq\gamma
_{n}\right\}  .
\]
Using the fact that $F(e_{0})=0$, we have
\begin{align}
&  \sup_{\psi\in S^{\alpha}(L)}\frac{1}{4}\cdot\mathbb{P}_{\psi}\left(
|\widehat{F}_{n}-F(\psi)|\geq\frac{\eta_{n}}{2}\right) \nonumber\\
&  \geq\frac{1}{4}\max\left\{  \mathbb{P}_{e_{0}}\left(  |\widehat{F}_{n}%
|\geq\frac{\eta_{n}}{2}\right)  ,\sup_{u\in s^{\alpha}(L,\gamma_{n}%
),F(\psi_{u})\geq\eta_{n}}\mathbb{P}_{\psi_{u}}\left(  |\widehat{F}_{n}%
-F(\psi_{u})|\geq\frac{\eta_{n}}{2}\right)  \right\} \nonumber\\
&  \geq\frac{1}{8}\left\{  \mathbb{P}_{e_{0}}\left(  |\widehat{F}_{n}%
|\geq\frac{\eta_{n}}{2}\right)  +\sup_{u\in s^{\alpha}(L,\gamma_{n}%
),F(\psi_{u})\geq\eta_{n}}\mathbb{P}_{\psi_{u}}\left(  |\widehat{F}_{n}%
-F(\psi_{u})|\geq\frac{\eta_{n}}{2}\right)  \right\} \nonumber\\
&  \geq\frac{1}{8}\left\{  \mathbb{P}_{e_{0}}\left(  |\widehat{F}_{n}%
|\geq\frac{\eta_{n}}{2}\right)  +\sup_{u\in s^{\alpha}(L,\gamma_{n}%
),F(\psi_{u})\geq\eta_{n}}\mathbb{P}_{\psi_{u}}\left(  |\widehat{F}_{n}%
|<\frac{\eta_{n}}{2}\right)  \right\}  \label{eq.lb.r}%
\end{align}
where in the last inequality we used that $|\widehat{F}_{n}|<\eta_{n}/2$ and
$F(\psi_{u})\geq\eta_{n}$ imply $|\widehat{F}_{n}-F(\psi_{u})|\geq\eta_{n}%
/2$.
Note also that $F(\psi_{u})=F\left(  u\right)  $ for $\left\vert
u\right\rangle \in\mathcal{H}_{0}$; we now consider the testing problem with
hypotheses
\begin{equation}
\left\{
\begin{array}
[c]{ll}%
H_{0}: & |u\rangle=|0\rangle\\
H_{1}(\alpha,L,\gamma_{n},\eta_{n}): & |u\rangle,\text{ with }u\in s^{\alpha
}(L,\gamma_{n})\text{ and }F(u)\geq\eta_{n}.
\end{array}
\right.  \label{eq.test-for-quad}%
\end{equation}
Let $\Delta=\Delta(\eta_{n})=I(|\widehat{F}_{n}|\geq\eta_{n}/2)$ be the test
that accepts the null hypothesis when $\Delta=0$ and rejects the null
hypothesis when $\Delta=1$. Then the right-hand side of \eqref{eq.lb.r} is
lower bounded by the sum of the error probability of type I and of the maximal
error probability of type II of $\Delta$. We can describe $\Delta$ as a binary
POVM $M=(M_{0},M_{1})$, depending on $\eta_{n}$: $M(\eta_{n})=(M_{0}(\eta
_{n}),M_{1}(\eta_{n}))$. Thus,
\begin{equation}
\mathbb{P}_{e_{0}}\left(  |\widehat{F}_{n}|\geq\frac{\eta_{n}}{2}\right)
=\mathrm{Tr}(|e_{0}\rangle\langle e_{0}|^{\otimes n}\cdot M_{1})
\end{equation}
and
\begin{equation}
\mathbb{P}_{\psi_{u}}\left(  |\widehat{F}_{n}|<\frac{\eta_{n}}{2}\right)
=\mathrm{Tr}(|\psi_{u}\rangle\langle\psi_{u}|^{\otimes n}\cdot M_{0}).
\label{IF4}%
\end{equation}
By putting together \eqref{IF1}-\eqref{IF4}, we get that the minimax risk has
the lower bound
\[
R_{n}^{F}\geq\frac{1}{8}\inf_{M}\left(  \langle e_{0}^{\otimes n}|M_{1}%
|e_{0}^{\otimes n}\rangle+\underset{u\in s^{\alpha}(L,\gamma_{n}),\newline
F(u)\geq\eta_{n}}{\sup}\langle\psi_{u}^{\otimes n}|M_{0}|\psi_{u}^{\otimes
n}\rangle\right)  .
\]

Now, using the local asymptotic equivalence Theorem \ref{th.qlan} with respect
to the state $|\psi_{0}\rangle:=|e_{0}\rangle$ we map the i.i.d. ensemble
$|\psi_{u}\rangle^{\otimes n}$ to the Gaussian state $|G(u)\rangle
\in\mathcal{F}(\mathcal{H}_{0})$. The lower bound becomes
\begin{equation}
R_{n}^{F}\geq\frac{1}{8}\inf_{M}\left(  \langle0|M_{1}|0\rangle+\underset{u\in
s^{\alpha}(L,\gamma_{n}),\newline F(u)\geq\eta_{n}}{\sup}\langle G(\sqrt
{n}u)|M_{0}|G(\sqrt{n}u)\rangle\right)  +o(1) \label{IF5a}%
\end{equation}
where the infimum is taken over tests $M=(M_{0},M_{1})$ and the $o(1)$ terms
stems from the vanishing Le Cam distance $\Delta(\mathcal{Q}_{n}(e_{0}%
,\gamma_{n}),\mathcal{G}_{n}(e_{0},\gamma_{n}))$. The lower bound has been
transformed into a testing problem for the Gaussian model.

In order to bound from below the maximal error probability of type II, we
define a prior distribution on the set of alternatives and average over the
whole set with respect to this a priori distribution. Similarly to the
classical proofs of lower bounds, our construction will lead to a test of
simple hypotheses: the former null and the constructed averaged state. Assume
that $\{u_{j}\}_{j\geq1}$ are all independently distributed, such that $u_{j}$
has a complex (bivariate) Gaussian distribution $N_{2}(0,\frac{1}{2}\sigma
_{j}^{2}\cdot I_{2})$ for all $j$ from 1 to $N$, and that $u_{j}=0$ for all
$j>N$, where $I_{2}$ is the $2\times2$ identity matrix. The $\sigma_{j}^{2}$
are defined as
\begin{equation}
\sigma_{j}^{2}=\lambda\left(  1-\left(  \frac{j}{N}\right)  ^{2\alpha}\right)
_{+}, \label{def:opt-1}%
\end{equation}
where $\lambda,N>0$ are selected such that
\begin{equation}
\text{ }\sum_{j\geq1}j^{2\alpha}\sigma_{j}^{2}=L(1-\varepsilon)\text{ and
}\sum_{j\geq1}j^{2\beta}\sigma_{j}^{2}=n^{-1+\beta/\alpha}(1+\varepsilon),
\label{def:opt-2}%
\end{equation}
for an arbitrary $\varepsilon>0$. Let us denote by $\Pi$ the joint prior
distribution of $\{u_{j}\}_{j\geq1}$.

Such a choice of the prior distribution was first introduced in
\cite{Ermakov} for establishing sharp minimax risk bounds for
nonparametric testing in the Gaussian white noise model. This construction
represents an analog of the prior distribution used in Pinsker's theory for
sharp estimation of functions. In our case, using a Gaussian prior as an
alternative hypothesis leads to the well-known Gaussian thermal state.

The essence of this construction is that the random vectors $u=\{u_{j}%
\}_{j\geq1}$ concentrate asymptotically, with probability tending to 1, on the
spherical segment
\[
\{u\in\ell_{2}(\mathbb{N}):C\,n^{-1}\leq\Vert u\Vert^{2}\leq C\,n^{-1}%
(1+2\varepsilon^{\prime})\},
\]
for $\varepsilon^{\prime}>0$ depending on $\varepsilon$ and some constant
$C>0$ depending on $\alpha$ and $\beta$ described later on, and on the
alternative set of hypothesis, $H_{1}(\alpha,L,\gamma_{n},\eta_{n})$. Note
that the spherical segment is included in the set $\Vert u\Vert\leq\gamma_{n}%
$, as $\gamma_{n}=(\log n)^{-1}\gg n^{-1/2}$. The asymptotic concentration
is proved by the following lemma.

\begin{lemma}
\label{lemma:prior} A unique solution $\left(  \lambda,N\right)  $ of
(\ref{def:opt-1}), (\ref{def:opt-2}), exists for $n$ large enough and admits
an asymptotic expansion with respect to $n$
\begin{align*}
\lambda\sim n^{-1-1/2\alpha}C_{\lambda}\frac{(1+\varepsilon)^{(\alpha
+1/2)/(\alpha-\beta)}}{(1-\varepsilon)^{(\beta+1/2)/(\alpha-\beta)}}  &
\text{, }C_{\lambda}=\frac{((2\beta+1)(2\beta+2\alpha+1))^{(\alpha
+1/2)/(\alpha-\beta)}}{2\alpha(L(2\alpha+1)(4\alpha+1))^{(\beta+1/2)/(\alpha
-\beta)}}\nonumber\\
N\sim n^{1/2\alpha}C_{N}\left(  \frac{1-\varepsilon}{1+\varepsilon}\right)
^{1/(2(\alpha-\beta))}  &  \text{, }C_{N}=\left(  \frac{L(2\alpha
+1)(4\alpha+1)}{(2\beta+1)(2\beta+2\alpha+1)}\right)  ^{1/(2(\alpha-\beta))}.
\nonumber\\
\end{align*}
The independent complex Gaussian random variables $u_{j}\sim N_{2}(0,\frac
{1}{2}\sigma_{j}^{2}I_{2})$, with $\sigma_{j}$'s and $\left(  \lambda
,N\right)  $ given in (\ref{def:opt-1}), (\ref{def:opt-2}), are such that, for
an arbitrary $\varepsilon>0$,
\begin{align}
\mathbb{P}\left(  C\,n^{-1}\leq\sum_{j=1}^{N}\left\vert u_{j}\right\vert
^{2}\leq C\,n^{-1}(1+2\varepsilon^{\prime})\right)   &  \rightarrow
1,\label{IN1}\\
\mathbb{P}\left(  \sum_{j=1}^{N}j^{2\alpha}\left\vert u_{j}\right\vert
^{2}\leq L\right)   &  \rightarrow1,\label{IN2}\\
\mathbb{P}\left(  \sum_{j=1}^{N}j^{2\beta}\left\vert u_{j}\right\vert ^{2}\geq
n^{-1+\beta/\alpha}\right)   &  \rightarrow1, \label{IN3}%
\end{align}
where $C=C_{\lambda}\cdot C_{N}\cdot2\alpha/(2\alpha+1)$ is a positive
constant depending on $\alpha$ and $\beta$, and $\varepsilon^{\prime}>0$
depends only on $\varepsilon$.
\end{lemma}

\begin{proof}
[Proof of Lemma~\ref{lemma:prior}]The solution of the problem (\ref{def:opt-1}%
), (\ref{def:opt-2}) can be found in \cite{Ermakov} (see also
\cite{JiNussbaum}, Lemma A.1 ) for $\beta=0$; a similar reasoning applies here.
Let us prove that the random variables $\{u_{j}\}_{j=1,...,N}$ satisfy
(\ref{IN1}) to (\ref{IN3}). We have
\begin{align}
\sum_{j=1}^{N}\sigma_{j}^{2}  &  =\lambda\sum_{j=1}^{N}\left(  1-\left(
\frac{j}{N}\right)  ^{2\alpha}\right)  \sim\lambda N\frac{2\alpha}{2\alpha
+1}\nonumber\\
&  \sim C_{\lambda}C_{N}\frac{2\alpha}{2\alpha+1}n^{-1}(1+\varepsilon
)^{\alpha/(\alpha-\beta)}(1-\varepsilon)^{-\beta/(\alpha-\beta)}%
=C\,n^{-1}(1+\varepsilon^{\prime}), \label{sumsigma2}%
\end{align}
where we denote $\varepsilon^{\prime}=(1+\varepsilon)^{\alpha/(\alpha-\beta
)}(1-\varepsilon)^{-\beta/(\alpha-\beta)}-1$ which is positive for all
$\varepsilon\in\left(  0,1\right)  $.

Note that $E\left\vert u_{j}\right\vert ^{2}=\sigma_{j}^{2}$ and $Var\left(
\left\vert u_{j}\right\vert ^{2}\right)  =\sigma_{j}^{4}$. We have
\begin{eqnarray*}
&&\mathbb{P}\left(  C\,n^{-1}\leq\sum_{j=1}^{N}\left\vert u_{j}\right\vert
^{2}\leq C\,n^{-1}(1+2\varepsilon^{\prime})\right) \\
& =&1-\mathbb{P}\left(
\sum_{j=1}^{N}\left\vert u_{j}\right\vert ^{2}<C\,n^{-1}\right)
-\mathbb{P}\left(  \left\vert u_{j}\right\vert ^{2}>C\,n^{-1}(1+2\varepsilon
^{\prime})\right)  .
\end{eqnarray*}
Now, by the Markov inequality,
\begin{align*}
\mathbb{P}\left(  \sum_{j=1}^{N}\left\vert u_{j}\right\vert ^{2}%
<C\,n^{-1}\right)   &  =\mathbb{P}\left(  \sum_{j=1}^{N}(\left\vert
u_{j}\right\vert ^{2}-\sigma_{j}^{2})<C\,n^{-1}-C\,n^{-1}(1+\varepsilon
^{\prime}+o(1))\right) \\
&  \leq\mathbb{P}\left(  \sum_{j=1}^{N}(\sigma_{j}^{2}-\left\vert
u_{j}\right\vert ^{2})>C\,n^{-1}(\varepsilon^{\prime}+o(1))\right) \\
&  \leq\frac{\sum_{j=1}^{N}Var(\left\vert u_{j}\right\vert ^{2})}%
{C^{2}\,n^{-2}\varepsilon^{\prime2}/2}\leq\frac{2\sum_{j=1}^{N}\sigma_{j}^{4}%
}{C^{2}\,n^{-2}\varepsilon^{\prime2}}\\
&  \asymp\frac{\lambda^{2}N}{C^{2}\,n^{-2}\varepsilon^{\prime2}}\asymp
n^{-1/2\alpha}=o(1).
\end{align*}
Moreover,
\[
\mathbb{P}\left(  \sum_{j=1}^{N}\left\vert u_{j}\right\vert ^{2}%
>C\,n^{-1}(1+2\varepsilon^{\prime})\right)  =\mathbb{P}\left(  \sum_{j=1}%
^{N}(\left\vert u_{j}\right\vert ^{2}-\sigma_{j}^{2})>C\,n^{-1}(\varepsilon
^{\prime}+o(1))\right),
\]
which is an $o(1)$ and this finishes the proof of (\ref{IN1}).

Also, in view of (\ref{def:opt-2}), we have
\begin{align*}
\mathbb{P}\left(  \sum_{j=1}^{N}j^{2\alpha}\left\vert u_{j}\right\vert
^{2}>L\right)   &  =\mathbb{P}\left(  \sum_{j=1}^{N}j^{2\alpha}(\left\vert
u_{j}\right\vert ^{2}-\sigma_{j}^{2})>L\,\varepsilon\right) \\
&  \leq\frac{\sum_{j=1}^{N}j^{4\alpha}Var\left(  \left\vert u_{j}\right\vert
^{2}\right)  }{L^{2}\,\varepsilon^{2}}=\frac{\sum_{j=1}^{N}j^{4\alpha}%
\sigma_{j}^{4}}{L^{2}\,\varepsilon^{2}}\\
&  \asymp\frac{\lambda^{2}N^{4\alpha+1}}{L^{2}\,\varepsilon^{2}}\asymp
n^{-1/2\alpha}=o(1),
\end{align*}
proving (\ref{IN2}). Also,
\begin{align*}
\mathbb{P}\left(  \sum_{j=1}^{N}j^{2\beta}\left\vert u_{j}\right\vert
^{2}<n^{-1+\beta/\alpha}\right)   &  \leq\mathbb{P}\left(  \sum_{j=1}%
^{N}j^{2\beta}(\left\vert u_{j}\right\vert ^{2}-\sigma_{j}^{2})<-n^{-1+\beta
/\alpha}\varepsilon\right) \\
&  \leq\frac{\sum_{j=1}^{N}j^{4\beta}Var(\left\vert u_{j}\right\vert ^{2}%
)}{n^{-2+2\beta/\alpha}\,\varepsilon^{2}}=\frac{\sum_{j=1}^{N}j^{4\beta}%
\sigma_{j}^{4}}{n^{-2+2\beta/\alpha}\,\varepsilon^{2}}\\
&  \asymp\frac{\lambda^{2}N^{4\beta+1}}{n^{-2+2\beta/\alpha}\,\varepsilon^{2}%
}\asymp n^{-1/2\alpha}=o(1),
\end{align*}
proving (\ref{IN3}).
\end{proof}

Let us go back to (\ref{IF5a}) and bound from below the maximal error
probability of type II by the averaged risk, with respect to our prior measure
$\Pi$:
\begin{align*}
&  \underset{u\in s^{\alpha}(L),F(u)\geq\eta_{n}}{\sup}\langle G(\sqrt
{n}\,u)|M_{0}|G(\sqrt{n}\,u)\rangle\\
& \geq\int_{H_{1}(\alpha,L,\gamma_{n}%
,\eta_{n})}\mathrm{Tr}(|G(\sqrt{n}\,u)\rangle\langle G(\sqrt{n}\,u)|\cdot
M_{0})\Pi(du)\\
&  =\mathrm{Tr}\left(  \int|G(\sqrt{n}\,u)\rangle\langle G(\sqrt{n}%
\,u)|\Pi(du)\cdot M_{0}\right) \\
& -\int_{H_{1}(\alpha,L,\gamma_{n},\eta_{n}%
)^{C}}\mathrm{Tr}(|G(\sqrt{n}\,u)\rangle\langle G(\sqrt{n}\,u)|\cdot M_{0}%
)\Pi(du)\\
&  \geq\mathrm{Tr}\left(  \int|G(\sqrt{n}\,u)\rangle\langle G(\sqrt{n}%
\,u)|\Pi(du)\cdot M_{0}\right)  -\Pi(H_{1}(\alpha,L,\gamma_{n},\eta_{n})^{C}).
\end{align*}
In the last inequality we used that $\mathrm{Tr}(|G(\sqrt{n}\,u)\rangle\langle
G(\sqrt{n}\,u)|\cdot M_{0})\leq1$. By Lemma~\ref{lemma:prior}, $\Pi
(H_{1}(\alpha,L,\gamma_{n},\eta_{n})^{C})=o(1)$ and thus we deduce from
(\ref{IF5a}) that
\begin{align*}
&R_{n}^{F} \geq \\
&\frac{1}{8}\inf_{M} \left(  \mathrm{Tr}\left(  |G(0)\rangle\langle
G(0)|\cdot M_{1}\right) +\mathrm{Tr}\left(  \int |G(\sqrt{n}\,u)\rangle\langle
G(\sqrt{n}\,u)|\Pi(du)\cdot M_{0}\right)  \right) \\
& +o(1).
\end{align*}
We recognize in the previous line the sum of error probabilities of type I and
II for testing two simple quantum hypotheses, i.e. the underlying state is
either $|G(0)\rangle$ or the mixed state
\[
\Phi:=\int|G(\sqrt{n}\,u)\rangle\langle G(\sqrt{n}\,u)|\Pi(du).
\]
As a last step of the proof, we characterize more precisely the previous mixed
Gaussian state as a thermal state and use classical results from quantum
testing of two simple hypotheses to give the bound from below of the testing
risk. Recall from Section \ref{sec.Fock.spaces}, equation
\eqref{eq.factorisation.Gaussian} that coherent states $|G(\sqrt{n}%
\,u)\rangle$ factorize as tensor product of one-mode coherent states with
displacements $u_{j}$, i.e. $\otimes_{j\geq1}|G(\sqrt{n}u_{j})\rangle$.
A coherent state with displacement $z=x+iy$ with $x,\,y\in\mathbb{R}$ is fully
characterized by its Wigner function given by equation
\eqref{eq.wigner.coherent}.
Since the prior is Gaussian, our mixed state $\Phi$ is Gaussian and can be
written
\begin{align*}
& \int|G(\sqrt{n}\,u)\rangle\langle G(\sqrt{n}\,u)|\Pi(du) \\
 &  =\left(
\bigotimes_{j=1}^{N}\int|G(\sqrt{n}\,u_{j})\rangle\langle G(\sqrt{n}%
\,u_{j})|\Pi_{j}(du_{j})\right)  \otimes\left(  \bigotimes_{j\geq
N+1}|0\rangle\langle0|\right) \\
&  :=\bigotimes_{j=1}^{N}\Phi_{j}\otimes\left(  \bigotimes_{j\geq
N+1}|0\rangle\langle0|\right)
\end{align*}
where $\Pi_{j}$ represents the bivariate centred Gaussian distribution with
covariance matrix $\sigma_{j}^{2}/2\cdot I_{2}$ over the complex plane
$u_{j}=x_{j}+iy_{j}$. Using equation \eqref{eq.twirling}, and setting
$\sigma^{2}=n\sigma_{j}^{2}/2$ there, we find that the individual modes with
index $j\leq N$ are centred Gaussian thermal states $\Phi_{j}=\Phi(r_{j})$
(cf. definition \eqref{eq.thermal.state}) with $r_{j}=n\sigma_{j}^{2}%
/(n\sigma_{j}^{2}+1)$.

In order to bound from below the right-hand side term in (\ref{IF5a}) we use
the theory of quantum testing of two simple hypotheses
\[
H_{0}:\otimes_{j\geq1}\Phi(0)\quad\text{against}\quad H_{1}:\otimes_{j=1}%
^{N}\Phi(r_{j})\otimes_{j\geq N+1}\Phi(0).
\]
Using (\ref{HHbound}), it is easy to see that this testing problem is
equivalent to
\[
H_{0}:\left(  \Phi(0)\right)  ^{\otimes N}\quad\text{against}\quad
H_{1}:\otimes_{j=1}^{N}\Phi(r_{j}).
\]
As the vacuum and the thermal state are both diagonalized by the Fock basis,
they commute, which reduces the problem to a classical test between the
$N$-fold products of discrete distributions $H_{0}:\{\mathcal{G}(0)\}^{\otimes
N}$ and $H_{1}:\{\otimes_{j=1}^{N}\mathcal{G}(r_{j})\}$. In view of the form
(\ref{eq.thermal.state}) of the thermal state, $\mathcal{G}(r_{j})$ is the
geometric distribution $\left\{  (1-r_{j})r_{j}^{k}\right\}  _{k=0}^{\infty}$
and $\mathcal{G}(0)$ is the degenerate distribution concentrated at $0$. The
optimal testing error is given by the maximum likelihood test which decides
$H_{0}$ if and only if all observations are $0$. The type I error is 0 and the
type II error is
\[
\prod_{j=1}^{N}(1-r_{j})=\prod_{j=1}^{N}\frac{1}{n\sigma_{j}^{2}+1}\geq
\exp\left(  -n\sum_{j=1}^{N}\sigma_{j}^{2}\right)  \geq\exp(-c),
\]
for some $c>0$, where in the last inequality we used (\ref{sumsigma2}). Using
this in (\ref{IF5a}), we get as a lower bound
\[
R_{n}^{F}\geq\exp(-c)+o(1)\geq c_{0},
\]
where $c_{0}>0$ is some constant depending on $c$. This finishes the proof.
\end{proof}


\begin{proof}[Proof of Theorem~\ref{thm:Test-1}]
Let $\varphi_{n}=c_{n}n^{-1/2}$ for a positive sequence $c_{n}$. Let
$M_{n}=(\rho_{0}^{\otimes n},I-\rho_{0}^{\otimes n})$ be the well-known
projection test for the problem (\ref{test-a}). Then%
\begin{align*}
R_{n}^{T}(M_{n})  &  =\mathrm{Tr}(\rho^{\otimes n}\cdot\rho_{0}^{\otimes
n})+\mathrm{Tr}(\rho_{0}^{\otimes n}\cdot(I-\rho_{0}^{\otimes n}))\\
&  =\left(  \mathrm{Tr}(\rho\cdot\rho_{0})\right)  ^{n}=|\langle\psi|\psi
_{0}\rangle|^{2n}.
\end{align*}
Let us recall that for any pure states $\rho=|\psi\rangle\langle\psi|$ and
$\rho_{0}=|\psi_{0}\rangle\langle\psi_{0}|$, we have
\begin{equation}
\Vert\rho-\rho_{0}\Vert_{1}=2\sqrt{1-|\langle\psi|\psi_{0}\rangle|^{2}}\text{
}, \label{trace-dist-pure}%
\end{equation}
thus $|\langle\psi|\psi_{0}\rangle|^{2}=1-\frac{1}{4}\Vert\rho-\rho_{0}%
\Vert_{1}^{2}$ and hence
\[
R_{n}^T(M_{n})=\left(  1-\frac{1}{4}\Vert\rho-\rho_{0}\Vert_{1}^{2}\right)
^{n}.
\]
For any $\rho$ satisfying the alternative hypothesis $H_{1}(\varphi_{n})$, we
have $\Vert\rho-\rho_{0}\Vert_{1}\geq\varphi_{n}$ and consequently
\begin{align}
\mathbb{P}_{e}^{M_{n}}\left(  \varphi_{n}\right)   &  \leq\left(  1-\frac
{1}{4}\varphi_{n}^{2}\right)  ^{n}=\left(  1-\frac{c_{n}^{2}}{4}n^{-1}\right)
^{n}\nonumber\\
&  \leq\left(  \exp\left(  -\frac{c_{n}^{2}}{4}n^{-1}\right)  \right)
^{n}=\exp\left(  -\frac{c_{n}^{2}}{4}\right)  .
\nonumber
\end{align}
If now $\varphi_{n}/\varphi_{n}^{\ast}\rightarrow\infty$ then $c_{n}%
\rightarrow\infty$ and $\mathbb{P}_{e}^{M_{n}}\left(  \varphi_{n}\right)
\rightarrow0$, so that the second relation in (\ref{sep-rate-def}) is fulfilled.

Consider now the case $\varphi_{n}/\varphi_{n}^{\ast}\rightarrow0$ so that
$c_{n}\rightarrow0$. For any vector $v\in\mathcal{H}$ define
\begin{equation}
\left\Vert v\right\Vert _{\alpha}^{2}=\sum_{j=0}^{\infty}\left\vert
\left\langle e_{j}|v\right\rangle \right\vert ^{2}j^{2\alpha};
\label{alpha-seminorm-def}%
\end{equation}
then $\left\Vert v\right\Vert _{\alpha}$ is a seminorm on the space of $v$
fulfilling $\left\Vert v\right\Vert _{\alpha}^{2}<\infty$. The assumption that
$\rho_{0}=|\psi_{0}\rangle\langle\psi_{0}|\in S^{\alpha}\left(  L^{\prime
}\right)  $ means that $\left\Vert \psi_{0}\right\Vert _{\alpha}^{2}\leq
L^{\prime}<L$. For some $N>0$, consider the linear space%
\[
\mathcal{H}_{0,N}=\left\{  u\in\mathcal{H}:\left\langle u|\psi_{0}%
\right\rangle =0,\left\langle u|e_{j}\right\rangle =0,j>N\right\}  ;
\]
it is nonempty if $N\geq1$. Let $u\in\mathcal{H}_{0,N}$ $,$ $\left\Vert
u\right\Vert =1$ be an unit vector; and for $\varepsilon>0$ consider%
\begin{equation}
\psi_{u,\varepsilon}=\psi_{0}\sqrt{1-\varepsilon^{2}}+\varepsilon u.
\label{perturb-1}%
\end{equation}
Then $\left\Vert \psi_{u,\varepsilon}\right\Vert =1$, $\rho_{u,\varepsilon
}=|\psi_{u,\varepsilon}\rangle\langle\psi_{u,\varepsilon}|$ is a pure state,
and
\[
|\langle\psi_{u,\varepsilon}|\psi_{0}\rangle|^{2}=1-\varepsilon^{2}.
\]
According to (\ref{trace-dist-pure}) we then have
\[
\left\Vert \rho_{u,\varepsilon}-\rho_{0}\right\Vert _{1}=2\sqrt{1-|\langle
\psi_{u,\varepsilon}|\psi_{0}\rangle|^{2}}=2\varepsilon
\]
so for a choice $\varepsilon=c_{n}n^{-1/2}/2$ it follows $\left\Vert
\rho_{u,\varepsilon}-\rho_{0}\right\Vert _{1}=\varphi_{n}$ and $\rho
_{u,\varepsilon}\in B\left(  \varphi_{n}\right)  $. On the other hand, by
(\ref{perturb-1}) and the triangle inequality
\[
\left\Vert \psi_{u,\varepsilon}\right\Vert _{\alpha}\leq\sqrt{1-\varepsilon
^{2}}\left\Vert \psi_{0}\right\Vert _{\alpha}+\varepsilon\left\Vert
u\right\Vert _{\alpha}.
\]
Now $\left\Vert u\right\Vert _{\alpha}<\infty$ for $u\in\mathcal{H}_{0,N}$,
and by assumption $\left\Vert \psi_{0}\right\Vert _{\alpha}<L^{1/2}$, so for
sufficiently large $n$
\[
\left\Vert \psi_{u,\varepsilon}\right\Vert _{\alpha}\leq L^{1/2}%
\]
and thus $\rho_{u,\varepsilon}\in S^{\alpha}\left(  L\right)  $. Thus
$\rho_{u,\varepsilon}\in S^{\alpha}\left(  L\right)  \cap B\left(  \varphi
_{n}\right)  $ for sufficiently large $n$. By\textbf{ (}\ref{HHbound}) the
optimal error probability for testing between states $\rho_{u,\varepsilon}$
and $\rho_{0}$ fulfills
\[
\inf_{M \text{ binary POVM}}R_{n}^{T}(\rho_{0}^{\otimes n},\rho_{u,\varepsilon
}^{\otimes n}, M)=1-\frac{1}{2}\left\Vert \rho_{0}^{\otimes n}-\rho
_{u,\varepsilon}^{\otimes n}\right\Vert _{1}%
\]%
\begin{align}
&  =1-\sqrt{1-|\langle\psi_{0}^{\otimes n}|\psi_{u,\varepsilon}^{\otimes
n}\rangle|^{2}}=1-\sqrt{1-|\langle\psi_{0}|\psi_{u,\varepsilon}\rangle|^{2n}%
}\nonumber\\
&  =1-\sqrt{1-\left(  1-\varepsilon^{2}\right)  ^{n}}=1-\sqrt{1-\left(
1-c_{n}^{2}n^{-1}/4\right)  ^{n}}. \nonumber 
\end{align}
Obviously if $c_{n}^{2}\rightarrow0$ then $\left(  1-c_{n}^{2}n^{-1}/4\right)
^{n}\rightarrow1$ so that
\[
\inf_{M \text{ binary POVM}}R_{n}^{T}(\rho_{0}^{\otimes n},\rho_{u,\varepsilon
}^{\otimes n},M )\geq1+o\left(  1\right)  .
\]
But since $\rho_{u,\varepsilon}\in S^{\alpha}\left(  L\right)  \cap B\left(
\varphi_{n}\right)  $ we have
\[
\mathbb{P}_{e}^{\ast}\left(  \varphi_{n}\right)  \geq\inf_{M \text{ binary
POVM}}R_{n}^{T}(\rho_{0}^{\otimes n},\rho_{u,\varepsilon}^{\otimes n}%
,M)\geq1+o\left(  1\right)  ,
\]
so that the first relation in (\ref{sep-rate-def}) is shown.
\end{proof}


\begin{proof}[Proof of Theorem~\ref{thm:Test-2}]
It suffices to prove that if $\varphi_{n}=c_{n}n^{-1/2}$ with $c_{n}%
\rightarrow c>0$ then $\mathbb{P}_{e}^{\ast}\left(  \varphi_{n}\right)
\rightarrow\exp\left(  -c^{2}/4\right)  $. In view of the upper bound
(\ref{upper-bound-test-2}), if suffices to prove
\begin{equation}
\mathbb{P}_{e}^{\ast}\left(  \varphi_{n}\right)  \geq\exp\left(
-c^{2}/4\right)  \left(  1+o\left(  1\right)  \right)  .
\label{suffices-lowbound}%
\end{equation}
Recall (cf. (\ref{trace-dist-pure})) that for any pure states $\rho
=|\psi\rangle\langle\psi|$ and $\rho_{0}=|\psi_{0}\rangle\langle\psi_{0}|$,
the condition \newline $\Vert\rho-\rho_{0}\Vert_{1}\geq\varphi_{n}$ in $H_{1}%
(\varphi_{n})$ is equivalent to a condition for the fidelity $F^{2}(\rho
,\rho_{0})=|\langle\psi|\psi_{0}\rangle|^{2}\leq1-\varphi_{n}^{2}/4$.

Let $\mathcal{H}_{0}\subset\mathcal{H}$ be the orthogonal complement of
$\mathbb{C}|\psi_{0}\rangle$ in $\mathcal{H}$. Consider the vector
\[
\psi_{u}=\sqrt{1-\Vert u\Vert^{2}}\cdot\psi_{0}+u,\quad u\in\mathcal{H}_{0}%
\]
and the corresponding pure state $|\psi_{u}\rangle\langle\psi_{u}|$ defined in
terms of the local vector $u$. We restrict the alternative hypothesis to a
smaller set of states such that $\Vert u\Vert\leq\gamma_{n}$, with $\gamma
_{n} = (\log n)^{-1}$. Since the fidelity is given by $F^{2}(\rho
_{0},|\psi_{u}\rangle\langle\psi_{u}|)=|\langle\psi_{u}|\psi_{0}\rangle
|^{2}=1-\Vert u\Vert^{2}$, the restricted hypothesis is characterised by
\[
1-\gamma_{n}^{2}\leq F^{2}(\rho_{0},|\psi_{u}\rangle\langle\psi_{u}%
|)\leq1-\varphi_{n}^{2}/4,\quad\mathrm{or}\quad\varphi_{n}^{2}/4\leq\Vert
u\Vert^{2}\leq\gamma_{n}^{2}.
\]
and additionally by $\left\Vert \psi_{u}\right\Vert _{\alpha}^{2}\leq L$ where
$\left\Vert \cdot\right\Vert _{\alpha}$ is given by (\ref{alpha-seminorm-def}).

Consider again the linear space $\mathcal{H}_{0,N}$ defined in the proof of
Theorem \ref{thm:Test-2} for a choice $N=N_{n}\sim\log \log n$. Since
$\mathcal{H}_{0,N}\subset\mathcal{H}_{0}$, we can further restrict the local
vector $u$ to $u\in\mathcal{H}_{0,N}$. Note that for $u\in\mathcal{H}_{0,N}$
and $\Vert u\Vert\leq\gamma_{n}$ we have
\begin{align*}
\left\Vert u\right\Vert _{\alpha}^{2}  &  =\sum_{j=0}^{N}\left\vert
\left\langle e_{j}|u\right\rangle \right\vert ^{2}j^{2\alpha}\leq N^{2\alpha
}\Vert u\Vert^{2}\leq N^{2\alpha}\gamma_{n}^{2}\\
&  \sim (\log \log n)^{2\alpha} (\log n)^{-2}=o\left(  1\right)  .
\end{align*}
It follows that
\[
\left\Vert \psi_{u}\right\Vert _{\alpha}\leq\sqrt{1-\left\Vert u\right\Vert
^{2}}\left\Vert \psi_{0}\right\Vert _{\alpha}+\left\Vert u\right\Vert
_{\alpha}\leq L^{1/2}
\]
for sufficiently large $n$, thus $\psi_{u}\in S^{\alpha}\left(  L\right)  $.
We can now write the test problem with restricted alternative as
\[%
\begin{array}
[c]{ll}%
H_{0}: & \rho=\rho_{0}\\
H_{1}^{\prime}(\varphi_{n}): & \rho=|\psi_{u}\rangle\langle\psi_{u}|\text{:
}u\in\mathcal{H}_{0,N}\text{, }\varphi_{n}/2\leq\Vert u\Vert\leq\gamma_{n}.
\end{array}
\]


By the strong approximation proven in Theorem~\ref{th.qlan} we get that the
models
\[
\{|\psi_{u} \rangle\langle\psi_{u}|^{\otimes n}, \, \|u\|\leq\gamma_{n}\}
\quad\text{and}\quad\{ |G(\sqrt{n}u)\rangle\langle G(\sqrt{n} u)|, \,
\|u\|\leq\gamma_{n}\}
\]
are asymptotically equivalent, where $G(\sqrt{n}u)$ is the coherent vector in
the Fock space $\Gamma_{s}(\mathcal{H}_{0})$ pertaining to $\sqrt{n} u$. Note
that this proof is very similar to the previous proofs of lower bounds, with a
major difference: the reduced set of states under the alternative hypothesis
is defined with repect to $\rho_{0}$ given by the null hypothesis $H_{0}$
instead of an arbitrary state previously.

In the asymptotically equivalent Gaussian white noise model, the modified
hypotheses concern Gaussian states which can be written in terms of their
coherent vectors as
\begin{equation*}%
\begin{array}
[c]{ll}%
H_{0}: & |G(0)\rangle\\
H_{1}(\varphi_{n}): & |G(\sqrt{n}u)\rangle\text{: }u\in\mathcal{H}%
_{0,N}\text{, }\varphi_{n}/2\leq\Vert u\Vert\leq\gamma_{n}\text{. }%
\end{array}
\end{equation*}

In order to prove the theorem it is sufficient to prove that
\begin{align}
&  \inf_{M_{n}}\sup_{\varphi_{n}/2\leq\Vert u\Vert\leq\gamma_{n}\text{, }%
u\in\mathcal{H}_{0,N}}R_{n}^{T}(|G(0)\rangle\langle G(0)|,|G(\sqrt{n}%
u)\rangle\langle G(\sqrt{n}u)|,M_{n})\label{gaussian-low-bound}\\
&  \geq\exp\left(  -c^{2}/4\right)  +o\left(  1\right)
\label{gaussian-low-bound-a}%
\end{align}
as $n\rightarrow\infty$.

Note that $\dim$ $\mathcal{H}_{0,N}=N$; let $\{g_{j},\,j=1,\ldots,N\}$ be an
orthogonal basis of $\mathcal{H}_{0,N}$ and let $|u\rangle=\sum_{j=1}^{N}%
u_{j}|g_{j}\rangle$. The quantum Gaussian white noise model $\{|G(\sqrt
{n}u)\rangle,u\in\mathcal{H}_{0,N},\,\Vert u\Vert\leq\gamma_{n}\}$ is then
equivalent to the quantum Gaussian sequence model $\{\otimes_{j=1}^{N}%
|G(\sqrt{n}u_{j})\rangle,\,\Vert u\Vert\leq\gamma_{n}\}$. From now on
$|G(z)\rangle$ denotes the coherent vector in the Fock space $\mathcal{F}%
(\mathbb{C})$ pertaining to $z:=x+iy\in\mathbb{C}$. Recall that such a state
is fully characterized by its Wigner function $W_{G(z)}$, which in the case of
coherent states is the density fuction of a bivariate Gaussian distribution.

We shall bound from below the maximal type 2 error probability in the risk
$R_{n}^{T}(M_n) $ in (\ref{gaussian-low-bound})
\begin{equation}
\sup_{\varphi_{n}/2\leq\Vert u\Vert\leq\gamma_{n}\text{, }u\in\mathcal{H}%
_{0,N}}\mathrm{Tr}\left(  |G(\sqrt{n}u)\rangle\langle G(\sqrt{n}u)|\cdot
M_{n,0}\right)  \label{gaussian-low-bound-2}%
\end{equation}
by an average over $u$, where the average is taken with respect to a prior
distribution defined as follows. Assume that $u_{j}$, $j=1,\ldots,N$ are
independently distributed following a complex centered Gaussian law with
variance $\frac{\sigma^{2}}{2}I_{2}$, where $\sigma^{2}=\frac{c^{2}}{4n}%
\frac{1+\varepsilon}{N}$, for some fixed and arbitrary small $\varepsilon>0$,
and $I_{2}$ is the 2 by 2 identity matrix.

\begin{lemma}
\label{lemma:priortest} Let $\Pi$ be the distribution of independent complex
random variables $u_{j}$, for $j=1,...,N$, each one distributed as
\[
N\left(  0,\frac{{\sigma^{2}}}{2}I_{2}\right)  ,\quad\sigma^{2}=\frac{c^{2}%
}{4n}\frac{1+\varepsilon}{N},
\]
for fixed $\varepsilon>0$ and $N \sim \log \log n$. Then as $n\rightarrow\infty$
\[
\mathbb{P}\left(  \frac{c_{n}^{2}}{4n}\leq\Vert u\Vert^{2}\leq\frac{c_{n}^{2}%
}{4n}(1+\varepsilon)^{2}\right)  \rightarrow1,\quad\text{as }n\rightarrow
\infty,
\]
and in particular if $\gamma_{n}=(\log n)^{-1}$ then $\mathbb{P}\left(  \varphi_{n}/2\leq\Vert u\Vert\leq
\gamma_{n}\right)  \rightarrow1,\quad\text{as }n\rightarrow\infty$.
\end{lemma}

\begin{proof}
We have
\begin{align*}
\mathbb{P}\left(  \Vert u\Vert^{2}<\frac{c_{n}^{2}}{4n}\right)   &
=\mathbb{P}\left(  \sum_{j=1}^{N}(|u_{j}|^{2}-\sigma^{2})<\frac{c_{n}^{2}}%
{4n}-N\frac{c^{2}}{4n}\frac{1+\varepsilon}{N}\right) \\
&  \leq\frac{\mathrm{Var}(\sum_{j=1}^{N}|u_{j}|^{2})}{\left(  c_{n}^{2}%
-c^{2}\left(  1+\varepsilon\right)  \right)  ^{2}/16n^{2}}=\frac{N\sigma^{4}%
}{\left(  c^{2}\varepsilon+o(1)\right)  ^{2}/16n^{2}}\\
&  =\frac{Nc^{4}(1+\varepsilon)^{2}/16n^{2}N^{2}}{\left(  c^{2}\varepsilon
+o(1)\right)  ^{2}/16n^{2}}=\left(  \frac{1+\varepsilon}{\varepsilon
+o(1)}\right)  ^{2}\frac{1}{N}=o(1),
\end{align*}
since $N\sim\log \log n\rightarrow\infty$. Similarly, as $(1+\varepsilon
)^{2}>1+\varepsilon$, one shows that
\[
\mathbb{P}\left(  \Vert u\Vert^{2}>\frac{c_{n}^{2}}{4n}(1+\varepsilon
)^{2}\right)  \rightarrow0,
\]
as $n\rightarrow\infty$ and thus we get
\[
\mathbb{P}\left(  \frac{c_{n}^{2}}{4n}\leq\Vert u\Vert^{2}\leq\frac{c_{n}^{2}%
}{4n}(1+\varepsilon)^{2}\right)  \rightarrow1.
\]
As $\gamma_{n}^{2}= (\log n)^{-2}$ decays slower than $c_{n}^{2}/n$,
and $\varphi_{n}/2=c_{n}n^{-1/2}/2$, we deduce that
$$
\mathbb{P}\left(
\varphi_{n}/2\leq\Vert u\Vert\leq\gamma_{n}\right)  \rightarrow 1
$$
as $n\rightarrow\infty$ which ends the proof of the lemma.
\end{proof}

Let us denote by $\Pi$ the prior distribution introduced in
Lemma~\ref{lemma:priortest}. Let us go back to (\ref{gaussian-low-bound-2})
and bound the expression from below as follows:%
\begin{align*}
&  \sup_{\varphi_{n}/2\leq\Vert u\Vert\leq\gamma_{n}\text{, }u\in
\mathcal{H}_{0,N}}\mathrm{Tr}\left(  |G(\sqrt{n}u)\rangle\langle G(\sqrt
{n}u)|\cdot M_{n,0}\right) \\
&  \geq\int_{\varphi_{n}/2\leq\Vert u\Vert\leq\gamma_{n}}\mathrm{Tr}%
(|G(\sqrt{n}u)\rangle\langle G(\sqrt{n}u)|M_{n,0})\Pi(du)
\end{align*}%
\begin{align*}
&  \geq\int\mathrm{Tr}(|G(\sqrt{n}u)\rangle\langle G(\sqrt{n}u)|M_{n,0}%
)\Pi(du)\\
& -\int_{\left\{  \varphi_{n}/2\leq\Vert u\Vert\leq\gamma_{n}\right\}
^{c}}\mathrm{Tr}(|G(\sqrt{n}u)\rangle\langle G(\sqrt{n}u)|M_{n,0})\Pi(du)\\
&  \geq\int\mathrm{Tr}(|G(\sqrt{n}u)\rangle\langle G(\sqrt{n}u)|M_{n,0}%
)\Pi(du)-\Pi\left(  \left\{  \varphi_{n}/2\leq\Vert u\Vert\leq\gamma
_{n}\right\}  ^{c}\right)  .
\end{align*}
By Lemma~\ref{lemma:priortest}, we get for (\ref{gaussian-low-bound})%
\[
\sup_{\varphi_{n}/2\leq\Vert u\Vert\leq\gamma_{n}\text{, }u\in\mathcal{H}%
_{0,N}}R_{n}^{T}(G(0),G(\sqrt{n}u),M_{n})
\]%
\begin{equation}
\geq\mathrm{Tr}(|G(0)\rangle\langle G(0)|M_{n,1})+\mathrm{Tr}\left(
\int|G(\sqrt{n}u)\rangle\langle G(\sqrt{n}u)|\Pi(du)\cdot M_{n,0}\right)
+o(1). \label{gaussian-low-bound-3}%
\end{equation}
The integral on the right side is a mixed state which can be written as
\[
\Phi:=\int|G(\sqrt{n}u)\rangle\langle G(\sqrt{n}u)|\Pi(du)=\bigotimes
_{j=1}^{N}\int|G(\sqrt{n}u_{j})\rangle\langle G(\sqrt{n}u_{j})|\cdot\Pi
_{j}(du_{j}).
\]
Similarly to the proof of Theorem \ref{thm:Flowbounds} we use equation
\eqref{eq.twirling} to show that each of the Gaussian integrals above produces
a thermal (Gaussian) state
\[
\Phi(r)=\left(  1-r\right)  \sum_{k=0}^{\infty}r^{k}|k\rangle\langle k|,\qquad
r=\frac{n\sigma^{2}}{n\sigma^{2}+1}.
\]
Since $|G(0)\rangle\langle G(0)|=\Phi(0)$, the main terms in
(\ref{gaussian-low-bound-3}) are the sum of error probabilities for testing
two simple hypothesis $H_{0}:\Phi(0)^{\otimes N}$ against $H_{1}%
:\Phi(r)^{\otimes N}$. Moreover, we have two commuting product states under
the two simple hypotheses, which reduces the problem to a classical test
between the $N$-fold products of discrete distributions $H_{0}:\{\mathcal{G}%
(0)\}^{\otimes N}$ and $H_{1}:\{\mathcal{G}(r)\}^{\otimes N}$. Here
$\mathcal{G}(r)$ is the geometric distribution $\left\{  (1-r)r^{k}\right\}
_{k=0}^{\infty}$; in particular s $\mathcal{G}(0)$ is the degenerate
distribution concentrated at $0$. The optimal testing error is given by the
maximum likelihood test which decides $H_{0}$ if and only if all observations
are $0$. The type 1 error is 0 and the type 2 error is%
\begin{align*}
\left(  1-r\right)  ^{N}  &  =(n\sigma^{2}+1)^{-N}\geq\exp(-N\cdot n\sigma
^{2})\\
&  =\exp\left(  -Nn\frac{c^{2}}{4n}\frac{1+\varepsilon}{N}\right)
=\exp\left(  -\frac{c^{2}\left(  1+\varepsilon\right)  }{4}\right)  .
\end{align*}
Since $\varepsilon>0$ was arbitrary, this establishes the lower bound
(\ref{gaussian-low-bound-a}) and thus (\ref{suffices-lowbound}).
\end{proof}


\bibliographystyle{imsart-number}
\bibliography{biblioQAE-m9}

\end{document}